\documentclass[11pt]{amsart}

\usepackage{subfiles}
\usepackage{comment}
\usepackage{float}
\usepackage{marginnote}
\usepackage{tabu}
\usepackage{euscript}
\usepackage[dvipsnames]{xcolor}   		             		
\usepackage{graphicx}			
\usepackage{amssymb}
\usepackage{mathrsfs}
\usepackage{amsthm}
\usepackage{amsmath}
\usepackage{stmaryrd}
\usepackage{tikz}
\usepackage{tikz-cd}
\usepackage{accents}
\usepackage{upgreek}
\usepackage{enumerate}
\usepackage{bm}
\usepackage{mathtools}
\usepackage[all]{xy}
\usepackage{caption}
\usepackage{url}
\usepackage{float}
\usepackage{todonotes} 
\usepackage{colonequals}
\usepackage{bbm}
\usepackage{longtable}
\usepackage[full]{textcomp}
\usepackage[cal=cm]{mathalfa}
\usepackage{xparse}
\usepackage{comment}
\usepackage[noadjust]{cite} 

\usepackage{faktor} 
\usepackage{xfrac} 

\usetikzlibrary{calc}
\usetikzlibrary{fadings}
\usetikzlibrary{decorations.pathmorphing}
\usetikzlibrary{decorations.pathreplacing}
\usepackage{tikz,tikz-cd,tikz-3dplot}
\usepackage{pgfplots}
\usetikzlibrary{arrows,shadows,positioning, calc, decorations.markings, 
hobby,quotes,angles,decorations.pathreplacing,intersections,shapes}
\usepgflibrary{shapes.geometric}
\usetikzlibrary{fillbetween,backgrounds}
\usetikzlibrary{arrows.meta} 

\usepackage[margin=1.05in]{geometry}

\tikzset{
  commutative diagrams/.cd, 
  arrow style=tikz, 
  diagrams={>=stealth}
}
\tikzset{
  arrow/.pic={\path[tips,every arrow/.try,->,>=#1] (0,0) -- +(0,4pt);},
  pics/arrow/.default={triangle 90}
}
\tikzset{->-/.style={decoration={
  markings,
  mark=at position .6 with {\arrow{latex}}},postaction={decorate}}
  }
\tikzset{
  c/.style={every coordinate/.try}
}

\theoremstyle{definition}
\newtheorem{innercustomthm}{Theorem}
\newenvironment{customthm}[1]
  {\renewcommand\theinnercustomthm{#1}\innercustomthm}
  {\endinnercustomthm}

\theoremstyle{definition}
\newtheorem{innercustomcor}{Corollary}
\newenvironment{customcor}[1]
  {\renewcommand\theinnercustomcor{#1}\innercustomcor}
  {\endinnercustomcor}

  \theoremstyle{definition}
\newtheorem{innercustomprop}{Proposition}
\newenvironment{customprop}[1]
  {\renewcommand\theinnercustomprop{#1}\innercustomprop}
  {\endinnercustomprop}
  
\makeatletter
\def\@tocline#1#2#3#4#5#6#7{\relax
  \ifnum #1>\c@tocdepth 
  \else
    \par \addpenalty\@secpenalty\addvspace{#2}%
    \begingroup \hyphenpenalty\@M
    \@ifempty{#4}{%
      \@tempdima\csname r@tocindent\number#1\endcsname\relax
    }{%
      \@tempdima#4\relax
    }%
    \parindent\z@ \leftskip#3\relax \advance\leftskip\@tempdima\relax
    \rightskip\@pnumwidth plus4em \parfillskip-\@pnumwidth
    #5\leavevmode\hskip-\@tempdima
      \ifcase #1
       \or\or \hskip 1em \or \hskip 2em \else \hskip 3em \fi%
      #6\nobreak\relax
    \dotfill\hbox to\@pnumwidth{\@tocpagenum{#7}}\par
    \nobreak
    \endgroup
  \fi}
\makeatother

\newcounter{marginnote}
\DeclareMathAlphabet{\mathpzc}{OT1}{pzc}{m}{it}

\usepackage[backref=page]{hyperref}
\hypersetup{
  colorlinks   = true,          
  urlcolor     = blue!65!black,          
  linkcolor    = blue!65!black,          
  citecolor   = blue             
}

\pgfplotsset{compat=1.18}

\theoremstyle{definition}
\newtheorem{theorem}{Theorem}[section]

\newtheorem{corollary}[theorem]{Corollary}

\newtheorem{lemma}[theorem]{Lemma}
\newtheorem{proposition}[theorem]{Proposition}

\newtheorem*{runningexample*}{Running example}

\newtheorem*{aside*}{Aside}

\newtheorem{definition}[theorem]{Definition}
\newtheorem{example}[theorem]{Example}

\newtheorem*{notation*}{Notation} 
\newtheorem{proposition-definition}[theorem]{Proposition-Definition}
\newtheorem{theorem-definition}[theorem]{Theorem-Definition} 

\DeclareMathOperator{\id}{id}

\newcommand{\bcd}{\begin{center}\begin{tikzcd}}
\newcommand{\ecd}{\end{tikzcd}\end{center}}

\newcommand{\e}{\mathrm{e}}

\newcommand{\calQ}{\mathcal{Q}}

\newcommand{\calM}{\mathcal{M}}

\newcommand{\calP}{\mathcal{P}} 
\newcommand{\calR}{\mathcal{R}} 
\newcommand{\calT}{\mathcal{T}} 
\newcommand{\calH}{\mathcal{H}} 
\newcommand{\calK}{\mathcal{K}} 
\newcommand{\calG}{\mathcal{G}} 
\newcommand{\calW}{\mathcal{W}} 
\newcommand{\calX}{\mathcal{X}} 
\newcommand{\calY}{\mathcal{Y}} 

\newcommand{\raag}{A_\Gamma} 
\newcommand{\sal}{\mathbb{S}_{\Gamma}} 
\newcommand{\uag}{U(A_{\Gamma})} 
\newcommand{\vcduag}{\textsc{vcd}(U(A_{\Gamma}))} 
\newcommand{\link}[1]{\operatorname{lk}({#1})} 
\newcommand{\str}[1]{\operatorname{st}({#1})} 
 
\newcommand{\mx}[1]{\operatorname{max}({#1})}
\newcommand{\salb}[1]{\mathbb{S}_\Gamma^{#1}} 
\newcommand{\spine}{K_{\Gamma}} 

\newcommand{\aut}{\operatorname{Aut}(A_\Gamma)} 
\newcommand{\autraag}[1]{\operatorname{Aut}(A_{#1})} 
\newcommand{\symaut}{\operatorname{\Sigma Aut}(A_\Gamma)} 
\newcommand{\out}{\operatorname{Out}(A_\Gamma)} 
\newcommand{\symout}{\operatorname{\Sigma Out}(A_\Gamma)} 
\newcommand{\ssslash}{\smash{\sslash}} 
\newcommand{\symspine}{K_\Gamma^\Sigma} 
\newcommand{\Kmin}[1]{\smash{K_\Gamma^{\operatorname{min}(#1)}}} 
\newcommand{\vcdsymout}{\textsc{vcd}(\Sigma \operatorname{Out}(A_{\Gamma}))} 
\newcommand{\outw}{\operatorname{Out}_{\mathcal{W}}(A_\Gamma)} 

\begin{document}


\title[Outer space for symmetric automorphisms of RAAGs]{Outer space and finiteness properties for symmetric automorphisms of RAAGs, and generalisations}
\author{Gabriel Corrigan}

\begin{abstract}
    We define the symmetric (outer) automorphism group of a right-angled Artin group and construct for it a (spine of) Outer space. This `symmetric spine' is a contractible cube complex upon which the symmetric outer automorphism group acts properly and cocompactly. One artefact of our technique is a strengthening of the proof of contractibility of the untwisted spine, mimicking the original proof that Culler--Vogtmann Outer space is contractible, which may be of independent interest. We apply our results to derive finiteness properties for certain subgroups of outer automorphisms. In particular, we prove that the subgroup consisting of those outer automorphisms which permute any given finite set of conjugacy classes of a right-angled Artin group is of type \emph{VF}, and we show that the virtual cohomological dimension of the symmetric outer automorphism group is equal to both the dimension of the symmetric spine and the rank of a free abelian subgroup.
\end{abstract}

\maketitle

\tableofcontents

\section{Introduction}\label{sec:intro}

\subsection{Overview of results}\label{subsec:overview of results}

Given a finite simplicial graph~$\Gamma$, define the associated \emph{right-angled Artin group (RAAG)}~$\raag$ by the presentation 
\[\raag \coloneqq \left\langle V(\Gamma) \;\;\vert\;\; [a, b] = 1 \; \text{if} \; \{a, b\} \in E(\Gamma) \right\rangle.\]
Hence, when~$\Gamma$ has no edges, the corresponding right-angled Artin group is simply the free group of rank~$|V(\Gamma)|$, and when~$\Gamma$ is a complete graph,~$\raag \cong \smash{\mathbb{Z}^{|V(\Gamma)|}}$.

We define a \emph{symmetric automorphism} of a right-angled Artin group~$\raag$ to be one which sends each generator to a conjugate of a generator, or a conjugate of the inverse of a generator. These form the \emph{symmetric automorphism group}~$\symaut \leq \autraag{\Gamma}$, and its image under the quotient by inner automorphisms is the \emph{symmetric outer automorphism group}~$\symout \leq \out$.

Conjectured and partially proved by H. Servatius \cite{ServatiusAutRAAGs89}, and confirmed by Laurence \cite{LaurenceAutRAAGs95}, is a well-known generating set of~$\out$. One family of these generators is the so-called \emph{twists}. Any outer automorphism which can be written as a product of generators which are not twists is called \emph{untwisted}, and these form the \emph{untwisted subgroup}~$\uag \leq \out$. Every symmetric outer automorphism is untwisted, so~$\symout \leq \uag$.

This is germane, as Charney--Stambaugh--Vogtmann \cite{CharneyStambaughVogtmann17} have constructed an \emph{untwisted Outer space}~$\mathcal{O}_\Gamma^U$ (denoted~$\Sigma_\Gamma$ in \cite{CharneyStambaughVogtmann17}), which is a contractible complex with a proper~$\uag$-action, and generalises the influential \emph{Culler--Vogtmann Outer space}~$CV_n$ \cite{CullerVogtmann86}. Just like~$CV_n$,~$\mathcal{O}_\Gamma^U$ has a deformation retract called its \emph{spine}~$\spine$, which naturally has the structure of a cube complex and has a proper and cocompact~$\uag$-action. Our first theorem is an analogous result for the symmetric outer automorphism group.

\begin{customthm}{A}[\ref{thm:symspine is a retract of Kmin}]\label{customthm:there exists a symmetric spine}
    There exists a contractible cube complex upon which~$\symout$ acts properly and cocompactly.
\end{customthm}

We call this complex the \emph{symmetric spine} for~$\raag$ and denote it by~$\symspine$. It is a cube complex and is a natural~$\symout$-equivariant deformation retract of a \emph{symmetric Outer space}~$\mathcal{O}_\Gamma^\Sigma$. This is obtained by considering the cubes of~$\symspine$ to be arbitrary rectilinear parallelepipeds, in precisely the same way that~$\spine$ embeds into~$\mathcal{O}_\Gamma^U$ as an equivariant deformation retract.

This is a generalisation of a result of Collins \cite{CollinsSymSpine89}, who provided a subcomplex~$K_n^\Sigma$ of the spine~$K_n$ of~$CV_n$. This allowed Collins to determine the \emph{virtual cohomological dimension} ($\textsc{vcd}$) of the symmetric (outer) automorphism group of the free group~$F_n$. Analogously, our second theorem shows that the~$\textsc{vcd}$ of~$\symout$ matches the dimension of the symmetric spine~$\symspine$.

\begin{customthm}{B}[\ref{thm:vcdsymout = dim(symspine)}]\label{customthm:vcdsymout = dim(symspine)}
    The virtual cohomological dimension~$\vcdsymout$ is equal to the dimension of the symmetric spine~$\symspine$.
\end{customthm}

\subsection{Context and discussion of proof strategy}\label{subsec:discussion of proof strategy}

The proof of \emph{Theorem \ref{customthm:there exists a symmetric spine}} transfers Collins's work on the symmetric spine for free groups \cite{CollinsSymSpine89} to the setting of Charney--Stambaugh--Vogtmann's spine of untwisted Outer space for right-angled Artin groups \cite{CharneyStambaughVogtmann17}. Both these papers rest heavily on Culler--Vogtmann's seminal paper \cite{CullerVogtmann86} in which the original Outer space~$CV_n$ was introduced. We now briefly mention the similarities and differences of these papers with the present document; in \S\ref{sec:discussion of proof} we expand on this discussion more carefully. 

The cube complex structure of~$\spine$ is profitably viewed as a union of stars of \emph{marked Salvetti complexes} (see \S\ref{subsec:RAAGs} and \emph{Definition \ref{def:(symmetric) marked Gamma complex}}). For any graph~$\Gamma$, the \emph{Salvetti complex}~$\sal$ is a cube complex with fundamental group isomorphic to~$\raag$ (generalising the~$n$-petalled rose, which is the Salvetti complex for the free group~$F_n$), while a \emph{marking} on a Salvetti complex can be thought of as a labelling of the generators of~$\pi_1(\sal)$ with the elements of a generating set of~$\raag$. An automorphism~$\varphi$ of~$\raag$ acts on~$\spine$ by changing the marking. The symmetric spine~$\symspine$ is built using \emph{symmetric marked Salvetti complexes} (\emph{Definitions \ref{def:symmetric Gamma-complex} \& \ref{def:(symmetric) marked Gamma complex}}). 

Our strategy for proving the symmetric spine~$\symspine$ is contractible may be summarised as follows:

\begin{enumerate}[(i)]
    \item define a norm~$\lVert \cdot \rVert'$ on marked Salvetti complexes such that every symmetric marked Salvetti complex has norm less than every non-symmetric marked Salvetti complex;
    \item adapt \S6 of \cite{CharneyStambaughVogtmann17}, which proves contractibility of the untwisted spine~$\spine$, to our new norm~$\lVert \cdot \rVert'$. The construction with this norm yields an intermediary complex~$\smash{K_\Gamma^{\text{sym}}}$, which we can deduce is contractible;
    \item adapt \S4 of \cite{CollinsSymSpine89}, which retracts~$K_n$ to~$K_n^\Sigma$, to the theory of~$\Gamma$-complexes, defining a deformation retraction~$\smash{K_\Gamma^{\text{sym}}} \to \symspine$.
\end{enumerate}

The overall architecture of our argument is heavily influenced by Collins's proof of contractibility of the symmetric spine in the free group case, which this work generalises. However, we are forced to make one significant departure from Collins's paper \cite{CollinsSymSpine89}, as we now describe.

A key part of Collins's original argument relies on the original proof, and not just the statement, of contractibility of the spine~$K_n$ of Culler--Vogtmann Outer space given in \cite{CullerVogtmann86}. Along the way, Culler and Vogtmann actually proved contractibility of a host of intermediary subcomplexes of~$K_n$ which are not strictly necessary for the proof that~$K_n$ is contractible. However, Collins observed that one of these by-products,~$K_{\operatorname{min}(W)}$, was an ideal complex from which to define a deformation retraction to the symmetric spine~$K_n^\Sigma$. In our setting, the analogue of~$K_{\operatorname{min}(W)}$ is the complex~$\smash{K_\Gamma^{\text{sym}}}$ as in Step (ii) above. However, the proof of contractibility of the untwisted spine~$\spine$ by Charney--Stambaugh--Vogtmann \cite{CharneyStambaughVogtmann17} is more direct than the analogous proof of Culler--Vogtmann, so we do not have a contractible~$\smash{K_\Gamma^{\text{sym}}}$ for free.

Consequently, aside from the generalisation of Collins's work to the setting of the untwisted spine, our main technical contribution is in defining, and proving contractibility of, intermediary subcomplexes of~$\spine$ analogous to those that appear in Culler--Vogtmann's original proof. One of these intermediary subcomplexes is~$\smash{K_\Gamma^{\text{sym}}}$ as in Step (ii) above. 

This is achieved by the following theorem. Fix a finite simplicial graph~$\Gamma$. For an arbitrary finite set~$\calW$ of conjugacy classes of~$\raag$, we define a norm~$\lVert \cdot \rVert_\calW$ which orders the marked Salvetti complexes lexicographically with respect to a certain ordered abelian group. We define the complex~$\Kmin{\calW}$ to be the union of the stars of all marked Salvetti complexes with minimal first entry in this lexicographical ordering.

\begin{customthm}{C}[\ref{thm:Kmin is contractible}]\label{customthm:Kmin is contractible}
    Let~$\calW$ be an arbitrary finite set of conjugacy classes of~$\raag$. Then~$\Kmin{\calW}$ is contractible.
\end{customthm}

The complex~$\smash{K_\Gamma^{\text{sym}}}$ in Step (ii) above is realised as one of these~$\Kmin{\calW}$.

Aside from the proof of \emph{Theorem \ref{customthm:there exists a symmetric spine}}, we give another application of \emph{Theorem \ref{customthm:Kmin is contractible}}. The proof of (\cite{CullerVogtmann86}, Corollary 6.1.4) generalises seamlessly from free groups to any right-angled Artin group, once \emph{Theorem \ref{customthm:Kmin is contractible}} is established. For any finite set~$\calW$ of conjugacy classes of a right-angled Artin group~$\raag$, let~$\outw$ be the subgroup consisting of those outer automorphisms which permute the elements of~$\calW$. Recall that a group is of \emph{type VF} if it has a finite-index subgroup~$G$ which has a finite~$K(G, 1)$-complex (see \emph{Definition \ref{def:type F}} for more detail).

\begin{customcor}{D}[\ref{cor:outw is type VF}]\label{customcor:outw is of type VF}
    For any finite set~$\calW$ of conjugacy classes of a right-angled Artin group~$\raag$, the group~$\outw$ is type \emph{VF}.
 \end{customcor}

It is worth remarking on the calculation of~$\vcdsymout$. A case readily available for analogy is that of the untwisted subgroup,~$\uag$. Millard--Vogtmann \cite{MillardVogtmann2019} found free abelian subgroups of~$\uag$ of rank~$\rho_\Gamma$, the \emph{principal rank}; hence this is a lower bound for~$\vcduag$. On the other hand, it is a standard theorem that if a discrete group~$G$ acts properly and cocompactly on a proper contractible complex~$X$, then~$\textsc{vcd}(G) \leq \dim(X)$. Therefore~$\dim(\spine)$ is an upper bound for~$\vcduag$. However, Millard--Vogtmann also showed that for some graphs~$\Gamma$, the principal rank $\rho_\Gamma$ is strictly less than~$\dim(\spine)$. Hence at least one of two somewhat surprising things occurs:~$\vcduag$ is not algebraically realised by the (naturally-generated) free abelian subgroups of~$\uag$, or not geometrically realised by the (also natural) untwisted spine. It has also been shown that the gap between~$\vcduag$ and~$\dim(\spine)$ can be arbitrarily large \cite{CorriganRetractionsVCD25}. [At time of writing, we know of no graph~$\Gamma$ for which~$\vcduag \neq \rho_\Gamma$.] However, we show that there are no such intricacies in the case of the symmetric outer automorphism group~$\symout$. Adapting techniques from \cite{MillardVogtmann2019}, we have the following.

\begin{customprop}{E}[\ref{prop:symout contains free abelian subgp of rank MSigma(L)} \& \ref{prop:MSigma(L) = MSigma(V)}]\label{customprop:symout contains free abelian subgroup of rank MSigma(L)}
    Fix a right-angled Artin group~$\raag$. Then
    \begin{enumerate}[(i)]
        \item~$\vcdsymout$ is bounded below by the \emph{symmetric principal rank}~$\rho_\Gamma^\Sigma$, which is the rank of a (naturally-generated) free abelian subgroup;
        \item the dimension of the symmetric spine~$\symspine$ is equal to the principal rank~$\rho_\Gamma$.
    \end{enumerate}
\end{customprop}

Combining these two statements with the bound~$\vcdsymout \leq \dim(\symspine)$ yields \emph{Theorem \ref{customthm:vcdsymout = dim(symspine)}}. In particular,~$\vcdsymout$ is realised both algebraically, by a free abelian subgroup, and geometrically, by the dimension of the relevant `spine' (as is the case for free groups, for example). 

From private correspondence, we believe that independent upcoming work of Peio Ardaiz Gale and Conchita Mart\'inez P\'erez also shows, using different methods, that~$\vcdsymout$ is realised as the rank of a free abelian subgroup.

There exists an algorithm, due to Day--Sale--Wade \cite{DaySaleWadeVCDAlgorithm19}, which computes the \textsc{vcd} of any \emph{relative outer automorphism group (RORG)} of a right-angled Artin group. Examples of RORGs include~$\out$,~$\uag$ (virtually), and the \emph{pure symmetric outer automorphism group}~$P\Sigma\operatorname{Out}(\raag)$, which consists of those symmetric outer automorphisms that send each generator to a conjugate of itself. In the notation of \cite{DaySaleWadeVCDAlgorithm19},~$P\Sigma\operatorname{Out}(\raag)$ is the relative outer automorphism group~$\operatorname{Out}\left(\raag;\;\calH^t\right)$ where~$\calH = \{\langle v\rangle \;\colon\; v \in V(\Gamma)\}$. Since~$P\Sigma\operatorname{Out}(\raag)$ is a finite index subgroup of~$\symout$, we have~$\vcdsymout = \textsc{vcd}\left(P\Sigma\operatorname{Out}(\raag)\right)$. Hence~$\vcdsymout$ (and by \emph{Theorem \ref{customthm:vcdsymout = dim(symspine)}},~$\dim(\symspine)$) can be calculated using Day--Sale--Wade's algorithm. We remark also that this algorithm has been implemented in Python by Yutong Dai; a script can be found on Wade's website \cite{RicWadeWebsite}. Neither that algorithm, nor our work, gives a closed form for~$\vcdsymout$ purely in terms of the properties of~$\Gamma$.

Much study has been made of automorphisms of free products. With groundwork laid by Fouxe-Rabinovitch (\cite{FouxeRabinovitchAutsFreeProductsI40}, \cite{FouxeRabinovitchAutsFreeProductsII41}) and by work (\cite{CollinsZieschangFreeProductsI84}, \cite{McCoolBasisConjugatingFreeGroups86}, \cite{CollinsSymSpine89}) stimulated by Culler--Vogtmann's Outer space~$CV_n$ \cite{CullerVogtmann86} in the context of work on symmetric automorphisms of free groups, McCullough and Miller \cite{McCulloughMillerSymAutsFreeProducts96} built a complex~$K(G)$, somewhat analogous to~$CV_n$, for studying the (outer) symmetric automorphism groups of a free product~$G = G_1 \ast \cdots \ast G_k$. This allowed them to conclude various cohomological finiteness results, among others. A generalisation of this is Guirardel--Levitt's \emph{Outer space for a free product}~$G_1 \ast \cdots \ast G_k \ast F_r$, where each~$G_i$ is freely indecomposable and not infinite cyclic \cite{GuirardelLevittOuterSpaceFreeProduct07}. In another direction, Griffin \cite{GriffinDiagonalComplexes13} introduced a moduli space of combinatorial objects called `cactus products' and uses this to compute the integral homology of~$\Sigma\operatorname{Aut}(G_1 \ast \cdots \ast G_k)$. If a right-angled Artin group can be written as a free product, all of these results will apply; however, note that a right-angled Artin group~$\raag$ splits as a free product if and only if~$\Gamma$ is disconnected. 

\subsection*{Organisation}

In \S\ref{sec:preliminaries ((symmetric) auts of RAAGs and finiteness properties)} we gather the necessary preliminaries regarding right-angled Artin groups and their (symmetric) automorphisms, as well as the necessary technology of~$\Gamma$-Whitehead automorphisms and~$\Gamma$-partitions. In sections \S\ref{sec:symmetric Gamma-complexes} and \S\ref{sec:symmetric spine} we construct the symmetric spine~$\symspine$, prove it is connected, and has a proper and cocompact~$\symout$-action. In \S\ref{sec:discussion of proof} we discuss in detail the strategy of our proof of contractibility of~$\symspine$ and its relation to the existing literature. In \S\ref{sec:contractibility of Kmin} we prove \emph{Theorem \ref{customthm:Kmin is contractible}} and deduce \emph{Corollary \ref{customcor:outw is of type VF}}, and then in \S\ref{sec:retraction of Kmin to symspine} we complete the proof of \emph{Theorem \ref{customthm:there exists a symmetric spine}}; that is, contractibility of the symmetric spine. Finally, in \S\ref{sec:dimension of sym spine}, we prove \emph{Theorem \ref{customthm:vcdsymout = dim(symspine)}}, that the virtual cohomological dimension of the symmetric outer automorphism group is equal to the dimension of the symmetric spine.

\subsection*{Acknowledgements}

I would like to thank my doctoral supervisors, Rachael Boyd and Tara Brendle, for their guidance and support while completing this project. I would also like to thank Karen Vogtmann for enlightening discussions (and essentially planting the seed that became this project by making me aware of Collins's inspiring work \cite{CollinsSymSpine89}). I am very grateful to Ric Wade for many helpful comments regarding the pure symmetric outer automorphism group and virtual duality groups, including pointing out \cite{WiedmerRAAGsCommensurateOutRAAG24} to me. In particular, Ric suggested using \emph{Theorem \ref{customthm:Kmin is contractible}} to obtain the result that became \emph{Corollary \ref{customcor:outw is of type VF}}, for which I am indebted. Finally, I thank all those mentioned here for reading an earlier draft of this paper.

\section{(Symmetric) Automorphisms of RAAGs and finiteness properties}\label{sec:preliminaries ((symmetric) auts of RAAGs and finiteness properties)}

In this section we provide the necessary preliminaries regarding right-angled Artin groups and their (symmetric) automorphisms, introduce basic graph-theoretic terminology and notation which will be employed throughout the paper, and recall some relevant finiteness properties of discrete groups.

\subsection{Right-angled Artin groups}\label{subsec:RAAGs}

\begin{definition}
    Let~$\Gamma$ be a finite simplicial graph (that is, it has no loops and no multi-edges) whose vertices are labelled. Write~$V(\Gamma)$ and~$E(\Gamma)$ for its vertex and edge sets, respectively. The associated \emph{right-angled Artin group (RAAG)}, denoted~$\raag$, is defined by the presentation 
    \[\raag = \langle V(\Gamma) \;\;\vert\;\; [a, b] = 1 \;\text{if}\; \{a, b\} \in E(\Gamma)\rangle.\]
    \noindent$\Gamma$ is the \emph{defining graph} for~$\raag$. We will blur the distinction between a vertex of~$\Gamma$ and the corresponding generator of~$\raag$, and will communicate the data of a RAAG simply by its defining graph.
\end{definition}

Every Artin group has an associated cell complex called the \emph{Salvetti complex} \cite{Salvetti87}. For a RAAG~$\raag$, the Salvetti complex~$\sal$ is a cube complex with a particularly simple description, as follows. There is one 0-cell, to which is attached a directed 1-cell for each generator. Then for each edge~$\{a, b\}$ in~$\Gamma$, we glue in a 2-torus along the corresponding commutator~$aba^{-1}b^{-1}$. Continuing inductively, for each~$k$-clique in~$\Gamma$ (which corresponds to a set of~$k$ mutually commuting generators), we glue in a~$k$-torus whose constituent~$(k-1)$-tori correspond to the subcliques of our~$k$-clique which have~$k-1$ vertices. 

For example, the Salvetti complex of the free group of rank~$n$ is simply the~$n$-petalled rose graph, while the Salvetti complex of~$\mathbb{Z}^n$ is an~$n$-torus. The Salvetti complex of a RAAG associated to a triangle-free graph is the presentation complex of the RAAG. Note that~$\pi_1(\sal) \cong \raag$. Charney and Davis proved \cite{CharneyDavisFiniteKpi1sforArtingroups} that~$\sal$ is a~$K(\raag, 1)$ space, resolving the~$K(\pi, 1)$ conjecture for RAAGs; it follows that RAAGs are torsion-free and biautomatic (as shown by Niblo and Reeves \cite{NibloReeves97biautomaticity}), for instance. 

We refer the reader to Charney's introduction to RAAGs for further details \cite{CharneyRAAGsSurvey}. 

\subsection{Graph-theoretic preliminaries}\label{subsec:graph-theoretic preliminaries}

In this section, we collect some notation and terminology regarding defining graphs of RAAGs which we shall use throughout. For more details and proofs, we refer to \cite{CharneyVogtmannFinitenessProperties09}.

\begin{definition}\label{def:link and star}
    For~$\Gamma$ a simplicial graph, the \emph{link}~$\link{v}$ of a vertex~$v \in V(\Gamma)$ is the full subgraph spanned by the vertices adjacent to~$v$. The \emph{star} of~$v$ is the full subgraph spanned by~$v$ and all vertices adjacent to it; denote it~$\str{v}$.
\end{definition}

We define a relation~$\leq$ on the vertices of~$\Gamma$ by saying that~$v \leq w$ if~$\link{v}$ is contained in~$\str{w}$. We may define an equivalence relation~$\sim$ on~$V(\Gamma)$ by saying that~$v \sim w$ if both~$v \leq w$ and~$w \leq v$. Then~$\leq$ is a partial order on the set of equivalence classes. We will write~$[v]$ for the equivalence class of~$v$. 

A vertex~$v$ is said to be \emph{maximal} if its equivalence class is maximal with respect to this partial order, i.e., for all~$v' \in [v]$, there is no vertex~$w \notin [v]$ with~$v' \leq w$.

One way to have~$v \leq w$ is if~$\link{v} \subseteq \link{w}$; in this case we write~$v \leq_\circ w$. We write~$v <_\circ w$ if there is a strict containment~$\link{v} \subsetneq \link{w}$; in this case we say that~$w$ \emph{dominates}~$v$. 

\begin{definition}\label{def:principal vertex}
    We say that a vertex~$v \in V(\Gamma)$ is \emph{principal} if there is no~$w \in V(\Gamma)$ with~$v <_\circ w$. Otherwise, we say that~$v$ is \emph{non-principal}.
\end{definition}

All maximal vertices are principal. As observed in \cite{MillardVogtmann2019}, not all principal vertices are maximal. For example, in the triangle with leaves at two of its vertices, the third vertex of the triangle is principal but not maximal.

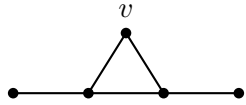
\begin{figure}[H]
    \begin{center}
        \begin{tikzpicture}
            \fill (-1.5, 0) circle (2pt);
            \fill (-0.5, 0) circle (2pt);
            \fill (0.5, 0) circle (2pt);
            \fill (1.5, 0) circle (2pt);
            \fill (0, 0.8) circle (2pt);
            \draw[black, thick] (-1.5, 0) -- (-0.5, 0);
            \draw[black, thick] (-0.5, 0) -- (0.5, 0);
            \draw[black, thick] (0.5, 0) -- (1.5, 0);
            \draw[black, thick] (-0.5, 0) -- (0, 0.8);
            \draw[black, thick] (0.5, 0) -- (0, 0.8);
            \node at (0, 0.8) [above = 2pt, black, thick]{$v$};
        \end{tikzpicture}
        \caption{$v$ is a principal vertex but is not maximal.}
        \label{fig:example of principal but not maximal}
    \end{center}
\end{figure}

\subsection{Culler--Vogtmann Outer space~$CV_n$ and its spine~$K_n$}\label{subsec:Culler-Vogtmann Outer space}

Let~$F_n$ denote the free group of rank~$n$. In \cite{CullerVogtmann86}, Culler and Vogtmann introduced what is now called \emph{Culler--Vogtmann Outer space}, a contractible complex~$CV_n$ which has a proper action of~$\operatorname{Out}(F_n)$.~$CV_n$ has a natural deformation retract~$K_n$ called its \emph{spine}, which is a contractible cube complex with a proper and cocompact action of~$\operatorname{Out}(F_n)$. The spine~$K_n$ is a very good `picture' of~$\operatorname{Out}(F_n)$: for example, Bridson and Vogtmann \cite{BridsonVogtmann01} proved that the group of simplicial automorphisms of~$K_n$ is precisely~$\operatorname{Out}(F_n)$.

For surveys of some of the applications of~$CV_n$, we refer the reader to \cite{VogtmannOuterSpaceSurvey}, \cite{VogtmannAutFnsurvey18}, \cite{VogtmannOuterspaceCIRMnotes}.

Generalising~$CV_n$, Charney and Vogtmann (initially with Crisp \cite{CharneyCrispVogtmann07}, then with Stambaugh \cite{CharneyStambaughVogtmann17}, and finally the full construction with Bregman \cite{BregmanCharneyVogtmann23}) have produced an `Outer space'~$\mathcal{O}_\Gamma$ for any RAAG~$\raag$. In particular, there is an \emph{untwisted Outer space}~$\mathcal{O}^U_\Gamma$, first constructed in \cite{CharneyStambaughVogtmann17}, which is a contractible complex with a proper action of~$\uag$. In direct analogy with~$CV_n$ retracting onto its spine~$K_n$, untwisted Outer space also has a \emph{spine}, denoted~$\spine$. The construction of~$\spine$ shall be closely mirrored in \S\ref{subsec:symmetric marked Gamma-complexes and the symmetric spine}, when we build the symmetric spine~$\symspine$.

\subsection{Finiteness properties of groups}\label{subsec:finiteness properties}

We briefly recall the definition of \emph{virtual cohomological dimension} and that of a group being of \emph{type F} or \emph{type VF}; the reader is referred to \cite{BrownCohomologyofGroups} and \cite{GeogheganTopologicalGroupsBook08} for further details.

\subsubsection{Virtual cohomological dimension}\label{subsubsec:vcd}

Cohomological dimension may be defined for any unital nonzero commutative ring~$R$, but we shall only be interested in the case~$R = \mathbb{Z}$.

\begin{definition}[cf. ~\cite{BrownCohomologyofGroups}, \emph{Lemma VIII.2.1}]
    Let~$G$ be a discrete group, and consider~$\mathbb{Z}$ as the trivial~$\mathbb{Z}G$-module, where~$\mathbb{Z}G$ is the group ring. Define~$\textsc{cd}(G) = \textsc{cd}_{\mathbb{Z}}(G)$, the \emph{cohomological dimension} of~$G$, to be the least integer~$n$ such that~$\mathbb{Z}$ admits a projective resolution of~$\mathbb{Z}G$-modules of length~$n$, if such~$n$ exists, or to be infinite otherwise.
\end{definition}

Equivalently (and hence the nomenclature),~$\textsc{cd}(G)$ is the supremum of the set of~$n$ such that the cohomology group~$H^n(G; M)$ does not vanish, for some~$\mathbb{Z}G$-module~$M$.

Observe that the cohomological dimension of a group with torsion is always infinite, as all the even cohomology groups with coefficients in~$\mathbb{Z}G$ are nonzero. However, we have the following result of Serre \cite{Serre71groupcohomology}, which can be found in \textit{Ch.~VIII.3} of \cite{BrownCohomologyofGroups}.

\begin{theorem}[Serre]
    Let~$G$ be virtually torsion-free. Then all finite-index subgroups of~$G$ have the same cohomological dimension.
\end{theorem}

Hence for a group~$G$ which contains a finite-index torsion-free subgroup~$H$, the \emph{virtual cohomological dimension} of~$G$, defined to be~$\textsc{vcd}(G) \coloneqq \textsc{cd}(H)$, is well-defined.

Charney and Vogtmann have shown \cite{CharneyVogtmannFinitenessProperties09} that for any finite simplicial graph~$\Gamma$,~$\out$ is virtually torsion-free and has finite virtual cohomological dimension. Hence~$\textsc{vcd}(\symout)$ is also finite.

We have the following geometric control over the virtual cohomological dimension.

\begin{theorem}[\cite{BrownCohomologyofGroups}, \emph{Theorem VIII.11.1}]\label{thm:vcd upper bound realisation theorem}
    Let~$G$ act properly and cocompactly on a proper contractible CW-complex~$X$. Then~$\textsc{vcd}(G) \leq \dim(X)$.
\end{theorem}

\subsubsection{Type \emph{F} and type \emph{VF}}\label{subsubsec:type F and VF}

\begin{definition}\label{def:type F}
    A group~$G$ is of \emph{type F} if there exists a finite aspherical CW-complex with fundamental group isomorphic to~$G$. We say that a group is of \emph{type VF} if it has a finite-index subgroup which is type \emph{F}.
\end{definition}

Being of type \emph{F} is a very strong finiteness condition; in particular, it implies finite generation and finite presentability. Groups with torsion cannot be type \emph{F}, so the strongest we can ask of them is to be of type \emph{VF} (that is, `virtually type \emph{F}'). A group of type \emph{(V)F} always satisfies the homological finiteness properties of being \emph{types (V)FL} and \emph{(V)FP}; in particular a group of type \emph{VF} has finite virtual cohomological dimension.

The following appears as Corollary 5.12 in \cite{DayWadeRORGs19}; it follows from \emph{Theorem 7.3.4} of \cite{GeogheganTopologicalGroupsBook08}. 

\begin{lemma}\label{lem:sufficient condition for type F}
    If a group~$G$ acts simplicially and cocompactly on a contractible simplicial complex~$X$ such that all stabilisers are of type \emph{F}, then~$G$ is of type~$F$. 
\end{lemma}

In particular, if~$G$ acts freely, simplicially, and cocompactly on a simplicial complex, then~$G$ is of type \emph{F}.

In particular, the result that the untwisted subgroup~$\uag$ (see \S\ref{subsec:Gamma-Whitehead partitions and automorphisms}) acts properly and cocompactly on the untwisted spine~$\spine$ \cite{CharneyStambaughVogtmann17} proves that~$\uag$ is of type \emph{VF} (and so~$\out$ is too, when~$\uag$ is finite-index in~$\out$). More recently, Day--Wade \cite{DayWadeRORGs19} proved that~$\out$ is type \emph{VF} for every defining graph~$\Gamma$.

\subsection{Symmetric automorphisms}\label{subsec:symmetric automorphisms}

\begin{definition}\label{def:symmetric automorphisms}
    We say that an automorphism of a RAAG~$\raag$ is \emph{symmetric} if it sends every generator to a conjugate of a generator, or the conjugate of an inverse of a generator.
\end{definition}

Observe that the composition of two symmetric automorphisms is again symmetric. We denote the \emph{symmetric automorphism group} of~$\raag$ by~$\symaut$; this is the subgroup of~$\aut$ consisting of all symmetric automorphisms.

We say that an outer automorphism~$[\varphi] \in \out$ is \emph{symmetric} if it has at least one symmetric representative~$\varphi \in \aut$. The composition of two symmetric outer automorphisms will therefore have a symmetric representative, so we may form the \emph{symmetric outer automorphism group},~$\symout$; this is the image of~$\symaut$ under the quotient~$\aut \twoheadrightarrow \out$.

\subsection{$\Gamma$-Whitehead partitions and automorphisms}\label{subsec:Gamma-Whitehead partitions and automorphisms}

Conjecture and partially proved by H. Servatius \cite{ServatiusAutRAAGs89} and completely proved by Laurence \cite{LaurenceAutRAAGs95},~$\aut$ has the following set of generators (and hence~$\out$ is generated by their images).

\begin{theorem-definition}[\cite{LaurenceAutRAAGs95}, Theorem 6.9]\label{thmdef:generators for Aut(RAAG)}
    For any~$\Gamma$,~$\operatorname{Aut}(\raag)$ is generated by the following set of \emph{elementary transformations}.
    \begin{enumerate}
        \item \emph{graph symmetries}: an automorphism of~$\Gamma$ permutes the generators of~$\raag$ and thus induces an element of~$\aut$.
        \item \emph{Inversions}: a map which sends a generator~$a \mapsto a^{-1}$ and fixes all other generators.
        \item \emph{Transvections}: for two vertices~$v, w$ of~$\Gamma$ with~$v \leq w$, the map which sends~$v \mapsto vw$ or~$wv$ and fixes all other generators is an automorphism. These are split into two types:
        \begin{itemize}
            \item \emph{twists}, if~$v$ and~$w$ are adjacent (in which case~$vw = wv$);
            \item \emph{folds} if~$v, w$ are not adjacent. In particular,~$v \mapsto vw$ is a \emph{right fold}, while~$v \mapsto wv$ is a \emph{left fold}.
        \end{itemize}
        \item \emph{Partial conjugations}: for a vertex~$v$, let~$C$ be a connected component of~$\Gamma \setminus \str{v}$. The map which sends every generator~$c \in C$ to the conjugation~$vcv^{-1}$ and fixes all other generators is called a \emph{partial conjugation} (Laurence and Servatius refer to these as \emph{locally inner} automorphisms).  
    \end{enumerate}
\end{theorem-definition}

Folds and partial conjugations can be realised as products of \emph{$\Gamma$-Whitehead automorphisms}, which are defined using so-called \emph{$\Gamma$-Whitehead partitions}. 

Let~$V = V(\Gamma)$ and write~$V^\pm = V \cup V^{-1}$. Define the \emph{double} of~$\Gamma$ to be the graph~$\Gamma^\pm$ which has vertex set~$V^\pm$ and where two vertices are joined by an edge if they commute in~$\raag$ but are not inverses. A \emph{$\Gamma$-Whitehead partition}~$\calP$ based at a vertex~$m \in V$ is a partition of~$V^\pm$ into three sets: two \emph{sides}~$P$,~$\overline{P}$, and the \emph{link}~$\link{\calP}$. The link~$\link{\calP}$ contains all vertices adjacent to~$m$ in~$\Gamma^\pm$. We place~$m$ and~$m^{-1}$ in different sides. The connected components of~$\Gamma^\pm \setminus \str{m}^\pm$ are then distributed between the sides~$P$ and~$\overline{P}$, subject to the condition that~$P$ and~$\overline{P}$ must each contain at least two vertices. We call the pair~$\left(\calP, m\right)$ a~$\Gamma$-Whitehead pair, and write~$\calP = \left(P \;\vert\;\overline{P}\;\vert\;\link{\calP}\right)$. We will often refer to~$\Gamma$-Whitehead partitions as \emph{$\Gamma$-partitions} (or even just \emph{partitions} when~$\Gamma$ is clear). Examples of~$\Gamma$-Whitehead partitions are given in \emph{Figure \ref{fig:compatibility is not transitive}}.

If a vertex~$v$ and its inverse lie on different sides of a~$\Gamma$-Whitehead partition~$\calP$, then we say that~$\calP$ \emph{splits}~$v$. Note that if~$\calP$ splits~$v \neq m$, then~$v$ and~$v^{-1}$ must lie in different connected components of~$\Gamma^\pm \setminus \str{m}^\pm$, which means that~$\link{v} \subseteq \link{m}$.

Each~$\Gamma$-Whitehead pair~$\left(\calP, m\right)$ determines an automorphism~$\varphi\left(\calP, m\right)$ of~$\raag$ as follows. 
\begin{itemize}
    \item If~$\calP$ splits~$v \neq m$, then:
    \begin{itemize}
        \item if~$v \in P$, then~$\varphi\left(\calP, m\right) \colon v \mapsto vm^{-1}$;
        \item if~$v \in \overline{P}$, then~$\varphi\left(\calP, m\right) \colon v \mapsto mv$.
    \end{itemize}
    \item If~$\{v, v^{-1}\} \subseteq P$, then~$\varphi\left(\calP, m\right) \colon v \mapsto mvm^{-1}$.
    \item For all other~$v$ (including~$v = m$),~$\varphi\left(\calP, m\right)(v) = v$.
\end{itemize}

We call~$m$ the \emph{multiplier} of the~$\Gamma$-Whitehead automorphism~$\varphi(\calP, m)$. 

One may realise a fold~$v \mapsto vm^{-1}$ or~$v \mapsto mv$ by taking a~$\Gamma$-Whitehead partition based at~$m$ with positive side~$\{m, v\}$ or~$\{m, v^{-1}\}$ respectively. Similarly, a partial conjugation of the connected component~$C \subseteq \Gamma^\pm \setminus \str{m}^\pm$ by~$m$ is realised by taking the positive side~$\{m, C\}$. Every~$\Gamma$-Whitehead automorphism is a product of folds and partial conjugations.

The group generated by~$\Gamma$-Whitehead automorphisms along with inversions is studied in \cite{BregmanCharneyVogtmannFiniteSubgroupsI24}, where it is denoted~$U^0(\raag)$. Adding in automorphisms induced by graph symmetries generates~$\uag$, the \emph{untwisted subgroup}. Further adding twists to this generates the full outer automorphism group~$\out$.

\subsection{The pure symmetric automorphism group}\label{subsec:pure symmetric automorphism group}

The \emph{pure symmetric automorphism group} of~$\raag$, denoted~$P\Sigma\operatorname{Aut}(\raag)$, is the finite-index subgroup of~$\symaut$ which requires all generators to be sent to a conjugate of themselves---that is, it excludes graph symmetries and inversions. The image after taking the quotient by inner automorphisms is the \emph{pure symmetric outer automorphism group}~$P\Sigma\operatorname{Out}(\raag)$. Also known as the \emph{basis-conjugating (outer) automorphism group}, it has been studied by Toinet \cite{ToinetPresBasisConjugacyRAAGs11}, Koban--Piggott \cite{KobanPiggottBNSSymmetricAutosRAAGs13}, and Day--Wade \cite{DayWadeSubspaceBNSInvPSOutRAAG18}; in particular, a presentation is given by the first two of these (while a generating set was given by Laurence \cite{LaurenceAutRAAGs95}).

As discussed in \S\ref{subsec:discussion of proof strategy}, the pure symmetric outer automorphism group is an example of a \emph{relative outer automorphism group (RORG)}. The class of RORGs was introduced in full generality by Day and Wade \cite{DayWadeRORGs19}, although it encompassed many previously-studied examples. We refer the reader to (\cite{DayWadeRORGs19}, \S6) for a more complete survey of the literature surrounding RORGs.

In \cite{DayWadeSubspaceBNSInvPSOutRAAG18}, Day--Wade gave necessary and sufficient conditions for when~$P\Sigma\operatorname{Out}(\raag)$ is itself a RAAG, which gives a way to find examples of RAAGs whose outer automorphism groups are not \emph{virtual duality groups} (see also \cite{WadeBrueckVirtualDualityRAAG23} and \cite{WiedmerRAAGsCommensurateOutRAAG24}). This is in contrast to both ends of the RAAGs spectrum:~$\operatorname{Out}(F_n)$ and~$\operatorname{GL}_n(\mathbb{Z})$ are both virtual duality groups.

Since~$P\Sigma\operatorname{Out}(\raag)$ is a finite index subgroup of~$\symout$, these two groups have the same virtual cohomological dimension. 

\section{Symmetric~$\Gamma$-complexes}\label{sec:symmetric Gamma-complexes}

\subsection{Symmetric~$\Gamma$-partitions}\label{subsec:symmetric Gamma-partitions}

\begin{definition}\label{def:symmetric Gamma-partitions}
    We say that a~$\Gamma$-partition is \emph{symmetric} if it splits no vertices (other than its base).
\end{definition}

\begin{proposition}\label{prop:generating set for symmetric automorphism group}
   ~$\symout$ is generated by inversions, graph symmetries, and those~$\Gamma$-Whitehead automorphisms which correspond to symmetric~$\Gamma$-partitions.
\end{proposition}

\begin{proof}
    It is clear that all inversions and graph symmetries are symmetric automorphisms. 

    Let~$(\calP, m)$ be a~$\Gamma$-Whitehead pair with~$\calP$ symmetric. By definition, a symmetric~$\Gamma$-partitions splits no vertices other than its base, so for every~$v \neq m$ with~$v \in P$, we have~$v^{-1} \in P$. Hence the corresponding~$\Gamma$-Whitehead automorphism~$\varphi\left(\calP, m\right)$ sends~$v \mapsto mvm^{-1}$ for each~$v \in P$, and fixes all other generators. Therefore~$\varphi\left(\calP, m\right)$ is also a symmetric automorphism. 

    Hence the subgroup~$H \leq \out$ generated by the (images of) inversions, graph symmetries, and~$\Gamma$-Whitehead automorphisms defined by symmetric~$\Gamma$-partitions consists entirely of symmetric automorphisms, so~$H \leq \symout$.

    Let~$[\psi] \in \symout$, and let~$\psi \in \symaut$ be a symmetric representative for~$[\psi]$. By composing~$\psi$ with a series of inversions, we obtain an automorphism~$\psi'$ which does not send any element of~$V$ to an element of~$V^{-1}$. 

    Writing~$V = \{v_1, \dots, v_n\}$, we may therefore write~$\psi'(v_i) = g_i v_{\pi(i)} g_i^{-1}$ for each~$i = 1, \dots, n$, where~$\pi$ is some element of the symmetric group~$S_n$.
    
    Note that if~$v$ and~$w$ are adjacent vertices in~$\Gamma$, then~$\psi'(v)$ and~$\psi'(w)$ must be (conjugates of) adjacent vertices of~$\Gamma$. In other words, the permutation of the vertices of~$\Gamma$ induced by~$\pi$ must respect the adjacency relations in~$\Gamma$---that is, it is realisable as a graph symmetry. Composing~$\psi'~$with the inverse of the automorphism of~$\raag$ induced by this graph symmetry yields an automorphism~$\psi''$ which sends each generator~$v \in V$ to a conjugate of itself. 

    Now~$\psi''$ is an element of the pure symmetric outer automorphism group~$P\Sigma\operatorname{Out}(\raag)$. By work of Laurence \cite{LaurenceAutRAAGs95} (see also Toinet \cite{ToinetPresBasisConjugacyRAAGs11} and Koban--Piggott \cite{KobanPiggottBNSSymmetricAutosRAAGs13}), this group has a standard generating set consisting of partial conjugations~$\varphi_C^m$, defined as follows:
    \begin{equation*}
        \varphi_C^m(v) = 
        \begin{cases}
            mvm^{-1} \quad \quad \;\text{if~$v \in C$;} \\
            v \quad \quad \quad \quad \quad \text{otherwise.}
        \end{cases}
    \end{equation*}
    where~$C$ is a connected component of~$\Gamma \setminus \str{m}$. Now,~$\varphi_C^m$ is precisely the~$\Gamma$-Whitehead automorphism defined by the~$\Gamma$-partition based at~$m$ which has positive side~$\{m, C, C^{-1}\}$. This is a symmetric~$\Gamma$-partition. Hence~$\psi$ is a product of inversions, graph symmetries, and~$\Gamma$-Whitehead automorphisms defined by symmetric~$\Gamma$-partitions, so~$[\psi] \in H$, whence we have~$H = \symout$ as required.  
\end{proof}

\subsection{$\Gamma$-complexes: a recap}\label{subsec:recap on Gammma-complexes}

We recall the construction of \emph{$\Gamma$-complexes} from \cite{CharneyStambaughVogtmann17}. This introductory exposition, until the end of \S\ref{subsec:Gamma-complexes}, closely follows \S3 of \cite{BregmanCharneyVogtmannFiniteSubgroupsI24}.

\subsubsection{Special cube complexes and hyperplane collapses}\label{subsubsec:sccs and hyperplane collapses}

$\Gamma$-complexes are certain \emph{special} cube complexes, with some extra structure. Recall that a \emph{special cube complex} is one which is locally~$\operatorname{CAT}(0)$ and has no hyperplanes which self-intersect or are \emph{one-sided}, \emph{self-osculating}, or \emph{inter-osculating}; we refer to the original article of Haglund and Wise for the definitions \cite{HaglundWiseSCCs08}.

For any hyperplane~$H$ of a special cube complex~$X$, we define the edges \emph{dual} to~$H$ to be those edges of~$X$ whose interior has non-empty intersection with~$H$. We can define an \emph{orientation} on~$H$ by giving a consistent choice of orientation to the edges dual to~$H$.

In a special cube complex~$X$, one may define the \emph{hyperplane collapse}~$X \ssslash H$ for any hyperplane~$H$. This is a map~$c_H \colon X \to X \ssslash H$ which collapses~$\kappa(H)$, the carrier of~$H$, orthogonally onto~$H$. For any collection~$\calH$ of hyperplanes, we write~$X \ssslash \calH$ for the cube complex obtained by collapsing all the hyperplanes of~$\calH$ (in any order). We will call~$\calH$ \emph{acyclic} if the collapse map~$c_\calH \colon X \to X \ssslash \calH$ is a homotopy equivalence. [The notation~$X\ssslash \mathcal{H}$ is not intended to evoke other notions using the double slash, such as homotopy quotients or GIT quotients; we take our lead from e.g. \cite{BregmanCharneyVogtmannFiniteSubgroupsI24}.]

\subsubsection{Blowups of the Salvetti complex}\label{subsubsec:blowups}

Let~$\calP = \left( P \;\vert\; \overline{P} \;\vert\; \link{\calP} \right)$ be a~$\Gamma$-partition based at a vertex~$m$, so~$\link{\calP} = \link{m}^\pm$. If~$\calP$ splits another vertex~$n$ with~$\link{m} = \link{n}$, then~$n$ may also serve as a base for~$\calP$. Changing the base changes the corresponding~$\Gamma$-Whitehead automorphism, but does not affect the combinatorial data of the~$\Gamma$-partition~$\calP$.

We say that~$\Gamma$-partitions~$\calP$ and~$\calQ$ are \emph{adjacent} if some (hence any) base of~$\calP$ commutes with some (hence any) base of~$\calQ$. We say that~$\calP$ and~$\calQ$ are \emph{compatible} if they are adjacent, or if there is some side of~$\calP$ which is disjoint from some side of~$\calQ$. We say that a set~$\Pi$ of~$\Gamma$-partitions is \emph{compatible} if its elements are pairwise compatible. We remark that compatibility is not transitive, as demonstrated by the example shown in \emph{Figure \ref{fig:compatibility is not transitive}}.

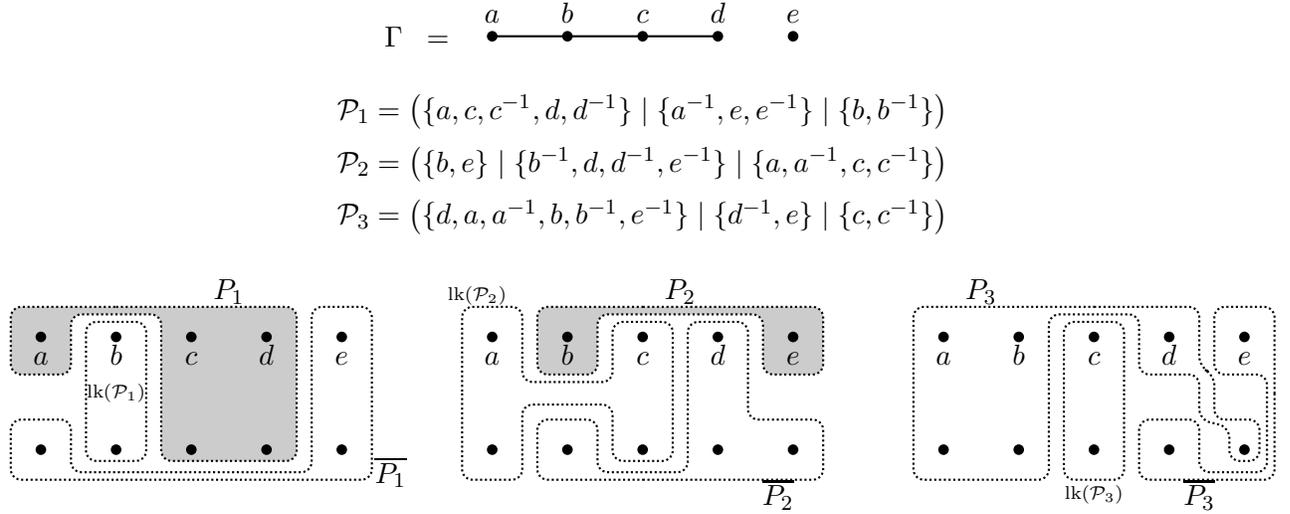
\begin{figure}
    \begin{center}
        \begin{tikzpicture}
            %
            %
            \fill (-2, 0.5) circle (2pt);
            \fill (-1, 0.5) circle (2pt);
            \fill (0, 0.5) circle (2pt);
            \fill (1, 0.5) circle (2pt);
            \fill (2, 0.5) circle (2pt);
            \draw[black, thick] (-2, 0.5) -- (-1, 0.5);
            \draw[black, thick] (-1, 0.5) -- (0, 0.5);
            \draw[black, thick] (0, 0.5) -- (1, 0.5);
            \node at (-2, 0.5) [above = 1pt, black, thick]{$a$};
            \node at (-1, 0.5) [above = 1pt, black, thick]{$b$};
            \node at (0, 0.5) [above = 1pt, black, thick]{$c$};
            \node at (1, 0.5) [above = 1pt, black, thick]{$d$};
            \node at (2, 0.5) [above = 1pt, black, thick]{$e$};
            \node at (-3, 0.5) [black, thick]{$\Gamma \;\; =$};
            %
            %
            \node at (0, -0.5) [black, thick]{$\calP_1 = \left( \{a, c, c^{-1}, d, d^{-1}\} \;|\; \{a^{-1}, e, e^{-1}\} \;|\; \{b, b^{-1}\} \right)$};
            \node at (0, -1.2) [black, thick]{$\calP_2 = \left( \{b, e\} \;|\; \{b^{-1}, d, d^{-1}, e^{-1}\} \;|\; \{a, a^{-1}, c, c^{-1}\} \right)$};
            \node at (0, -1.9) [black, thick]{$\calP_3 = \left( \{d, a, a^{-1}, b, b^{-1}, e^{-1}\} \;|\; \{d^{-1}, e\} \;|\; \{c, c^{-1}\} \right)$};
            %
            %
            \draw [black, thick, rounded corners, densely dotted, fill = black!20!white] (-7, -3.1) to (-4.6, -3.1) to (-4.6, -5.15) to (-6.4, -5.15) to (-6.4, -3.2) to (-7.6, -3.2) to (-7.6, -4) to (-8.4, -4) to (-8.4, -3.1) to (-7, -3.1);
            \draw [black, thick, rounded corners, densely dotted] (-7, -5.4) to (-3.6, -5.4) to (-3.6, -3.1) to (-4.4, -3.1) to (-4.4, -5.3) to (-7.6, -5.3) to (-7.6, -4.6) to (-8.4, -4.6) to (-8.4, -5.4) to (-7, -5.4);
            \draw [black, thick, rounded corners, densely dotted] (-7.4, -4.25) to (-7.4, -5.15) to (-6.6, -5.15) to (-6.6, -3.3) to (-7.4, -3.3) to (-7.4, -4.25);
            \fill (-8, -3.5) circle (2pt);
            \fill (-7, -3.5) circle (2pt);
            \fill (-6, -3.5) circle (2pt);
            \fill (-5, -3.5) circle (2pt);
            \fill (-4, -3.5) circle (2pt);
            \fill (-8, -5) circle (2pt);
            \fill (-7, -5) circle (2pt);
            \fill (-6, -5) circle (2pt);
            \fill (-5, -5) circle (2pt);
            \fill (-4, -5) circle (2pt);
            \node at (-8, -4.05) [above = 1pt, black, thick]{$a$};
            \node at (-7, -4.05) [above = 1pt, black, thick]{$b$};
            \node at (-6, -4.05) [above = 1pt, black, thick]{$c$};
            \node at (-5, -4.05) [above = 1pt, black, thick]{$d$};
            \node at (-4, -4.05) [above = 1pt, black, thick]{$e$};
            \node at (-7, -4.25) [black, thick]{\tiny$\link{\calP_1}$};
            \node at (-5.5, -2.9) [black, thick]{$P_1$};
            \node at (-3.35, -5.3) [black, thick]{$\overline{P_1}$};
            %
            %
            \draw [black, thick, rounded corners, densely dotted, fill = black!20!white] (0, -3.1) to (2.4, -3.1) to (2.4, -4) to (1.6, -4) to (1.6, -3.2) to (-0.6, -3.2) to (-0.6, -4) to (-1.4, -4) to (-1.4, -3.1) to (0, -3.1);
            \draw [black, thick, rounded corners, densely dotted] (0, -5.4) to (2.4, -5.4) to (2.4, -4.6) to (1.4, -4.6) to (1.4, -3.3) to (0.6, -3.3) to (0.6, -5.3) to (-0.6, -5.3) to (-0.6, -4.6) to (-1.4, -4.6) to (-1.4, -5.4) to (0, -5.4);
            \draw [black, thick, rounded corners, densely dotted] (-2.4, -4.25) to (-2.4, -5.4) to (-1.6, -5.4) to (-1.6, -4.4) to (-0.4, -4.4) to (-0.4, -5.15) to (0.4, -5.15) to (0.4, -3.3) to (-0.4, -3.3) to (-0.4, -4.1) to (-1.6, -4.1) to (-1.6, -3.1) to (-2.4, -3.1) to (-2.4, -4.25);
            \fill (-2, -3.5) circle (2pt);
            \fill (-1, -3.5) circle (2pt);
            \fill (0, -3.5) circle (2pt);
            \fill (1, -3.5) circle (2pt);
            \fill (2, -3.5) circle (2pt);
            \fill (-2, -5) circle (2pt);
            \fill (-1, -5) circle (2pt);
            \fill (0, -5) circle (2pt);
            \fill (1, -5) circle (2pt);
            \fill (2, -5) circle (2pt);
            \node at (-2, -4.05) [above = 1pt, black, thick]{$a$};
            \node at (-1, -4.05) [above = 1pt, black, thick]{$b$};
            \node at (0, -4.05) [above = 1pt, black, thick]{$c$};
            \node at (1, -4.05) [above = 1pt, black, thick]{$d$};
            \node at (2, -4.05) [above = 1pt, black, thick]{$e$};
            \node at (0.5, -2.9) [black, thick]{$P_2$};
            \node at (1.8, -5.6) [black, thick]{$\overline{P_2}$};
            \node at (-2.2, -2.95) [black, thick]{\tiny$\link{\calP_2}$};
            %
            %
            \draw [black, thick, rounded corners, densely dotted] (6, -3.1) to (7.4, -3.1) to (7.4, -3.9) to (7.6, -4) to (7.6, -4.6) to (8.2, -4.6) to (8.2, -5.15) to (7.8, -5.15) to (7.8, -4.7) to (7.4, -4.6) to (7.4, -4) to (6.6, -4) to (6.6, -3.2) to (5.4, -3.2) to (5.4, -5.4) to (3.6, -5.4) to (3.6, -3.1) to (6, -3.1);
            \draw [black, thick, rounded corners, densely dotted] (8.4, -4.25) to (8.4, -5.4) to (6.6, -5.4) to (6.6, -4.6) to (7.4, -4.6) to (7.4, -5.3) to (8.3, -5.3) to (8.3, -4) to (7.6, -4) to (7.6, -3.1) to (8.4, -3.1) to (8.4, -4.25);
            \draw [black, thick, rounded corners, densely dotted] (6.4, -4.25) to (6.4, -5.4) to (5.6, -5.4) to (5.6, -3.3) to (6.4, -3.3) to (6.4, -4.25);
            \fill (4, -3.5) circle (2pt);
            \fill (5, -3.5) circle (2pt);
            \fill (6, -3.5) circle (2pt);
            \fill (7, -3.5) circle (2pt);
            \fill (8, -3.5) circle (2pt);
            \fill (4, -5) circle (2pt);
            \fill (5, -5) circle (2pt);
            \fill (6, -5) circle (2pt);
            \fill (7, -5) circle (2pt);
            \fill (8, -5) circle (2pt);
            \node at (4, -4.05) [above = 1pt, black, thick]{$a$};
            \node at (5, -4.05) [above = 1pt, black, thick]{$b$};
            \node at (6, -4.05) [above = 1pt, black, thick]{$c$};
            \node at (7, -4.05) [above = 1pt, black, thick]{$d$};
            \node at (8, -4.05) [above = 1pt, black, thick]{$e$};
            \node at (4.5, -2.9) [black, thick]{$P_3$};
            \node at (7.4, -5.6) [black, thick]{$\overline{P_3}$};
            \node at (6, -5.6) [black, thick]{\tiny$\link{\calP_3}$};
        \end{tikzpicture}
        \caption{An example of a graph~$\Gamma$ and three~$\Gamma$-partitions,~$\calP_1$ (based at~$a$),~$\calP_2$ (based at~$b$), and~$\calP_3$ (based at~$d$).~$\calP_1$ is compatible with~$\calP_2$ since their respective positive sides (shaded) have empty intersection (in fact,~$\calP_1$ and~$\calP_2$ are also adjacent). On the other hand,~$\calP_3$ is compatible with neither~$\calP_1$ nor~$\calP_2$. Since~$d$ does not commute with~$a$ or~$b$,~$\calP_3$ is not adjacent to~$\calP_1$ or~$\calP_2$, and one can verify that each side of~$\calP_3$ has non-empty intersection with each of~$P_1$,~$\overline{P_1}$,~$P_2$,~$\overline{P_2}$.~$\calP_1$ is a symmetric~$\Gamma$-partition, while~$\calP_2$ and~$\calP_3$ are not.}
        \label{fig:compatibility is not transitive}
    \end{center}
\end{figure}

Given a compatible set~$\Pi$ of~$\Gamma$-partitions, Charney--Stambaugh--Vogtmann \cite{CharneyStambaughVogtmann17} construct a \emph{blowup}~$\salb{\Pi}$ of the Salvetti complex~$\sal$. This is a special cube complex with no separating hyperplanes, and has some extra structure in the form of labellings and orientations, as follows.

$\salb{\Pi}$ has one hyperplane~$H_\calP$ for each~$\Gamma$-partition~$\calP \in \Pi$, and one hyperplane~$H_v$ for each vertex~$v \in V$. Moreover, the hyperplanes~$H_v$ corresponding to vertices have an orientation. The set~$\calH_\Pi = \{H_\calP \colon \calP \in \Pi\}$ is acyclic, and the cube complex obtained by the corresponding collapse is isomorphic to the Salvetti complex~$\sal$.

\begin{example}
    \begin{enumerate}[(i)]
        \item If~$\Gamma$ has~$n$ vertices and no edges, so~$\raag \cong F_n$ is free, then the Salvetti complex~$\sal$ is the~$n$-petalled rose, with each edge labelled by a vertex of~$V$. A blowup~$\salb{\Pi}$ is a finite connected graph with no separating edges and no bivalent vertices. The edges dual to hyperplanes labelled by~$\Gamma$-partitions form a maximal tree~$T$. The orientations on the hyperplanes corresponding to vertices~$v \in V$ induce orientations on each edge of~$\salb{\Pi} \setminus T$.
        \item If~$\Pi$ is empty, then~$\salb{\Pi}$ is precisely the Salvetti complex for~$\raag$. It has one oriented hyperplane for each vertex of~$\Gamma$, which induces orientations on the edges (each of which is labelled by a vertex~$v \in V$), and this data determines an isomorphism~$\pi_1(\sal) \cong \raag$. Cube complex automorphisms of~$\sal$ correspond to graph symmetries of~$\Gamma$.
    \end{enumerate}
\end{example}

Suppose that~$\Pi = \{\calP\}$ consists of a single~$\Gamma$-partition~$\left( P \;\vert\; \overline{P} \;\vert\; \link{P} \right)$. Then~$\salb{\Pi} = \salb{\{\calP\}}$ has two vertices~$x$ and~$\overline{x}$, which are joined by one unoriented edge dual to the hyperplane~$H_\calP$, as well as at least one oriented edge labelled by vertices of~$\Gamma$. One can read off~$\calP$ from the orientations and labels as follows:

\begin{itemize}
    \item For~$v \in V$, if there are two edges dual to~$H_v$, then~$v \in \link{\calP}$.
    \item Otherwise, there is only one edge~$e_v$ dual to~$H_v$. Recall that~$H_v$ and~$e_v$ are oriented. Now:
    \begin{itemize}
        \item if~$e_v$ terminates at~$x$, then~$v \in P$, while if it terminates at~$\overline{x}$, then~$v \in \overline{P}$;
        \item if~$e_v$ originates at~$x$, then~$v^{-1} \in P$, while if~$e_v$ originates at~$\overline{x}$, then~$v^{-1} \in \overline{P}$.
    \end{itemize}
\end{itemize}
The hyperplane~$H_\calP$ is isomorphic to the Salvetti complex for the RAAG~$A_{\link{\calP}}$.

\subsection{$\Gamma$-complexes}\label{subsec:Gamma-complexes}

\begin{definition}\label{def:Gamma-cx}
    A cube complex~$X$ is called a \emph{$\Gamma$-complex} if it is isomorphic to the underlying cube complex of a (possibly empty) blowup~$\salb{\Pi}$. A \emph{blowup structure} on~$X$ is the data that specifies it as a particular blowup~$\salb{\Pi}$: that is, each hyperplane is labelled by a vertex of~$\Gamma$ or by a~$\Gamma$-partition, the hyperplanes labelled by vertices are oriented, and the~$\Gamma$-partitions form a compatible set. Any blowup structure on~$X$ determines a \emph{collapse map}~$c_\Pi \colon X \to \sal$ which collapses all hyperplanes labelled by~$\Gamma$-partitions.
\end{definition}

A~$\Gamma$-complex may admit many blowup structures. For example, a~$\Gamma$-complex for an edgeless~$\Gamma$ is a graph which may have many maximal trees---and for each maximal tree, the remaining edges can be labelled by vertices of~$\Gamma$ and oriented in any way.

\begin{definition}\label{def:treelike set of hyperplanes}
    A set of hyperplanes~$\calT$ in a~$\Gamma$-complex is called \emph{treelike} if the hyperplane collapse~$c_\calT$ yields a cube complex isomorphic to~$\sal$.
\end{definition}

Suppose we have a cube complex~$X$ which we do not know to be a~$\Gamma$-complex \emph{a priori}. To prove that~$X$ is in fact a~$\Gamma$-complex, we must find an acyclic collection~$\calT$ of hyperplanes for which the collapse~$X \ssslash \calT$ is isomorphic to~$\sal$. Choosing one such isomorphism gives labels and orientations to the remaining hyperplanes. Next we must prove that each hyperplane~$H \in \calT$ actually determines a valid~$\Gamma$-partition. To do this, we can perform a hyperplane collapse along all of~$\calT \setminus H$, which (if~$\calT$ is to be treelike) should leave us with a complex with exactly two vertices which corresponds to a single blowup along the~$\Gamma$-partition labelling~$H$. As at the end of \S\ref{subsubsec:blowups}, we can then read off the (purported)~$\Gamma$-partition from this complex, by looking at which edges originate and terminate at which vertex. 

Charney--Stambaugh--Vogtmann \cite{CharneyStambaughVogtmann17} gave a construction which ensures that if we can one find one such treelike set in a cube complex~$X$ then the~$\Gamma$-partitions which label the hyperplanes of~$\calT$ form a compatible set; moreover, any treelike set in~$X$ enjoys this property. We summarise this construction in the following proposition and its proof.

\begin{proposition}\label{prop:treelike sets in Gamma-complexes}
    Let~$X$ be a~$\Gamma$-complex and let~$\calT$ be a treelike set of hyperplanes in~$X$. Then there is some compatible set~$\Pi$ of~$\Gamma$-partitions and an isomorphism~$X \cong \calT$ such that~$\calT$ is the set of hyperplanes labelled by the elements of~$\Pi$.
\end{proposition}

\begin{proof}
    For full details, we refer to Section 4 of \cite{CharneyStambaughVogtmann17}.

    A chosen isomorphism~$X \ssslash \calT \cong \sal$ orients and labels the hyperplanes not in~$\calT$ by vertices of~$\Gamma$. Label each edge dual to~$H \in \calT$ by~$H$. The cube subcomplex~$C \subseteq X$ formed of all cubes that have all edge-labels in~$\calT$ can be given the following explicit description: it is a union of products of graphs with subcomplexes of the Salvetti complex.~$\calT$ being treelike corresponds to a maximal forest in the union of these graphs.

   ~$C$ contains all the vertices of~$X$. A hyperplane~$H \in \calT$ disconnects~$C$ into two pieces---thus decomposing the vertices of~$X$ into two sets,~$Z$ and~$\overline{Z}$. Define a~$\Gamma$-partition~$\calP_H = \left( P \;\vert\; \overline{P} \;\vert\; \link{\calP}\right)$ as follows:
    \begin{itemize}
        \item if~$H_v \cap H \neq \emptyset$, then~$v, v^{-1} \in \link{\calP}$;
        \item if~$H_v \cap H = \emptyset$, and an edge dual to~$H_v$ terminates in~$Z$, then~$v \in P$, while if it terminates in~$\overline{Z}$, then~$v \in \overline{P}$;
        \item if~$H_v \cap H = \emptyset$, and an edge dual to~$H_v$ originates in~$Z$, then~$v^{-1} \in P$, while if it terminates in~$\overline{Z}$, then~$v^{-1} \in \overline{P}$.
    \end{itemize}

    It turns out that~$\calP_H$ \emph{is} a valid~$\Gamma$-partition. Roughly, one can find a vertex~$m \in V$ which is split by~$\calP_H$ for which~$\link{m}^\pm = \link{\calP_H}$ by showing that~$v \in \link{m}$ if and only if~$H_v \cap H \neq \emptyset$. Imposing treelikeness of~$\calT$ forces each connected component of~$\Gamma^\pm \setminus \str{m}^\pm$ to have its vertices all on one side of~$\calP_H$.~$\calP_H$ is forced to be thick by the isomorphism~$X \ssslash T \cong \sal$.

    Moreover, the set of~$\Gamma$-partitions constructed in this way form a compatible set~$\Pi$, and we have~$X \cong \salb{\Pi}$.
\end{proof}

So, any treelike set of hyperplanes in a~$\Gamma$-complex is the set of partitions labelled by~$\Gamma$-partitions in at least one blowup structure. The blowup structure is only specified once we additionally assign labels and orientations to the hyperplanes not in the treelike set, and this assignation can be changed by inverting vertices or according to a graph symmetry of~$\Gamma$. Moreover, changing these assignations will also change the~$\Gamma$-partitions labelling the hyperplanes in the treelike set, by the same signed permutation of vertices.

\subsection{Symmetric~$\Gamma$-complexes}

We are now ready to define \emph{symmetric}~$\Gamma$-complexes.

Fix~$\Gamma$ and a~$\Gamma$-complex~$X$. Give~$X$ a blowup structure~$X \cong \salb{\Pi}$, and let~$\calT$ be the (treelike) set of hyperplanes in~$\salb{\Pi}$ labelled by~$\Gamma$-partitions. As in the proof of \emph{Proposition \ref{prop:treelike sets in Gamma-complexes}}, any~$H \in \calT$ partitions the vertices of~$X$ into two sets, say~$\beta(H)$ and~$\overline{\beta}(H)$, and this can be used to read off the~$\Gamma$-partition~$\calP_H$.

Note that such a partition of the vertices of~$X$ into~$\beta(H)$~$\sqcup$~$\overline{\beta}(H)$ makes sense for any hyperplane~$H$ in any treelike set~$\calT$ of hyperplanes of~$X$, by \emph{Proposition \ref{prop:treelike sets in Gamma-complexes}}: for a given treelike set, pick a blowup structure on~$X$ for which~$\calT$ is precisely the set of partitions labelled by~$\Gamma$-partitions. 

\begin{definition}\label{def:symmetric Gamma-complex}
    We say that a~$\Gamma$-complex~$X$ is \emph{symmetric} if for any treelike set of hyperplanes~$\calT$ in~$X$, and for any~$H \in \calT$, the edges traversing between~$\beta(H)$ and~$\overline{\beta}(H)$ are dual to one of exactly two hyperplanes, one of which is~$H$ and the other of which is not in~$\calT$. That is, for each~$H \in \calT$, there is some fixed~$H' \notin \calT$, and the edges traversing between~$\beta(H)$ and~$\overline{\beta}(H)$ are dual to either~$H$ or~$H'$.
\end{definition}

Equivalently, we may define a~$\Gamma$-complex~$X$ to be \emph{symmetric} if every edge lies in a unique (up to homotopy) `simple' non-null-homotopic loop, where by a \emph{simple loop} we mean one that cannot be decomposed as a concatenation of more than one non-null-homotopic loop. We will prove this equivalence later, in \emph{Proposition \ref{prop:equivalence of definitions of symmetric Gamma-complex}}.

The following lemma justifies our terminology; its proof motivates the condition in \emph{Definition \ref{def:symmetric Gamma-complex}}.

\begin{lemma}\label{lem:symmetric Gamma-cx iff partitions are symmetric}
    A~$\Gamma$-complex~$X$ is symmetric if and only if every blowup structure~$X \cong \salb{\Pi}$ corresponds to a compatible set~$\Pi$ consisting only of symmetric~$\Gamma$-partitions.  
\end{lemma}

\begin{proof}
    The forward direction follows by construction. Put a blowup structure~$\salb{\Pi}$ on a symmetric~$\Gamma$-complex~$X$, and read off the corresponding~$\Gamma$-partitions. Each~$\Gamma$-partition~$\calP_H \in \Pi$ corresponds to a hyperplane~$H$ in a treelike set~$\calT$; now apply  to~$H$ the definition of symmetric~$\Gamma$-complex (\emph{Definition \ref{def:symmetric Gamma-complex}}) as follows. Let~$H' \notin \calT$ be as in \emph{Definition \ref{def:symmetric Gamma-complex}}; that is, all edges with one end in~$\beta(H)$ and the other in~$\overline{\beta}(H)$ are dual to either~$H$ or~$H'$. Since~$H' \notin \calT$,~$H'$ must be labelled by some vertex~$v \in V(\Gamma)$. Now, for any vertex~$w \in V(\Gamma)$,~$w$ is split by~$\calP_H$ if and only if the edges labelled by~$w$ have one end at a vertex in~$\beta(H)$ and the other end at a vertex in~$\overline{\beta}(H)$. Hence the only vertex split by~$\calP_H$ is~$v$. Hence~$\calP_H$ is symmetric.

    For the converse, let~$\calT$ be a treelike set of partitions in~$X$. By \emph{Proposition \ref{prop:treelike sets in Gamma-complexes}}, there is some blowup structure~$X \cong \salb{\Pi}$ such that~$\calT$ is precisely the set of hyperplanes labelled by elements of~$\Pi$. By hypothesis, each element of~$\Pi$ is a symmetric~$\Gamma$-partition. Let~$H \in \calT$; since~$\calP_H$ is symmetric, there is only one vertex~$v \in V(\Gamma)$ split by~$\calP_H$. This means that all edges with one end at a vertex in~$\beta(H)$ and the other end at~$\overline{\beta}(H)$ are labelled by either~$H$ or~$v$, so~$H$ satisfies the condition in \emph{Definition \ref{def:symmetric Gamma-complex}}. This is true for every element of~$\calT$, and every treelike set~$\calT$, so~$X$ is symmetric, as required.
\end{proof}

\section{The symmetric spine~$\symspine$}\label{sec:symmetric spine}

Fix a defining graph~$\Gamma$. In this section we build the \emph{symmetric spine}~$\symspine$. We emphasise that the construction of the spine of untwisted outer space~$\spine$ in \cite{CharneyStambaughVogtmann17} carries over to this setting with minimal modification.

\subsection{Preparatory results}\label{subsec:preparatory results}

For a~$\Gamma$-partition~$\calP$, let~$\operatorname{split}(\calP)$ be the set of vertices of~$V(\Gamma)$ which are split by~$\calP$. 

For~$\Pi$ a compatible set of~$\Gamma$-partitions, write~$c_\Pi \colon \salb{\Pi} \to \sal$ for the hyperplane collapse along all hyperplanes labelled by elements of~$\Pi$. Write~$c_\Pi^{-1}$ for a homotopy inverse of~$c_\Pi$. If~$\Pi = \{\calP\}$ is a singleton, we will write~$c_{\{\calP\}} = c_\calP$ (and~$\salb{\{\calP\}} = \salb{\calP}$).

We quote the following two results from \cite{CharneyStambaughVogtmann17}.

\begin{lemma}[\cite{CharneyStambaughVogtmann17}, Lemma 3.2]\label{lem:CSV17, 3.2}
    Let~$\calP$ be a~$\Gamma$-Whitehead partition. For each~$m \in \operatorname{split}(\calP)$ which is maximal in~$\operatorname{split}(\calP)$, there is an automorphism~$h_m$ of~$\salb{\calP}$.
    Then the composition 
    \[c_{\calP} \circ h_m \circ c_{\calP}^{-1} \; \colon \;\; \sal \to \salb{\calP} \smash{\xrightarrow{\cong}} \salb{\calP} \to \sal\] 
    induces the~$\Gamma$-Whitehead automorphism~$\varphi(\calP, m)$.
\end{lemma}

\begin{theorem}[\cite{CharneyStambaughVogtmann17}, Theorem 4.12]\label{thm:CSV17, 4.12}
    Let~$\salb{\Pi}$ be a blowup. Let~$\calH$ and~$\calK$ be two treelike sets of partitions in~$\salb{\Pi}$. For any~$K \in \calK$, there exists~$H \in \calH$ such that the set obtained from~$\calH$ by replacing~$H$ by~$K$ is again treelike.
\end{theorem}

\begin{lemma}\label{lem:symmetric blowup-collapses induce symmetric automorphisms}
    Let~$\Pi = \{\calP_1, \dots, \calP_n\}$ be a compatible set of symmetric~$\Gamma$-partitions, and let~$H_i$ be the hyperplane in~$\salb{\Pi}$ labelled by~$\calP_i$ for each~$i = 1, \dots, n$. Let~$\calK$ be any other treelike set of hyperplanes in~$\salb{\Pi}$. Then the automorphism of~$\raag$ induced by the blowup collapse 
    \[\sal \xrightarrow{c_\Pi^{-1}} \salb{\Pi} \to \left(\salb{\Pi}\right)_\calK \cong \sal\]
    lies in~$\symout$.
\end{lemma}

\begin{proof}
    The proof is exactly as that of (\cite{CharneyStambaughVogtmann17}, Corollary 4.13), using \emph{Lemma \ref{lem:CSV17, 3.2}}, and replacing the use of (\cite{CharneyStambaughVogtmann17}, Lemma 2.2) with the observation that~$\Gamma$-Whitehead automorphisms corresponding to symmetric~$\Gamma$-partitions lie in~$\symout$.
\end{proof}

\subsection{(Symmetric) marked~$\Gamma$-complexes and the (symmetric) spine}\label{subsec:symmetric marked Gamma-complexes and the symmetric spine}

In this subsection, we recall the definition of the (untwisted) spine~$\spine$, as in \cite{CharneyStambaughVogtmann17}, and analogously introduce the \emph{symmetric spine}~$\symspine$.

\begin{definition}\label{def:(symmetric) marked Gamma complex}
    A \emph{marked~$\Gamma$-complex} is a pair~$\sigma = (X, \alpha)$, where~$X$ is a~$\Gamma$-complex and~$\alpha \colon X \to \sal$ is a homotopy equivalence. We say that~$\alpha$ is an \emph{untwisted marking} if the composition 
    \[\sal {\xrightarrow{c_\Pi^{-1}}} \salb{\Pi} \cong X \xrightarrow{\alpha} \sal\] 
    induces an element of~$\uag$.

    A marked~$\Gamma$-complex~$(X, \alpha)$ is \emph{symmetric} if~$X$ is a symmetric~$\Gamma$-complex and~$\alpha$ is a \emph{symmetric marking}; that is, the composition 
    \[\sal {\xrightarrow{c_\Pi^{-1}}} \salb{\Pi} \cong X \xrightarrow{\alpha} \sal\] 
    induces an element of~$\symout$.
\end{definition}

Two marked~$\Gamma$-complexes~$\sigma = (X, \alpha)$,~$\sigma' = (X', \alpha')$ are \emph{equivalent} if there exists a cube complex isomorphism~$h \colon X \to X'$ such that~$\alpha' \circ h \simeq \alpha$. We will sometimes use the term `marked~$\Gamma$-complex' to refer to one of these equivalence classes.

For a (symmetric) marked~$\Gamma$-complex~$(X, \alpha)$ such that~$X$ is isomorphic to~$\sal$, we call the equivalence class of~$(X, \alpha)$ a \emph{(symmetric) marked Salvetti}.

We now define a partial order on marked~$\Gamma$-complexes. If~$\sigma = (X, \alpha)$ is a marked~$\Gamma$-complex,~$\calT$ is a treelike set of hyperplanes in~$X$ with some subset~$\calH \subseteq \calT$ and~$c_\calH \colon X \to X \ssslash \calH$ is the collapse map, denote~$\sigma_\calH \coloneqq \left( X \ssslash \calH, \alpha \circ c_\calH^{-1}\right)$. Now for two marked~$\Gamma$-complexes~$\sigma, \sigma'$, define~$\sigma' < \sigma$ if~$\sigma' = \sigma_\calH$ for some~$\calH$.

We can restrict this partial order to the set of symmetric~$\Gamma$-complexes. We are now ready to define the spine~$\spine$, and the symmetric spine~$\symspine$.

\begin{definition}\label{def:(symmetric) spine}
    The \emph{(untwisted) spine}~$\spine$ is the simplicial complex which is the geometric realisation of the poset of (equivalence classes of) marked~$\Gamma$-complexes.

    The \emph{symmetric spine}~$\symspine$ is the simplicial complex which is the geometric realisation of the poset of (equivalence classes of) symmetric marked~$\Gamma$-complexes.
\end{definition}

Identifying~$\out$ with the group of homotopy classes of maps~$\sal \to \sal$, define a left action of~$\uag$ (resp.~$\symout$) on~$\spine$ (resp.~$\symspine$) by~$\varphi \cdot (X, \alpha) = (X, \varphi \circ \alpha)$.

Charney--Stambaugh--Vogtmann \cite{CharneyStambaughVogtmann17} now prove that~$\spine$ is connected, and its~$\uag$-action is proper. We mimic their proofs to prove the analogous results in the symmetric case.

\begin{lemma}\label{lem:action on symmetric spine is proper}
    The action of~$\symout$ on~$\symspine$ is proper.
\end{lemma}

\begin{proof}
    Each symmetric marked~$\Gamma$-complex can be collapsed to finitely many marked Salvettis, so it is sufficient to prove that the stabiliser of some (hence any) marked Salvetti is finite. Any isomorphism~$\sal \to \sal$ takes the 1-skeleton to the 1-skeleton, so on~$\left( \sal, \operatorname{id}\right)$ induces a permutation of~$V^\pm$. So~$\left( \sal, \alpha \right) \sim \left(\sal, \operatorname{id}\right)$ if and only if~$\alpha \in \Omega(\raag) \leq \symout$, the (finite) group generated by graph symmetries and inversions.
\end{proof}

Let~$(\calP, m)$ be a symmetric~$\Gamma$-Whitehead pair and~$\alpha = \varphi\left(\calP, m\right)$ the corresponding (symmetric)~$\Gamma$-Whitehead automorphism. By \emph{Lemma \ref{lem:CSV17, 3.2}},~$\alpha$ is realised by a blowup-collapse 
\[\alpha = c_\calP \circ h_m \circ c_\calP^{-1} \; \colon \;\; \sal \to \salb{\calP} \xrightarrow{\cong} \salb{\calP} \to \sal.\] 
This gives a path in~$\symspine$ between~$\left(\sal, \operatorname{\id}\right)$ and~$\left(\sal, \alpha\right)$: 
\begin{equation*}
    \left(\sal, \operatorname{id}\right) < \left(\salb{\calP}, c_\calP\right) \sim \left(\salb{\calP}, c_\calP \circ h_m\right) = \left(\salb{\calP}, \alpha \circ c_\calP\right) > \left(\sal, \alpha\right).
\end{equation*} 
Now for any~$\varphi \in \symout$, we can translate this path by~$\varphi$ to obtain a path from~$\left(\sal, \varphi\right)$ to~$\left(\sal, \varphi \circ \alpha\right)$. This is called the \emph{Whitehead move} at~$\left(\sal, \varphi\right)$ associated to~$\left(\calP, m\right)$. If~$\sigma = \left(\sal, \varphi\right)$, we will write~$\sigma_m^\calP = \left(\sal, \varphi \circ \alpha\right)$.

Following \cite{CharneyStambaughVogtmann17}, we can restate \emph{Lemma \ref{lem:symmetric blowup-collapses induce symmetric automorphisms}} using this terminology. 

\begin{corollary}[Factorisation Lemma]\label{cor:factorisation lemma (i.e. CSV17, 4.19)}
    Let~$\sigma = \left(\sal, \alpha\right)$ be a marked Salvetti,~$\Pi = \{\calP_1, \dots, \calP_n\}$ a compatible set of symmetric~$\Gamma$-partitions, and~$\calT$ a treelike set of hyperplanes in~$\salb{\Pi}$. We write~$\sigma^\Pi = \left(\salb{\Pi}, c_\Pi \circ \alpha\right)$. 

    Then with a suitable ordering of the elements of~$\calH$, there is a chain~$\sigma = \sigma_0, \sigma_1, \dots, \sigma_n = \sigma_\calH^\Pi$ such that each~$\sigma_i$ is connected to~$\sigma_{i-1}$ by a Whitehead move.
\end{corollary}

\begin{proposition}\label{prop:symmetric spine is connected}
   ~$\symspine$ is connected.
\end{proposition}

\begin{proof}
    Let~$N \leq \symout$ be the subgroup generated by symmetric~$\Gamma$-Whitehead automorphisms (that is, generated by partial conjugations), and write~$\Omega(\raag)$ for the (finite) subgroup generated by graph symmetries and inversions. 

    Let~$\omega \in \Omega(\raag)$, and let~$\rho$ be a partial conjugation based at~$m \in V$. We show that~$\omega^{-1}\rho\omega$ is also a partial conjugation. Clearly, for any~$v \in V$, if~$\rho(\omega(v)) = \omega(v)$, then~$\left( \omega^{-1} \circ \rho \circ \omega\right) (v) = v$. The other possibility is that~$\rho(\omega(v)) = m \omega(v) m^{-1}$. Then 
    \[\left(\omega^{-1} \circ \rho \circ \omega\right) (v) = \omega^{-1}\left(m \omega(v) m^{-1}\right) = \omega^{-1}(m)v\omega^{-1}\left(m^{-1}\right).\] 
    Hence we recognise~$\omega^{-1}\rho\omega$ as a partial conjugation based at~$\omega^{-1}(m)$. Note that if~$\rho$ is symmetric, then~$\omega^{-1} \rho \omega$ is too.

    We conclude that~$N \triangleleft \symout$ is normal. Therefore, any~$\varphi \in \symout$ can be factored as a product~$\varphi = \varphi_\Omega \circ \varphi_N$, where~$\varphi_\Omega$ is a product of graph symmetries and inversions and~$\varphi_N$ is an element of~$N$.

    Since~$\varphi_\Omega$ can be represented by an isomorphism~$\sal \to \sal$, we have~$\left(\sal, \operatorname{\id}\right) \sim \left(\sal, \varphi_\Omega\right)$. By \emph{Corollary \ref{cor:factorisation lemma (i.e. CSV17, 4.19)}},~$\left(\sal, \varphi_\Omega\right)$ is connected by a path in~$\symspine$ to~$\left(\sal, \varphi_\Omega \circ \varphi_N\right)$. That is,~$\left(\sal, \operatorname{\id}\right)$ is connected to~$\left(\sal, \varphi\right)$ by a path in~$\symspine$. But by the definition of the poset of (equivalence classes of) symmetric marked~$\Gamma$-complexes, every vertex of~$\symspine$ lies in the star of some marked Salvetti, so we're done.
\end{proof}

The quotient of~$\symspine$ by the action of~$\symout$ is compact, so we now have a proper and cocompact action of~$\symout$ on the connected complex~$\symspine$. It remains to show that~$\symspine$ is in fact contractible.

\section{Discussion of proof of contractibility of~$\symspine$}\label{sec:discussion of proof}

As discussed in the introduction (\S\ref{subsec:discussion of proof strategy}), our proof of contractibility of~$\symspine$ proceeds in two broad steps, inspired by Collins \cite{CollinsSymSpine89}. We now outline our strategy in more detail and explain its relation to the main relevant literature. 

Culler--Vogtmann \cite{CullerVogtmann86} proved that Outer space~$CV_n$ is contractible by showing that its spine~$K_n$, a deformation retract of~$CV_n$, is contractible. Culler and Vogtmann defined a norm~$\lVert \cdot \rVert_W$ on marked rose graphs (which are the marked Salvettis in the free group case) for an arbitrary finite set~$W$ of conjugacy classes of the free group. They then proved that the subcomplex~$K_{\operatorname{min}(W)} \subseteq K_n$, defined as the union of all marked roses which are minimal with respect to the norm~$\lVert \cdot \rVert_W$, is in fact a deformation retract of~$K_n$. Their final step was to pick a set~$W$ for which~$K_{\operatorname{min}(W)}$ is obviously contractible, whence~$K_n$ is contractible. 

This means that~$K_{\operatorname{min}(W)}$ is contractible for any choice of~$W$. Collins \cite{CollinsSymSpine89} notes that by wisely choosing~$W$ so that the symmetric spine~$K_n^\Sigma$ is a subcomplex of~$K_{\operatorname{min}(W)}$, contractibility of~$K_n^\Sigma$ follows by proving that~$K_n^\Sigma$ is a retract of~$K_{\operatorname{min}(W)}$. This last part forms most of the technical work in \cite{CollinsSymSpine89}.

We set out to mimic this strategy, with marked Salvettis replacing marked roses, and the untwisted spine~$\spine$ replacing~$K_n$. However, Charney--Stambaugh--Vogtmann's proof \cite{CharneyStambaughVogtmann17} of the contractibility of~$\spine$ is more direct than the analogous proof of Culler--Vogtmann \cite{CullerVogtmann86}: they picked a single norm which totally well-orders the marked Salvettis, and built~$\spine$ star-of-marked-Salvetti by star-of-marked-Salvetti, with stars attached in increasing order according to this norm. Specialised to the free group case, this is almost a backwards version of Culler--Vogtmann's original proof: fixing a judicious choice of conjugacy classes at the start, so therefore starting with the star of a single marked Salvetti (which is obviously contractible), and building up to the entire spine. This has been explained in detail by Vogtmann \cite{VogtmannCVnReprise17}. The upshot is that the proof of contractibility of~$\spine$ in \cite{CharneyStambaughVogtmann17} does not automatically give analogous contractible intermediary subspaces such as the~$K_{\operatorname{min}(W)}$ in the free group case. 

Consequently, our main technical contribution is to define, and prove contractibility of, analogous spaces~$\Kmin{\calW}$ for~$\calW$ an arbitrary finite set of conjugacy classes in~$\raag$. To do this, we combine the idea of a norm~$\lVert \cdot \rVert_W$ from \cite{CullerVogtmann86} with the `build upwards' strategy of \cite{CharneyStambaughVogtmann17}. Once we have proven that this new norm gives a total well-order of the marked Salvettis, we can replicate Charney--Stambaugh--Vogtmann's inductive gluing argument. Stopping at the right height will build precisely the intermediary complex~$\Kmin{\calW}$. Hence, each of these intermediary complexes is contractible. Choosing~$\calW$ to be the set of conjugacy classes of length at most two `recovers' (or more precisely, simply is) Charney--Stambaugh--Vogtmann's proof of contractibilty of~$\spine$. On the other hand, choosing~$\calW$ to be the set of conjugacy classes of length one yields a complex~$\Kmin{\calW}$ which contains our symmetric spine~$\symspine$. This will be the space~$\smash{K_\Gamma^{\text{sym}}}$ as referred to in \S\ref{subsec:discussion of proof strategy}.

The main edit we make to Charney--Stambaugh--Vogtmann's proof of contractibility of~$\spine$ is to change the norm on marked Salvettis. All that remains is then to show that the same proof goes through with the new norm. Essentially, we must check that all of \S6 of \cite{CharneyStambaughVogtmann17} works with our norm. Some results in (\cite{CharneyStambaughVogtmann17}, \S6) do not depend on the norm and pass to our setting with no work. Others require edits or rewriting. 

In order to keep track of what needs to be transferred to our setting, and to keep track of the logical dependencies amongst (\cite{CharneyStambaughVogtmann17}, \S6), we have provided \emph{Figure \ref{fig:dependence diagram in CSV17 Ch. 6}}, with the following clarifying remarks.
\begin{enumerate}[(i)]
    \item 6.24 is the main theorem: contractibility of~$\spine$.
    \item 6.25 is the Poset Lemma (\emph{Lemma \ref{lem:poset lemma}}). 
    \item 6.18 is `Peak Reduction', while 6.19 is `Strong Peak Reduction'. 6.18 is a good comparison point to earlier literature, such as \cite{CullerVogtmann86}, but only 6.19 is used in later proofs in \cite{CharneyStambaughVogtmann17}. Their proofs are very similar.
\end{enumerate}

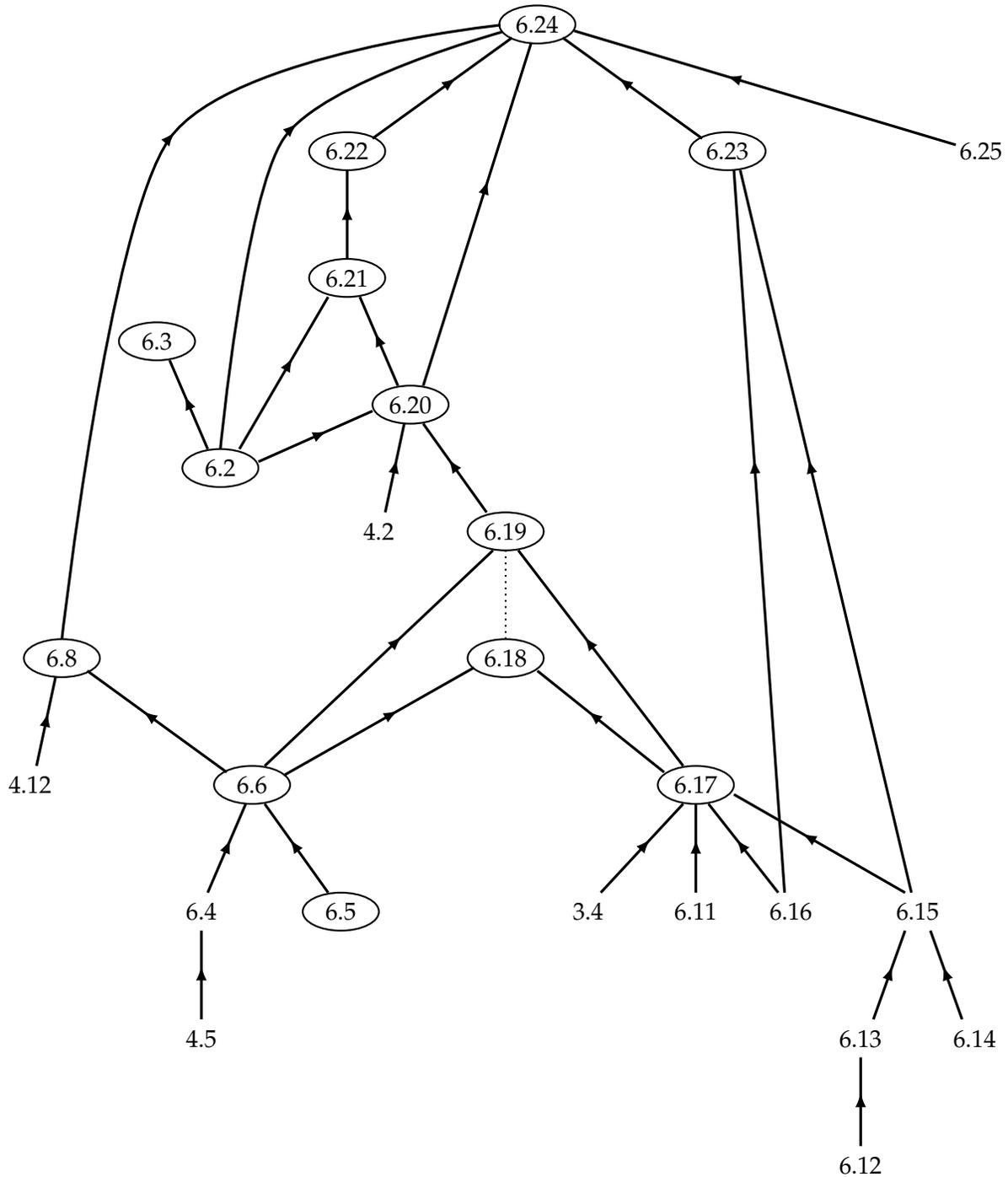
\begin{figure}
    \centering
    \begin{tikzpicture}
        %
        \node at (0, 0) [black, thick]{6.24};
        \node at (-3, -2) [black, thick]{6.22};
        \node at (3, -2) [black, thick]{6.23};
        \node at (7, -2) [black, thick]{6.25};
        \node at (-3, -4) [black, thick]{6.21};
        \node at (-5, -7) [black, thick]{6.2};
        \node at (-2, -6) [black, thick]{6.20};
        \node at (-2.5, -8) [black, thick]{4.2};
        \node at (-0.5, -8) [black, thick]{6.19};
        \node at (-6, -5) [black, thick]{6.3};
        \node at (-0.5, -10) [black, thick]{6.18};
        \node at (-4.5, -12) [black, thick]{6.6};
        \node at (-5.3, -14) [black, thick]{6.4};
        \node at (-3.1, -14) [black, thick]{6.5};
        \node at (-5.3, -16) [black, thick]{4.5};
        \node at (-7.5, -10) [black, thick]{6.8};
        \node at (-8, -12) [black, thick]{4.12};
        \node at (2.5, -12) [black, thick]{6.17};
        \node at (0.8, -14) [black, thick]{3.4};
        \node at (2.5, -14) [black, thick]{6.11};
        \node at (4, -14) [black, thick]{6.16};
        \node at (6, -14) [black, thick]{6.15};
        \node at (5.1, -16) [black, thick]{6.13};
        \node at (6.9, -16) [black, thick]{6.14};
        \node at (5.1, -18) [black, thick]{6.12};
        %
        %
        \draw[black, thick] (0, 0) ellipse (0.6 and 0.3); 
        \draw[black, thick] (-3, -2) ellipse (0.6 and 0.3); 
        \draw[black, thick] (3, -2) ellipse (0.6 and 0.3); 
        \draw[black, thick] (-3, -4) ellipse (0.6 and 0.3); 
        \draw[black, thick] (-5, -7) ellipse (0.6 and 0.3); 
        \draw[black, thick] (-2, -6) ellipse (0.6 and 0.3); 
        \draw[black, thick] (-0.5, -8) ellipse (0.6 and 0.3); 
        \draw[black, thick] (-6, -5) ellipse (0.6 and 0.3); 
        \draw[black, thick] (-0.5, -10) ellipse (0.6 and 0.3); 
        \draw[black, thick] (-4.5, -12) ellipse (0.6 and 0.3); 
        \draw[black, thick] (-3.1, -14) ellipse (0.6 and 0.3); 
        \draw[black, thick] (-7.5, -10) ellipse (0.6 and 0.3); 
        \draw[black, thick] (2.5, -12) ellipse (0.6 and 0.3); 
        %
        %
        \draw[black, very thick, ->-] (-2.6, -1.8) -- (-0.4, -0.21); 
        \draw[black, very thick, ->-] (2.6, -1.8) -- (0.4, -0.21); 
        \draw[black, very thick, ->-] (6.6, -1.9) -- (0.58, -0.1); 
        \draw[black, very thick, ->-] (-3, -3.7) -- (-3, -2.3); 
        \draw[black, very thick, ->-] (-2.2, -5.7) -- (-2.8, -4.3); 
        \draw[black, very thick, ->-] (-1.8, -5.7) -- (-0.1, -0.3); 
        \draw[black, very thick, ->-] (-4.7, -6.7) -- (-3.3, -4.3); 
        \draw[black, very thick, ->-] (-5.2, -6.7) -- (-5.8, -5.3); 
        \draw[smooth, tension=0.1, black, very thick, ->-] plot coordinates {(-5, -6.7) (-4, -1.8) (-0.56, -0.12)}; 
        \draw[black, very thick, ->-] (-4.4, -6.9) -- (-2.6, -6.1); 
        \draw[black, very thick, ->-] (-2.4, -7.7) -- (-2.1, -6.3); 
        \draw[black, very thick, ->-] (-0.8, -7.7) -- (-1.8, -6.3); 
        \draw[smooth, tension=0.1, black, very thick, ->-] plot coordinates {(-7.5, -9.7) (-5.8, -1.8) (-0.58, -0.01)}; 
        \draw[black, very thick, ->-] (-7.9, -11.7) -- (-7.6, -10.3); 
        \draw[black, very thick, ->-] (-4.9, -11.8) -- (-7.1, -10.2); 
        \draw[black, very thick, ->-] (-4, -11.85) -- (-1, -10.15); 
        \draw[black, very thick, ->-] (-5.2, -13.7) -- (-4.6, -12.3); 
        \draw[black, very thick, ->-] (-3.3, -13.7) -- (-4.3, -12.3); 
        \draw[black, very thick, ->-] (-5.3, -15.7) -- (-5.3, -14.3); 
        \draw[black, very thick, ->-] (-4.3, -11.7) -- (-0.7, -8.3); 
        \draw[black, very thick, ->-] (2.3, -11.7) -- (-0.3, -8.3); 
        \draw[black, very thick, ->-] (2, -11.8) -- (0, -10.2); 
        \draw[black, thick, dotted] (-0.5, -9.7) -- (-0.5, -8.3); 
        \draw[black, very thick, ->-] (1, -13.7) -- (2.3, -12.3); 
        \draw[black, very thick, ->-] (2.5, -13.7) -- (2.5, -12.3); 
        \draw[black, very thick, ->-] (3.8, -13.7) -- (2.7, -12.3); 
        \draw[black, very thick, ->-] (5.8, -13.7) -- (3.1, -12.15); 
        \draw[black, very thick, ->-] (3.9, -13.7) -- (3.1, -2.3); 
        \draw[black, very thick, ->-] (5.9, -13.7) -- (3.2, -2.3); 
        \draw[black, very thick, ->-] (5.3, -15.7) -- (5.8, -14.3); 
        \draw[black, very thick, ->-] (6.7, -15.7) -- (6.2, -14.3); 
        \draw[black, very thick, ->-] (5.1, -17.7) -- (5.1, -16.3); 
    \end{tikzpicture}
    \caption{Each node is a numbered item from \cite{CharneyStambaughVogtmann17}. A directed edge from a node~$A$ to a node~$B$ encodes that~$A$ is used in the proof of~$B$. A circled label signifies that the norm defined in \cite{CharneyStambaughVogtmann17} is directly used at some point in the statement or proof of that item; an uncircled label corresponds to an item for which this is not the case.}
    \label{fig:dependence diagram in CSV17 Ch. 6}
\end{figure}

On the face of it, we need to be able to reproduce \emph{Figure \ref{fig:dependence diagram in CSV17 Ch. 6}}---i.e., we need to prove that all the results corresponding to circled labels carry through with our new norms. However, we will be able to avoid depending on Corollary 6.20 from \cite{CharneyStambaughVogtmann17}. This is because Corollary 6.20 is used in two ways in \cite{CharneyStambaughVogtmann17}, both of which can be circumvented in our case, as follows.
\begin{itemize}
    \item In the proof of Corollary 6.21, and therefore (indirectly) in proving Proposition 6.22, which states that Charney--Stambaugh--Vogtmann's norm well-orders the set of marked Salvettis. However, we will prove that our norms well-order the set of marked Salvettis as a corollary of (\cite{CharneyStambaughVogtmann17}, Corollary 6.21 \& Proposition 6.22). This means that we do not need to go to the effort of proving results analogous to these in our setting---we can simply use Charney--Stambaugh--Vogtmann's as they stand, and hence (\cite{CharneyStambaughVogtmann17}, Corollary 6.20) also unaltered.
    \item In the proof of Theorem 6.24, where it is used to guarantee that when one inductively glues the star of a marked Salvetti~$\sigma$ to the union of the stars of marked Salvettis with norm less than~$\lVert \sigma \rVert$, the intersection along with one glues is nonempty. We skip this step in our corresponding proof, which means that we can only conclude that the spine~$\spine$ is a union of contractible components. However, since we know that~$\spine$ is connected, this knowledge is enough to deduce contractibility. 
\end{itemize}

Therefore, we are actually able to avoid adapting Corollary 6.20 to our setting. As can be seen from \emph{Figure \ref{fig:dependence diagram in CSV17 Ch. 6}}, this also obviates the need to transfer (Strong) Peak Reduction (Theorems 6.18, 6.19) and the Higgins--Lyndon Lemma (Lemma 6.17) to our setting. Whenever relevant, we will point out exactly where we avoid the use of Corollary 6.20 (and the results used in its proof). In \emph{Figure \ref{fig:dependence diagram for us}}, we lay out the various results, and their relation to each other, which together prove contractibility of~$\Kmin{\calW}$, and also indicate the results from \cite{CharneyStambaughVogtmann17} which they adapt.

We make one final comment regarding our choice of norm. In \cite{CullerVogtmann86}, the norm on marked roses is defined as follows. For an arbitrary finite set~$W$ of conjugacy classes in~$F_n$, and for~$\rho$ a marked rose, define 
\[\lVert \rho \rVert \coloneqq n \sum_{w \in W} l(w),\]
where~$l(w)$ is the length function naturally associated to the marking on~$\rho$. Then, for any fixed~$W$, one defines~$K_{\operatorname{min}(W)}$ as the union of the stars of all marked roses which are minimal with respect to this norm. The final step of the proof of contractibility of the spine~$K_n$ is to notice that if one chooses~$W$ to be the set of conjugacy classes of length at most two, then~$K_{\operatorname{min}(W)}$ consists of exactly one (star of a) marked rose~$\rho = (R_n, \operatorname{id})$, and so is contractible.

A slightly more complicated norm, this time on marked Salvettis instead of marked roses, is used in \cite{CharneyStambaughVogtmann17}. Let~$\calG = \left(g_1, g_2, \dots\right)$ be a list of all the conjugacy classes of a RAAG~$\raag$, and let~$\calG_0$ be the set of all conjugacy classes which can be represented by a word of length at most two. For any marked Salvetti~$\sigma = \left(\sal, \alpha\right)$ and any conjugacy class~$g$ of~$\raag$, define~$\ell_\sigma(g)$ to be the minimal length of a word in the free group~$F\langle V(\Gamma)\rangle$ representing an element of the conjugacy class~$\alpha^{-1}(g)$ in~$\raag$ (on roses, this is precisely the length function used in \cite{CullerVogtmann86} as above). They now define the norm of~$\sigma$ to be 
\[\lVert \sigma \rVert \coloneqq \left(\lVert \sigma \rVert_0,\; \lVert \sigma \rVert_1,\; \lVert \sigma \rVert_2,\; \dots\right),\]
where 
\[\lVert \sigma \rVert_0 \coloneqq \sum_{g \in \calG_0} \ell_\sigma(g) \;\;;\quad \quad \quad \lVert \sigma \rVert_i \coloneqq \ell_\sigma(g_i).\]
The norm is taken to live in the ordered abelian group~$\mathbb{Z} \times \mathbb{Z}^{\calG}$, ordered by the lexicographical order. In view of the suite of norms used in \cite{CullerVogtmann86}, it would seem natural to define our new norms by allowing~$\calG_0$ to be any finite set of conjugacy classes. However, we will actually do something slightly different, which will allow us to use more of the properties of Charney--Stambaugh--Vogtmann's norm~$\lVert \cdot \rVert$. For a finite set of conjugacy classes~$\calW$, we define 
\[\lVert \sigma \rVert_\calW \coloneqq \sum_{g \in \calW} \ell_\sigma(g),\] 
and then we define our new norm~$\lVert \cdot \rVert'$ by 
\[\lVert \sigma \rVert' \coloneqq \left( \lVert \sigma \rVert_\calW,\; \lVert \sigma \rVert\right) \in \mathbb{Z} \times \mathbb{Z} \times \mathbb{Z}^{\calG}.\] 
This trick enables us to avoid proving an analogue of Corollary 6.20, as discussed above. 

In \S\ref{sec:contractibility of Kmin}, we perform this adaptation of (\cite{CharneyStambaughVogtmann17}, \S6), proving contractibility of~$\Kmin{\calW}$ for any set of conjugacy classes~$\calW$. Then in \S\ref{sec:retraction of Kmin to symspine}, we adapt Collins's work \cite{CollinsSymSpine89}, proving that~$\symspine$ is a deformation retract of a particular~$\Kmin{\calW}$, which completes the proof.

We finish this section by presenting the dependence diagram for \S\ref{sec:contractibility of Kmin} (see \emph{Figure \ref{fig:dependence diagram for us}}), with the following accompanying explanatory notes.
\begin{enumerate}[(i)]
    \item \emph{Theorem 6.12} is the main theorem, that~$\Kmin{\calW}$ is contractible.
    \item \emph{Lemma \ref{lem:poset lemma}} is the Poset Lemma.
    \item We call \emph{Lemma \ref{lem:pushing lemma (for us)}} the `Pushing Lemma', in reference to the comparable result in \cite{CharneyStambaughVogtmann17}.
\end{enumerate}

\begin{figure}
    \centering
    \begin{tikzpicture}
        %
        \node at (0, 0) [black, thick]{\emph{Theorem \ref{thm:Kmin is contractible}}};
        \node at (0, 0) [black, thick, below = 3pt]{(cf. 6.24)};
        \node at (-7, -3) [black, thick]{\emph{Corollary \ref{cor:analogue of [CSV17], Cor 6.8}}};
        \node at (-7, -3) [black, thick, below = 3pt]{(cf. 6.8)};
        \node at (-7, -6) [black, thick]{\emph{Corollary \ref{cor:analogue of [CSV17], Cor 6.6}}};
        \node at (-7, -6) [black, thick, below = 3pt]{(cf. 6.6)};
        \node at (-7, -9) [black, thick]{\emph{Lemma \ref{lem:CSV17 Lem 6.4}}};
        \node at (-7, -9) [black, thick, below = 3pt]{(= 6.4)};
        \node at (-2, -3) [black, thick]{\emph{Proposition \ref{prop:new norm is a strict total well-ordering}}};
        \node at (-2, -3) [black, thick, below = 3pt]{(cf. 6.22)};
        \node at (-3.5, -5) [black, thick]{(\cite{CharneyStambaughVogtmann17}, 6.3)};
        \node at (0, -5) [black, thick]{(\cite{CharneyStambaughVogtmann17}, 6.21)};
        \node at (3, -3) [black, thick]{\emph{Lemma \ref{lem:pushing lemma (for us)}}};
        \node at (3, -3) [black, thick, below = 3pt]{(cf. 6.23)};
        \node at (2, -7.5) [black, thick]{\emph{Lemma \ref{lem:[CSV17], Lemma 6.16}}};
        \node at (2, -7.5) [black, thick, below = 3pt]{(= 6.16)};
        \node at (5, -6) [black, thick]{\emph{Corollary \ref{cor:Lemma 6.15 + Lemma 6.16}}};
        \node at (5, -9) [black, thick]{\emph{Lemma \ref{lem:[CSV17], Lemma 6.15 special case}}};
        \node at (5, -11) [black, thick]{(\cite{CharneyStambaughVogtmann17}, 6.15)};
        \node at (7, -3) [black, thick]{\emph{Lemma \ref{lem:poset lemma}}};
        \node at (7, -3) [black, thick, below = 3pt]{(= 6.25)};
        %
        %
        \draw[black, very thick, ->-] (-6, -2.6) -- (-1, -0.4); 
        \draw[black, very thick, ->-] (-7, -5.6) -- (-7, -3.6); 
        \draw[black, very thick, ->-] (-7, -8.6) -- (-7, -6.6); 
        \draw[black, very thick, ->-] (-1.5, -2.6) -- (-0.4, -0.6); 
        \draw[black, very thick, ->-] (-3.2, -4.6) -- (-2.5, -3.6); 
        \draw[black, very thick, ->-] (-0.3, -4.6) -- (-1.4, -3.6); 
        \draw[black, very thick, ->-] (2.6, -2.6) -- (0.6, -0.6); 
        \draw[black, very thick, ->-] (2.2, -7.1) -- (2.8, -3.6); 
        \draw[black, very thick, ->-] (4.6, -5.6) -- (3.4, -3.6); 
        \draw[black, very thick, ->-] (5, -8.6) -- (5, -6.4); 
        \draw[black, very thick, ->-] (3, -7.2) -- (4, -6.3); 
        \draw[black, very thick, ->-] (5, -10.6) -- (5, -9.4); 
        \draw[black, very thick, ->-] (6, -2.6) -- (1, -0.4); 
    \end{tikzpicture}
    \caption{Dependency diagram for \S\ref{sec:contractibility of Kmin} of this document. Each node either contains a result from this paper (in italics) or a result from \cite{CharneyStambaughVogtmann17}. All nodes containing a result from the present document also indicate the comparable (indicated by `cf.') or equivalent (indicated by `=') result from \cite{CharneyStambaughVogtmann17} (upright text). A directed edge from a node~$A$ to a node~$B$ indicates that~$A$ implies, or is used in the proof of,~$B$.}
    \label{fig:dependence diagram for us}
\end{figure}
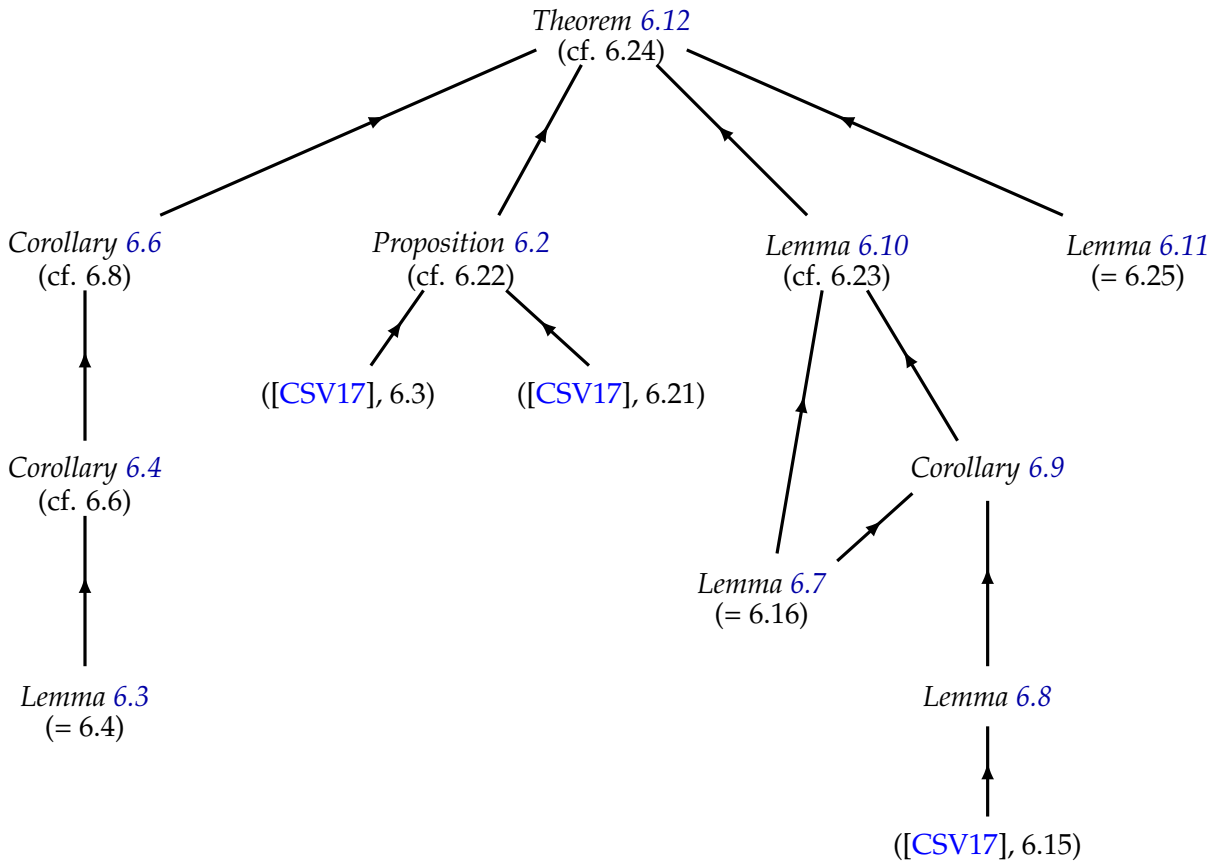

\section{Contractibility of~$\Kmin{\calW}$}\label{sec:contractibility of Kmin}

Fix~$\Gamma$. In this subsection, we prove that~$\Kmin{\calW}$ is contractible, for any choice of finite set of conjugacy classes~$\calW$.

\subsection{The~$\calW$-norm of a marked Salvetti}\label{subsec:the W-norm of a marked Salvetti}

In this section we define our adaptations of the norm defined in \cite{CharneyStambaughVogtmann17}. This subsection is a direct transfer of (\cite{CharneyStambaughVogtmann17}, \S6.1) to our setting.

For a marked Salvetti~$\sigma = \left(\sal, \alpha\right)$ and a conjugacy class~$g$ of~$\raag$, denote by~$\ell_\sigma(g)$ the minimal length of a word in the free group~$F\langle V(\Gamma)\rangle$ representing an element of the conjugacy class~$\alpha^{-1}(g)$ in~$\raag$. In particular, if~$\alpha = \operatorname{id}$, then~$\ell_\sigma(g)$ is the minimal word length of an element of~$g$.

Identifying elements of~$V(\Gamma)$ with the 1-skeleton of~$\sal$,~$\ell_\sigma(g)$ can be thought of as the minimal length of an edge-path in the 1-skeleton of~$\sal$ representing~$\alpha^{-1}(g)$. As observed in \cite{CharneyStambaughVogtmann17}, normal form for RAAGs (see \cite{CharneyRAAGsSurvey}) implies that this length function is well-defined. Moreover, if~$\alpha$ is an isometry of~$\sal$, then a minimal edge-path in the 1-skeleton of~$\sal$ representing~$\alpha^{-1}(g)$ is the same length as one representing~$g$---which matches the fact that~$\left(\sal, \alpha\right) \sim \left(\sal, \operatorname{id}\right)$.

Let~$\calG = \left(g_1,\; g_2,\;, \dots\right)$ be a list of the conjugacy classes of~$\raag$, and fix~$\calW$, an arbitrary finite subset of~$\calG$. Let~$\calG_0 \subseteq \calG$ be the subset of conjugacy classes represented by a length of at most two. 

\begin{definition}
    For a marked Salvetti~$\sigma = \left(\sal, \alpha\right)$, define the \emph{$\calW$-norm} of~$\sigma$ to be 
    \[\lVert \sigma \rVert' \coloneqq \left(\lVert \sigma \rVert_{\calW},\; \lVert \sigma \rVert_0,\; \lVert \sigma \rVert_1,\; \lVert \sigma \rVert_2,\; \dots\right) \;\; \in \mathbb{Z} \times \mathbb{Z} \times \mathbb{Z}^{\calG},\] 
    where 
    \begin{eqnarray*}
        \lVert \sigma \rVert_{\calW} & \coloneqq & \sum_{g \in \calW} \ell_\sigma(g); \\
        \lVert \sigma \rVert_0 & \coloneqq & \sum_{g \in \calG_0} \ell_\sigma(g); \\
        \lVert \sigma \rVert_i & \coloneqq & \ell_\sigma\left(g_i\right) \quad \text{for~$i \geq 1$},
    \end{eqnarray*}
    and we consider~$\mathbb{Z} \times \mathbb{Z} \times \mathbb{Z}^{\calG}$ as an ordered abelian group with the lexicographical ordering; write~$\mathbf{0} \coloneqq \left(0,\; 0,\; \dots\right)$ for its identity element.
\end{definition}

This follows the notation laid out in (\cite{CharneyStambaughVogtmann17} \S6); our only addition is the dependence on~$\calW$. In particular,~$\lVert \cdot \rVert' = \left(\lVert \cdot \rVert_{\calW},\; \lVert \cdot \rVert\right)$, where~$\lVert \cdot \rVert$ is the norm employed by Charney--Stambaugh--Vogtmann.

\begin{proposition}\label{prop:new norm is a strict total well-ordering}
   ~$\lVert \cdot \rVert'$ is a strict total well-ordering on the set of marked Salvettis.
\end{proposition}

\begin{proof}
    Due to (\cite{CharneyStambaughVogtmann17}, Corollary 6.3),~$\lVert \cdot \rVert$ is a strict total order on the set of marked Salvettis. It follows that~$\lVert \cdot \rVert'$ is also a strict total order, by construction. 
    
    It remains to show that~$\lVert \cdot \rVert'$ is a well-ordering. Suppose, for contradiction, that it is not, so there exists an infinite decreasing chain of marked Salvettis with respect to~$\lVert \cdot \rVert'$. Let~$\sigma$ be some member of this chain; we can then choose another,~$\tau$, such that~$\lVert \tau \rVert' < \lVert \sigma \rVert'$. Then either~$\lVert \tau \rVert_\calW < \lVert \sigma \rVert_\calW$, or~$\lVert \tau \rVert_\calW = \lVert \sigma \rVert_\calW$ and~$\lVert \tau \rVert_0 \leq \lVert \sigma \rVert_0$.

    By (\cite{CharneyStambaughVogtmann17}, Corollary 6.21), we know that for all~$N \geq \lVert \left(\sal, \operatorname{id}\right)\rVert_0$, there exist only finitely many marked Salvettis~$\sigma$ with~$\lVert \sigma \rVert_0 \leq N$. Since there are infinitely many~$\tau$ with~$\lVert \tau \rVert' < \lVert \sigma \rVert'$, we can pick~$\tau$ such that~$\lVert \tau \rVert_\calW < \lVert \sigma \rVert_\calW$. We can now repeat this argument \emph{ad infinitum}, obtaining an infinite sequence of marked Salvettis which are strictly decreasing with respect to~$\lVert \cdot \rVert_\calW$. In other words, we have an infinite strictly decreasing sequence of positive integers, a contradiction.
\end{proof}

In particular, there is a minimal marked Salvetti with respect to the norm~$\lVert \cdot \rVert'$.

As discussed in \S\ref{sec:discussion of proof}, this is one of the places in which we circumvent needing an analogue of (\cite{CharneyStambaughVogtmann17}, Corollary 6.20). 

\subsection{Effect of a Whitehead move on the~$\calW$-norm}\label{subsec:effect of Whitehead move on W-norm}

We now aim to transplant what we need from (\cite{CharneyStambaughVogtmann17}, \S6.2). We refer the reader to the introduction of that subsection for a discussion of the geometric interpretation of~$\ell_\sigma(g)$, and the motivation behind the following notation. 

Fix a marked Salvetti~$\left(\sal, \alpha\right)$. For a conjugacy class~$g \in \calG$, let~$w$ be a cyclically-reduced word for~$\alpha^{-1}(g)$. For a~$\Gamma$-partition~$\calP = \left(P \;\vert\;\overline{P}\;\vert\;\link{\calP}\right)$, let~$\vert \calP \vert_w$ denote the number of cyclically adjacent letters~$u_iu_{i+1}$ in~$w$ such that~$u_i$ and~$u_{i+1}^{-1}$ do not both lie in~$P \cup \link{\calP}$ or both lie in~$\overline{P} \cup \link{\calP}$. For any vertex~$v \in V(\Gamma)$, let~$\vert v \vert_w$ denote the number of occurrences of~$v$ or~$v^{-1}$ in~$w$.

We will use the following lemma from \cite{CharneyStambaughVogtmann17}. Recall the notation before \emph{Corollary \ref{cor:factorisation lemma (i.e. CSV17, 4.19)}}: for a~$\Gamma$-Whitehead automorphism~$\alpha = \varphi\left(\calP, m\right)$, and a marked Salvetti~$\sigma = \left(\sal, \varphi\right)$, we write~$\sigma_m^\calP$ for the Whitehead move at~$\sigma$ associated to~$\left(\calP, m\right)$; that is,~$\sigma_m^\calP = \left(\sal, \varphi \circ \alpha\right)$.

\begin{lemma}[\cite{CharneyStambaughVogtmann17}, Lemma 6.4]\label{lem:CSV17 Lem 6.4}
    Let~$\sigma = \left(\sal, \alpha\right)$ be a marked Salvetti, let~$\varphi = \left(\calP, m\right)$ be a~$\Gamma$-Whitehead automorphism, let~$g$ be a conjugacy class in~$\raag$, and let~$w$ be a minimal length word representing~$\alpha^{-1}(g)$. Then 
    \[\ell_{\sigma_m^\calP}(g) = \ell_\sigma(g) + \vert \calP \vert_w - \vert v \vert_w.\]

    More generally, if~$\sigma'$ is obtained from~$\sigma$ by blowing up a compatible set~$\Pi = \{\calP_1,\;\dots,\;\calP_k\}$ of~$\Gamma$-partitions and collapsing a treelike set~$\calH = \{H_1,\;\dots,\;H_k\}$ of hyperplanes dual to edges labelled~$e_{v_i}$, then 
    \[\ell_{\sigma'}(g) = \ell_\sigma(g) + \sum_{i=1}^k\vert\calP_i\vert_w - \sum_{i=1}^k\vert v_i\vert_{w}.\]
\end{lemma}

We now extend this notation as follows. For a marked Salvetti~$\sigma = \left(\sal, \alpha\right)$ and a~$\Gamma$-Whitehead automorphism~$\varphi = \left(\calP, v\right)$, define 
\[\vert \calP \vert'_\sigma \coloneqq \left( \vert\calP\vert_\calW,\; \vert\calP\vert_0,\; \vert\calP\vert_{w(g_1)},\; \vert\calP\vert_{w(g_2)},\;\dots\right),\] 
where
\begin{itemize}
    \item~$w(g)$ is a minimal length word representing~$\alpha^{-1}(g)$, for any conjugacy class~$g$ of~$\raag$;
    \item~$\vert\calP\vert_0 \coloneqq \sum_{g \in \calG_0} \vert \calP \vert_{w(g)}$;
    \item~$\vert\calP\vert_\calW \coloneqq \sum_{g \in \calW} \vert \calP \vert_{w(g)}$.
\end{itemize}

Similarly define~$\vert v \vert'_\sigma \coloneqq \left( \vert v \vert_\calW,\; \vert v \vert_0,\; \vert v \vert_{w(g_1)},\; \vert v \vert_{w(g_2)},\;\dots\right)$. We can now restate \emph{Lemma \ref{lem:CSV17 Lem 6.4}} as follows:

\begin{corollary}\label{cor:analogue of [CSV17], Cor 6.6}
    Let~$\Pi$ and~$\calH$ be as in \emph{Lemma \ref{lem:CSV17 Lem 6.4}}. Then 
    \[\lVert \sigma_\calH^\Pi \rVert' = \lVert \sigma \rVert' + \sum_{i=1}^k \vert \calP_i \vert'_\sigma - \sum_{i=1}^k \vert v_i \vert_\sigma.\]
\end{corollary}

This is a direct translation of (\cite{CharneyStambaughVogtmann17}, Corollary 6.6) to our new norms.

\begin{definition}\label{def:partition reductive for sigma in new norm}
    Let~$\sigma$ be a marked Salvetti and let~$\calP$ be a~$\Gamma$-partition. We say that~$\calP$ is \emph{$\calW$-reductive} for~$\sigma$ if for some~$v \in \mx{\calP}$ the~$\Gamma$-Whitehead automorphism~$\varphi = \varphi = \left(\calP, v\right)$ reduces~$\lVert \sigma \rVert'$, that is,~$\lVert \sigma_v^\calP\rVert' < \lVert \sigma \rVert'$ (or equivalently,~$\vert \calP \vert'_\sigma < \vert v \vert'_\sigma$).
\end{definition}

\begin{corollary}[cf. \cite{CharneyStambaughVogtmann17}, Corollary 6.8]\label{cor:analogue of [CSV17], Cor 6.8}
    Let~$\Pi$ and~$\calH$ be as in \emph{Lemma \ref{lem:CSV17 Lem 6.4}}. If~$\lVert \sigma_\calH^\Pi \rVert' < \lVert \sigma \rVert'$, then some~$\calP_i \in \Pi$ is~$\calW$-reductive for~$\sigma$.
\end{corollary}

\begin{proof}
    We follow the proof of (\cite{CharneyStambaughVogtmann17}, Corollary 6.8). By \emph{Theorem \ref{thm:CSV17, 4.12}}, we can order the elements of~$\calH$ so that if~$e_{v_i}$ is the edge dual to~$H_i$, then~$\left(\calP_i, v_i\right)$ is a~$\Gamma$-Whitehead pair. If~$\lVert \sigma_\calH^\Pi\rVert' < \lVert \sigma \rVert'$, then by \emph{Corollary \ref{cor:analogue of [CSV17], Cor 6.6}}, we have 
    \[\sum_{i=1}^k \vert \calP_i \vert'_\sigma - \sum_{i=1}^k \vert v_i \vert_\sigma < \mathbf{0}.\] 
    Hence we must have some~$i$ with~$\vert\calP_i\vert'_\sigma < \vert v_i\vert'_\sigma$. Hence~$\lVert \sigma_{v_i}^{\calP_i} \rVert' < \lVert \sigma \rVert'$, and we're done.
\end{proof}

\subsection{Proof of contractibility of~$\Kmin{\calW}$}\label{subsec:proof of contractibility of Kmin}

We have now established analogues of Lemma 6.2, Corollary 6.8, and Proposition 6.22 from \cite{CharneyStambaughVogtmann17}. We are skipping Corollary 6.20, and we have the Poset Lemma for free, so (as can be seen from \emph{Figure \ref{fig:dependence diagram in CSV17 Ch. 6}}), all that remains is to obtain an analogue of the `Pushing Lemma', (\cite{CharneyStambaughVogtmann17}, Lemma 6.23).

We will obtain an analogue with minimal work, since Charney--Stambaugh--Vogtmann's argument does not explicitly rely on their norm. However, the statement and proof do use notation and terminology which we have adapted to our new context, so we will give an abbreviation of the proof given in \cite{CharneyStambaughVogtmann17}, with our adapted notation where necessary. 

\emph{Star graphs} have been used to study automorphisms of free groups (see e.g. \cite{HoareCoinitialGraphs79}). Charney--Stambaugh--Vogtmann develop this theory to work for RAAGs (see \cite{CharneyStambaughVogtmann17}, \S6.3). We will not reproduce this theory, but we will need one consequence, which we rephrase slightly (in \emph{Corollary \ref{cor:Lemma 6.15 + Lemma 6.16}}).

Let~$\calP$ and~$\calQ$ be~$\Gamma$-partitions. We call the four side-side intersections~$P \cap Q$,~$P \cap \overline{Q}$,~$\overline{P} \cap Q$,~$\overline{P} \cap \overline{Q}$ \emph{quadrants} (or possibly \emph{$\left(\calP, \calQ\right)$-quadrants}). Two quadrants are \emph{opposite} if one can be obtained from the other by switching sides of both~$\calP$ and~$\calQ$. 

It will now be convenient to allow (as we previously didn't)~$\Gamma$-partitions to have one side a singleton. That is,~$\calP = \left(P \;\vert\; \overline{P} \;\vert\; \link{\calP}\right)$ with~$P = \{v\}$ for some~$v \in V^\pm$. We call such a~$\Gamma$-partition \emph{degenerate}. The associated~$\Gamma$-Whitehead automorphism is an inversion, and for every marked Salvetti~$\sigma$, and any word~$w \in F(V)$ we have~$\vert \calP \vert_w = \vert v \vert_w$---so in particular, a degenerate~$\Gamma$-partition can never be~$\calW$-reductive.

\begin{lemma}[\cite{CharneyStambaughVogtmann17}, Lemma 6.16]\label{lem:[CSV17], Lemma 6.16}
    For any incompatible~$\Gamma$-partitions~$\calP$ and~$\calQ$, with maximal elements~$v$ and~$w$ respectively, there is a pair of opposite quadrants such that each defines a (possibly degenerate)~$\Gamma$-partition with maximal element in~$\{v^\pm, w^\pm\}$.
\end{lemma}

The proof of this lemma is entirely combinatorial, and does not depend on any choice of norm. 

\begin{lemma}[\cite{CharneyStambaughVogtmann17}, Lemma 6.15, special case]\label{lem:[CSV17], Lemma 6.15 special case}
    Let~$\calP$,~$\calQ$ be incompatible~$\Gamma$-partitions. Suppose that~$P \cap \overline{Q}$ and~$\overline{P} \cap Q$ are the two opposite quadrants generated by \emph{Lemma \ref{lem:[CSV17], Lemma 6.16}}, and let~$\calX$ and~$\calY$ be the corresponding~$\Gamma$-partitions. Then for any minimal length word~$w$ representing a conjugacy class~$g$ of~$\raag$, we have 
    \[\vert \calX \vert_w + \vert \calY \vert_w \leq \vert \calP \vert_w + \vert \calQ \vert_w.\] 
\end{lemma}

\begin{corollary}\label{cor:Lemma 6.15 + Lemma 6.16}
    With the same setup as in \emph{Lemma \ref{lem:[CSV17], Lemma 6.15 special case}}, we have 
    \[\vert \calX \vert'_\sigma + \vert \calY \vert'_\sigma \leq \vert \calP \vert'_\sigma + \vert \calQ \vert'_\sigma.\]
\end{corollary}

We are now ready for our version of the `Pushing Lemma', (\cite{CharneyStambaughVogtmann17}, Lemma 6.23). 

\begin{lemma}[Pushing Lemma]\label{lem:pushing lemma (for us)}
    Fix a marked Salvetti~$\sigma$. Suppose that~$\left(M, m\right)$ is a~$\Gamma$-Whitehead pair which is~$\calW$-reductive for~$\sigma$ such that
    \begin{enumerate}[(i)]
        \item~$\link{M}$ is maximal among links of~$\Gamma$-partitions which are~$\calW$-reductive for~$\sigma$;
        \item for any~$\Gamma$-Whitehead pair~$\left(Q, w\right)$ with~$\link{\calQ} = \link{\calM}$, we have ~$\vert \calM \vert'_\sigma - \vert m \vert'_\sigma < \vert \calQ \vert'_\sigma - \vert w \vert'_\sigma$.
    \end{enumerate}
    Let~$\calP$ be a~$\Gamma$-partition which is~$\calW$-reductive for~$\sigma$ and which is not compatible with~$\calM$. Then at least one of the~$\left(\calP, \calM\right)$-quadrants involving the side of~$\calM$ containing~$m^{-1}$ determines a~$\Gamma$-partition which has link equal to~$\link{\calP}$ and is~$\calW$-reductive for~$\sigma$.
\end{lemma}

\begin{proof}
    The proof of Lemma 6.23 in \cite{CharneyStambaughVogtmann17} transfers exactly over to our case, using \emph{Corollary \ref{cor:Lemma 6.15 + Lemma 6.16}}. We give a short summary. 

   ~$\calP$ is~$\calW$-reductive for~$\sigma$, so~$\vert \calX \vert'_\sigma - \vert v \vert'_\sigma < \mathbf{0}$ for some~$v \in \max{\calP}$. 
    
    First suppose that~$v \in P$,~$v^{-1} \in \overline{P}$,~$m \in M$, and~$m^{-1} \in \overline{M}$. By Case (3) in the proof of \emph{Lemma \ref{lem:[CSV17], Lemma 6.16}} (\cite{CharneyStambaughVogtmann17}, Lemma 6.16), we have opposite~$\left(\calP, \calM\right)$-quadrants~$X$,~$Y$ such that~$\left(X, v\right)$ and~$\left(X, v^{-1}\right)$ are~$\Gamma$-Whitehead pairs, and opposite quadrants~$X'$,~$Y'$ such that~$\left(X', m\right)$ and~$\left(Y', m^{-1}\right)$. Applying \emph{Corollary \ref{cor:Lemma 6.15 + Lemma 6.16}} to both pairs of quadrants, recalling that~$\calP$ is~$\calW$-reductive for~$\sigma$, and using hypothesis (ii), one can deduce that 
    \[\left(\vert \calY \vert'_\sigma - \vert m \vert'_\sigma\right) + \left(\vert \calY' \vert'_\sigma - \vert m \vert'_\sigma\right) < \mathbf{0},\] 
    and so that one of the pairs~$\left(Y, m\right)$,~$\left(Y', m^{-1}\right)$ is~$\calW$-reductive for~$\sigma$ (and has link equal to~$\link{\calP}$), as required.

    In any other case, a similar argument holds. Apply \emph{Lemma \ref{lem:[CSV17], Lemma 6.16}} to find a pair of opposite~$\left(\calP, \calM\right)$-quadrants~$X$,~$Y$, and then using hypotheses (i) \& (ii) along with \emph{Corollary \ref{cor:Lemma 6.15 + Lemma 6.16}}, one can conclude that one or the other of these is reductive and has link equal to~$\link{\calP}$, as required. 
\end{proof}

Finally, we need the \emph{Poset Lemma}, a special case of Quillen's \emph{Theorem A} from \cite{QuillenHigherAlgebraicKTheory73}.

\begin{lemma}[Poset Lemma]\label{lem:poset lemma}
    Let~$P$ be a poset of finite height and let~$f \colon P \to P$ be a poset map such that~$f(p) \leq p$ for all~$p \in P$. Then~$f$ induces a deformation retraction from the geometric realisation of~$P$ to the geometric realisation of the image~$f(P)$.
\end{lemma}

We are now ready to prove our main result of this section, which was stated as \emph{Theorem \ref{customthm:Kmin is contractible}} in \S\ref{subsec:discussion of proof strategy}.

\begin{theorem}\label{thm:Kmin is contractible}
    For any~$\Gamma$ and any~$\calW$,~$\Kmin{\calW}$ is contractible.
\end{theorem}

\begin{proof}
    Since~$\lVert \cdot \rVert'$ is a strict total well-order (by \emph{Proposition \ref{prop:new norm is a strict total well-ordering}}), there is a unique minimal marked Salvetti with respect to~$\lVert \cdot \rVert'$. We start with the star of this marked Salvetti, and then inductively add on stars of marked Salvettis in increasing order. Since the star of a marked Salvetti is always contractible, our main job is to show that the intersection along which we glue is contractible, if it is nonempty. 

    To this end, for any marked Salvetti~$\sigma$, let 
    \[K_{<\sigma} \coloneqq \bigcup_{\lVert \tau \rVert' < \lVert \sigma \rVert'} \str{\tau}.\] 
    We prove that if~$\str{\sigma} \cap K_{<\sigma}$ is nonempty, then it is contractible. This intersection consists of blowups~$\sigma^\Pi$ which can be collapsed to a marked Salvetti of~$\calW$-norm smaller than~$\lVert \sigma \rVert'$, where~$\Pi$ is an \emph{ideal forest}---that is, a compatible set of~$\Gamma$-partitions. We can identify~$\str{\sigma} \cap K_{<\sigma}$ as the geometric realisation of the poset of ideal forests ordered by inclusion. We will apply the Poset Lemma to contract this poset to a single point. 

    By \emph{Corollary \ref{cor:analogue of [CSV17], Cor 6.8}}, if~$\Pi$ is in~$\str{\sigma} \cap K_{<\sigma}$, then~$\Pi$ contains a~$\Gamma$-partition~$\calP$ which is~$\calW$-reductive for~$\sigma$. Hence the map which throws out from~$\Pi$ those~$\Gamma$-partitions which are not~$\calW$-reductive for~$\sigma$ is a poset map, to which we can apply the Poset Lemma. Choose a pair~$\left(M, m\right)$ which satisfies the maximality conditions in the Pushing Lemma.

    At this point the rest of the proof as in (\cite{CharneyStambaughVogtmann17}, Theorem 6.24) follows through with no modification. This allows us to conclude that~$\str{\sigma} \cap K_{<\sigma}$ is contractible, if it is nonempty. Hence we can build the entirety of~$\spine$ as a union of contractible components. Since we already know that~$\spine$ is connected, we conclude that it is in fact contractible. 

    Moreover, we also build~$\Kmin{\calW}$ in this fashion, simply by stopping after incorporating all those marked Salvettis with minimal~$\lVert \cdot \rVert_\calW$-component of the~$\calW$-norm. Hence~$\Kmin{\calW}$ is also a union of contractible components. If~$\Kmin{\calW}$ is disconnected, then since~$\spine$ is connected, there must be some marked Salvetti~$\sigma \notin \Kmin{\calW}$ such that~$\str{\sigma} \cap K_{<\sigma}$ consists of two different components, which contradicts its contractibility, as we just proved. Hence~$\Kmin{\calW}$ had better be connected, and thus is contractible, as required. 
\end{proof}

\subsection{A finiteness result for outer automorphisms with respect to a set of conjugacy classes}

We pause the proof of contractibility of the symmetric spine to give a separate application of \emph{Theorem \ref{thm:Kmin is contractible}}. For any finite set~$\calW$ of conjugacy classes of~$\raag$, let~$\outw$ be the subgroup consisting of outer automorphisms that permute the elements of~$\calW$.~$\Kmin{\calW}$ is invariant under the action of this group. We can use this to prove that~$\outw$ is type \emph{VF}, mirroring the proof of (\cite{CullerVogtmann86}, Corollary 6.1.4), which is the analogous result in the free group setting. Recall (\emph{Definition \ref{def:type F}}) that a discrete group~$G$ is \emph{type F} if there exists a~$K(G, 1)$ which is a finite CW-complex, and a group is \emph{type VF} if it is virtually type \emph{F}.

\begin{corollary}\label{cor:outw is type VF}
    Let~$\calW$ be a finite set of conjugacy classes of~$\raag$. Then~$\outw$ is type \emph{VF}.
\end{corollary}

\begin{proof}
    The stabiliser of each cell under the action of~$\outw$ consists of graph symmetries and inversions, so we can pass to a torsion-free finite-index normal subgroup~$H \leq \outw$ which acts freely on~$\Kmin{\calW}$. By \emph{Lemma \ref{lem:sufficient condition for type F}}, it then suffices to prove that~$\faktor{\Kmin{\calW}}{\outw}$ is finite, as that implies that~$\faktor{\Kmin{\calW}}{H}$ is.

    Fix a marked Salvetti~$\sigma = (\sal, \alpha)$. For any~$k \in \mathbb{Z}_{\geq 1}$, there are only finitely many conjugacy classes~$g$ of~$\raag$ such that~$\ell_\sigma(g) \leq k$. Hence there can only be finitely many finite sets~$\calW'$ of conjugacy classes of~$\raag$ such that~$\lVert \sigma \rVert_\calW = \lVert \sigma \rVert_{\calW'}$.

    Choose a marked Salvetti~$\sigma_0$ in~$\Kmin{\calW}$. Let~$\calW = \calW_0, \calW_1, \dots, \calW_m$ be all the finite sets~$\calW'$ of conjugacy classes such that~$\lVert \sigma_0 \rVert_\calW = \lVert \sigma_0 \rVert_{\calW'}$ which are images of~$\calW$ under an outer automorphism of~$\raag$. Choose elements~$\operatorname{id} = \varphi_0, \varphi_1, \dots, \varphi_m \in \out$ such that~$\varphi_i(\calW_i) = \calW$. 

    For any marked Salvetti~$\sigma$ and any~$\varphi \in \out$, we have 
    \begin{eqnarray*}
        \lVert \varphi \cdot \sigma \rVert_\calW & = & \sum_{g \in \calW} \ell_{\varphi \cdot \sigma}(g) \\
        & = & \sum_{g \in \calW} \ell_\sigma(\varphi^{-1}(g)) \\
        & = & \sum_{h \in \varphi^{-1}(\calW)} \ell_\sigma(h) \\
        & = & \lVert \sigma \rVert_{\varphi^{-1}(\calW)}.
    \end{eqnarray*}
    Hence for each~$i$, we have 
    \[\lVert \varphi_i \cdot \sigma_0 \rVert_\calW = \lVert \sigma_0 \rVert_{\varphi_i^{-1}(\calW)} = \lVert \sigma_0 \rVert_{\calW_i} = \lVert \sigma_0 \rVert_\calW,\]
    so~$\varphi_i \cdot \sigma_0$ is a marked Salvetti in~$\Kmin{\calW}$.

    Let~$\sigma$ be any marked Salvetti in~$\Kmin{\calW}$, and choose an automorphism~$\varphi \in \out$ such that~$\varphi \cdot \sigma_0 = \sigma$. Then we have 
    \[\lVert \sigma_0 \rVert_\calW = \lVert \sigma \rVert_\calW = \lVert \varphi \cdot \sigma_0 \rVert_\calW = \lVert \sigma_0 \rVert_{\varphi^{-1}(\calW)},\]
    and so~$\varphi^{-1}(\calW) = \calW_i$ for some~$i$. Thus~$\calW = (\varphi_i \circ \varphi^{-1})(\calW)$. Now,~$(\varphi_i \circ \varphi^{-1}) \cdot \sigma = \varphi_i \cdot \sigma_0$, which means that~$\sigma$ is equivalent modulo~$\outw$ to one of a finite list of marked Salvettis. Hence the quotient~$\faktor{\Kmin{\calW}}{\outw}$ is finite, and we're done.
\end{proof}

\section{Retraction of~$\Kmin{\calW}$ to~$\symspine$}\label{sec:retraction of Kmin to symspine}

In this subsection, we fix~$\calW$ to be the set of conjugacy classes of~$\raag$ of length one and do Step (iii) as outlined in \S\ref{subsec:discussion of proof strategy}. The content in this subsection is a transparent generalisation of (\cite{CollinsSymSpine89}, \S4).

\subsection{Companion hyperplanes and edges}\label{subsec:companion hyperplanes and edges}

Let~$X$ be a~$\Gamma$-complex, and let~$\calT$ be a treelike set of hyperplanes in~$X$. Let~$H \in \calT$, with dual edges~$e_H$. Recall that~$H$ partitions the vertices of~$X$ into two sets~$\beta(H) \sqcup \overline{\beta}(H)$, and all edges labelled by~$H$ have one endpoint in~$\beta(H)$ and the other in~$\overline{\beta}(H)$. We call the edges traversing between~$\beta(H)$ and~$\overline{\beta}(H)$ \emph{companion edges} to~$e_H$, and their dual hyperplanes \emph{companion hyperplanes} to~$H$. 

\begin{definition}\label{def:symmetric relative to treelike set}
    We say that~$e_H$ or~$H$ is \emph{symmetric (relative to~$\calT$)} if~$H$ has a unique companion hyperplane (so companion edges to~$e_H$ are labelled by~$H$ or by this unique companion hyperplane). 
\end{definition}

\emph{Figure \ref{fig:illustration of companion edges and hyperplanes}} gives a schematic of this decomposition of the vertices of~$X$ into~$\beta(H)$ and~$\overline{\beta}(H)$ and an illustration of companion edges and hyperplanes.

\begin{figure}[H]
    \centering
    \begin{tikzpicture}
        %
        \draw (-6, 0) circle (2);
        \draw (6, 0) circle (2);
        \node at (-6, 0) [black, thick]{$\beta(H)$};
        \node at (6, 0) [black, thick]{$\overline{\beta}(H)$};
        %
        %
        \draw[black, thick, fill = black!20!white] (-2, 2.4) to (-1.6, 2.7) to (-1.6, 1.7) to (-2, 1.4) to (-2, 2.4);
        \draw[black, thick, fill = black!20!white] (-1, 1.4) to (-0.6, 1.7) to (-0.6, 0.7) to (-1, 0.4) to (-1, 1.4);
        \draw[black, thick, fill = black!20!white] (-1.5, -2.7) to (-1.1, -2.4) to (-1.1, -1.4) to (-1.5, -1.7) to (-1.5, -2.7);
        %
        %
        \draw[smooth, tension=0.5, black, thick] plot coordinates {(-4.8, 1.6) (0, 2.3) (4.8, 1.6)};
        \draw[smooth, tension=0.5, black, thick] plot coordinates {(-4.7, 1.52) (0, 2.2) (4.7, 1.52)};
        \draw[smooth, tension=0.5, black, thick] plot coordinates {(-4.57, 1.4) (0, 2) (4.57, 1.4)};
        \fill (0, 2.05) circle (0.3pt);
        \fill (0, 2.1) circle (0.3pt);
        \fill (0, 2.15) circle (0.3pt);
        \draw[smooth, tension=0.5, black, thick, ->-] plot coordinates {(-4.17, 0.8) (0, 1.3) (4.17, 0.8)};
        \draw[smooth, tension=0.5, black, thick, ->-] plot coordinates {(-4.13, 0.7) (0, 1.2) (4.13, 0.7)};
        \draw[smooth, tension=0.5, black, thick, ->-] plot coordinates {(-4.08, 0.55) (0, 0.95) (4.08, 0.55)};
        \fill (0, 1) circle (0.3pt);
        \fill (0, 1.075) circle (0.3pt);
        \fill (0, 1.15) circle (0.3pt);
        \draw[smooth, tension=0.5, black, thick, ->-] plot coordinates {(-4.8, -1.6) (0, -2.3) (4.8, -1.6)};
        \draw[smooth, tension=0.5, black, thick, ->-] plot coordinates {(-4.7, -1.52) (0, -2.2) (4.7, -1.52)};
        \draw[smooth, tension=0.5, black, thick, ->-] plot coordinates {(-4.57, -1.4) (0, -2) (4.57, -1.4)};
        \fill (0, -2.05) circle (0.3pt);
        \fill (0, -2.1) circle (0.3pt);
        \fill (0, -2.15) circle (0.3pt);
        %
        %
        \fill (0, -0.2) circle (0.6pt);
        \fill (0, -0.5) circle (0.6pt);
        \fill (0, -0.8) circle (0.6pt);
        %
        %
        \node at (-0.2, 2.45) [black, thick]{$e_H$};
        \node at (-1.9, 2.75) [black, thick]{$H$};
        \node at (0.6, 0.6) [black, thick]{$e_{H_1}$};
        \node at (-0.5, 0.4) [black, thick]{$H_1$};
        \node at (0.5, -1.7) [black, thick]{$e_{H_m}$};
        \node at (-1, -2.75) [black, thick]{$H_m$};
    \end{tikzpicture}
    \caption{A schematic of the decomposition of the vertices of a~$\Gamma$-complex~$X$ into two sets~$\beta(H)$ and~$\overline{\beta}(H)$, one either side of a hyperplane~$H$ in a treelike set~$\calT$. Here~$H_1,\; \dots,\; H_m$ are the companion hyperplanes to~$H$, while the edges labelled~$e_{H_i}$ are the companion edges to those labelled~$e_H$. Note that~$m \geq 1$ since~$\Gamma$-complexes have no separating hyperplanes.~$H$ is symmetric relative to~$\calT$ if and only if~$m = 1$.}
    \label{fig:illustration of companion edges and hyperplanes}
\end{figure}
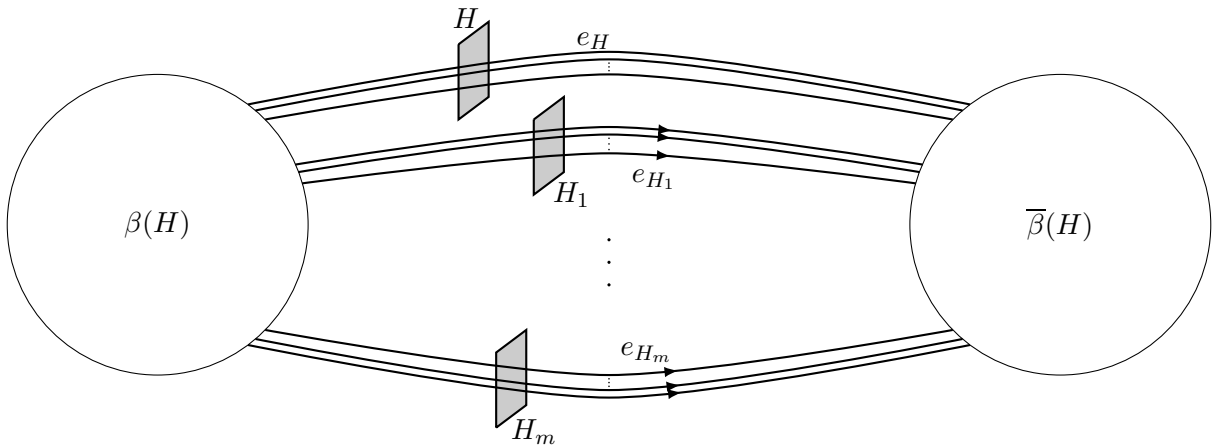

\begin{lemma}\label{lem:no two elements of treelike set are companions}
     Let~$\calT$ be a treelike set of hyperplanes in a~$\Gamma$-complex~$X$. Let~$H'$ be a companion hyperplane to~$H \in \calT$. Then~$H' \notin \calT$.
\end{lemma}

\begin{proof}
    Suppose that~$H' \in \calT$. Let~$\overline{\calT} \coloneqq \calT \setminus \{H, H'\}$, and perform the hyperplane collapse
    \[c_{\overline{\calT}} \colon X \to X \ssslash \overline{\calT} \cong \salb{\{\calP_H, \calP_{H'}\}}.\]
    Note that the image of~$H'$ (which we shall still denote~$H'$) will be a companion hyperplane to the image of~$H$ under this hyperplane collapse. We now explicitly examine the structure of this blowup. 

    \textbf{Case I:}~$\calP_H$ and~$\calP_{H'}$ are adjacent.

    In this case,~$\salb{\{\calP_H, \calP_{H'}\}}$ has four vertices, each labelled by a choice of sides~$\{P_H^\times, P_{H'}^\times\}$, and we have the subgraph in the 1-skeleton of~$\salb{\{\calP_H, \calP_{H'}\}}$ shown in \emph{Figure \ref{fig:portion of 1-skeleton of double blowup}}. 
    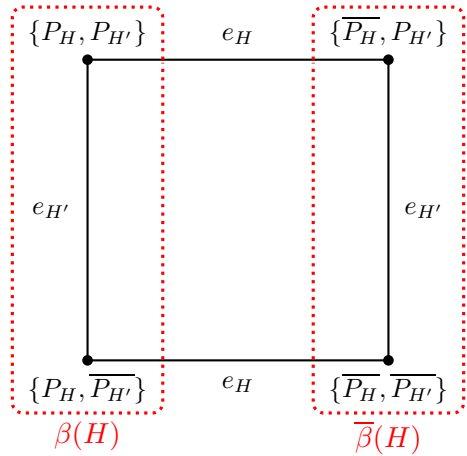
\begin{figure}
        \centering
        \begin{tikzpicture}
            \fill (-2, 2) circle (2pt);
            \fill (2, 2) circle (2pt);
            \fill (-2, -2) circle (2pt);
            \fill (2, -2) circle (2pt);
            \draw [black, thick] (-2, 2) -- (2, 2);
            \draw [black, thick] (-2, -2) -- (2, -2);
            \draw [black, thick] (-2, 2) -- (-2, -2);
            \draw [black, thick] (2, 2) -- (2, -2);
            \node at (-2, 2) [above = 1pt, black, thick]{\small$\{P_H, P_{H'}\}$};
            \node at (2, 2) [above = 1pt, black, thick]{\small$\{\overline{P_H}, P_{H'}\}$};
            \node at (-2, -2) [below = 1pt, black, thick]{\small$\{P_H, \overline{P_{H'}}\}$};
            \node at (2, -2) [below = 1pt, black, thick]{\small$\{\overline{P_H}, \overline{P_{H'}}\}$};
            \node at (0, 2) [above = 2pt, black, thick]{\small$e_H$};
            \node at (0, -2) [below = 2pt, black, thick]{\small$e_H$};
            \node at (-2, 0) [left = 2pt, black, thick]{\small$e_{H'}$};
            \node at (2, 0) [right = 2pt, black, thick]{\small$e_{H'}$};
            \draw[red, dotted, very thick, rounded corners] (-1, 2) -- (-1, 2.7) -- (-3, 2.7) -- (-3, -2.7) -- (-1, -2.7) -- (-1, 1.96);
            \draw[red, dotted, very thick, rounded corners] (1, 2) -- (1, 2.7) -- (3, 2.7) -- (3, -2.7) -- (1, -2.7) -- (1, 1.96);
            \node at (-2, -2.6) [below = 2pt, red, thick]{\textbf{$\beta(H)$}};
            \node at (2, -2.6) [below = 2pt, red, thick]{\textbf{$\overline{\beta}(H)$}};
        \end{tikzpicture}
        \caption{A portion of the 1-skeleton of~$\salb{\{\calP_H, \calP_{H'}\}}$, as considered in the proof of \emph{Lemma \ref{lem:no two elements of treelike set are companions}}.}
        \label{fig:portion of 1-skeleton of double blowup}
    \end{figure}

    Here it is evident that the edges labelled by~$H'$ do not have one endpoint in~$\beta(H)$ and the other in~$\overline{\beta}(H)$---so~$H'$ is not a companion to~$H$, a contradiction. 

    \textbf{Case II:}~$\calP_H$ and~$\calP_{H'}$ are not adjacent.

    Since~$\calP_H$ and~$\calP_{H'}$ are compatible, there is some side of~$\calP_H$ which has empty intersection with some side of~$\calP_{H'}$. Without loss of generality, suppose that~$P_H \cap \overline{P_{H'}} = \emptyset$. Then~$\salb{\{\calP_H, \calP_{H'}\}}$ has three vertices, labelled by the three other pairs of sides of~$\calP_H$ and~$\calP_{H'}$; that is, by~$\{P_H, P_{H'}\}$,~$\{\overline{P_H}, P_{H'}\}$, and~$\{\overline{P_H}, \overline{P_{H'}}\}$. We have~$\beta(H) = \{P_H, P_{H'}\}$, and no edges incident to this vertex are labelled by~$H'$, so once again,~$H$ and~$H'$ cannot be companions. 
\end{proof}

\begin{corollary}\label{cor:analogue of [Col89] Lem 4.2}
    A~$\Gamma$-complex is symmetric if and only if for every treelike set~$\calT$ of hyperplanes of~$X$ and every~$H \in \calT$,~$H$ is symmetric relative to~$\calT$.
\end{corollary}

\begin{proof}
    This is essentially a restatement of the definition of a symmetric~$\Gamma$-partition. By \emph{Lemma \ref{lem:no two elements of treelike set are companions}}, if~$H \in \calT$ is symmetric relative to~$\calT$ with unique companion hyperplane~$H'$, then we cannot have~$H' \in \calT$.
\end{proof}

\begin{lemma}\label{lem:analogue of [Col89] Lem 4.3}
    Let~$X$ be a~$\Gamma$-complex and~$\calT$ be a treelike set of hyperplanes of~$X$.
    \begin{enumerate}[(i)]
        \item Let~$H \in \calT$ be symmetric relative to~$\calT$ with unique companion hyperplane~$H'$. Define a new treelike set~$\widehat{\calT} \coloneqq \left( \calT \setminus H\right) \cup H'$. Then for any~$K \in \calT \setminus H$,~$K$ is symmetric relative to~$\calT$ if and only if~$K$ is symmetric relative to~$\widehat{\calT}$. 
        \item Let~$\calT'$ be another treelike set of hyperplanes of~$X$. Every hyperplane (or edge dual to one) in~$\calT$ is symmetric relative to~$\calT$ if and only if every hyperplane (or edge dual to one) in~$\calT'$ is symmetric relative to~$\calT'$.
    \end{enumerate}
\end{lemma}

\begin{proof}
    \begin{enumerate}[(i)]
        \item Let~$K \in \calT \setminus H$ (so in particular~$K \in \widehat{\calT}$) be symmetric relative to~$\calT$ with unique companion hyperplane~$K'$. By \emph{Lemma \ref{lem:no two elements of treelike set are companions}}, we have~$K' \notin \calT$. Similarly,~$H' \notin \calT$.

        If~$K' \neq H'$, then exchanging~$H$ and~$H'$ does not alter the companion edges of~$K$, in which case~$K$ is also symmetric relative to~$\widehat{\calT}$.

        Hence, we are reduced to the case~$K' = H'$; write~$K' = H' = H_v$, for some~$v \in V(\Gamma)$. Consider the hyperplane collapse~$c_{\overline{\calT}} \colon X \to X \ssslash \overline{\calT} \cong \salb{\{\calP_H, \calP_K\}}$, where~$\overline{\calT} \coloneqq \calT \setminus \{H, K\}$. By \emph{Lemma \ref{lem:no two elements of treelike set are companions}}, this hyperplane collapse does not collapse any of the companion edges to either~$H$ or~$K$, so if the image of~$K$ under this collapse (which we still denote~$K$) is symmetric with respect to~$\{H', K\}$, then~$K$ is symmetric with respect to~$\widehat{\calT}$.

        Since~$H$ and~$K$ each have~$H_v$ as a companion hyperplane, the partitions~$\calP_H$ and~$\calP_K$ must each split~$v$ (and only~$v$, since~$H$ and~$K$ are both symmetric with respect to~$\{H, K\}$). Hence we have~$\link{\calP_H} = \link{\calP_K} = \link{v}^\pm$; denote~$\link{v} = L$. Thus (cf. \cite{CharneyStambaughVogtmann17}, Example 4.8)~$\salb{\{\calP_H, \calP_K\}}$ contains a subcomplex isomorphic to~$\Theta_L \times \mathbb{S}_L$, where~$\mathbb{S}_L$ is the Salvetti complex for~$A_L$ and~$\Theta_L$ is the subgraph of the 1-skeleton of~$\salb{\{\calP_H, \calP_K\}}$ formed of the edges labelled by~$H$,~$K$, and all vertices~$w \in V(\Gamma)$ with~$\link{w} = L$. This subcomplex is precisely the union of the carriers of the hyperplanes of~$\salb{\{\calP_H, \calP_K\}}$ labelled by~$H$,~$K$, or~$w \in V(\Gamma)$ with~$\link{w} = L$. 

        Now,~$H$ and~$K$ are not adjacent, so~$\salb{\{\calP_H, \calP_K\}}$ has three vertices. Without loss of generality, suppose that~$P_H \subsetneq P_K$, so that these three vertices are labelled by~$\{P_H, P_K\}$,~$\{\overline{P_H}, P_K\}$, and~$\{\overline{P_H}, \overline{P_K}\}$, joined by edges labelled by~$H$ and~$K$ as shown in \emph{Figure \ref{fig:edges of Theta_L labelled by H and K}}:
        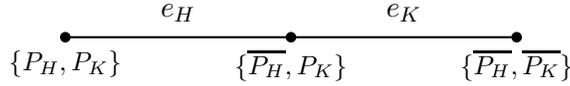
\begin{figure}[H]
            \centering
            \begin{tikzpicture}
                \fill (-3, 0) circle (2pt);
                \fill (0, 0) circle (2pt);
                \fill (3, 0) circle (2pt);
                \draw[black, thick] (-3, 0) -- (0, 0);
                \draw[black, thick] (0, 0) -- (3, 0);
                \node at (-3, 0) [below = 1pt, black, thick]{\small$\{P_H, P_K\}$};
                \node at (0, 0) [below = 1pt, black, thick]{\small$\{\overline{P_H}, P_K\}$};
                \node at (3, 0) [below = 1pt, black, thick]{\small$\{\overline{P_H}, \overline{P_K}\}$};
                \node at (-1.5, 0) [above = 2pt, black, thick]{$e_H$};
                \node at (1.5, 0) [above = 2pt, black, thick]{$e_K$};
            \end{tikzpicture}
            \caption{The edges of~$\Theta_L$ labelled by~$H$ and~$K$.}
            \label{fig:edges of Theta_L labelled by H and K} 
        \end{figure}

        Precisely one of~$v$ and~$v^{-1}$ is an element of~$P_H$; without loss of generality suppose~$v \in P_H$. Then we must have~$v^{-1} \in \overline{P_H} \cap \overline{P_K}$. Hence, the only edge labelled by~$v$ traverses from the vertex labelled~$\{P_H, P_K\}$ to the vertex labelled~$\{\overline{P_H}, \overline{P_K}\}$. Since~$H$ and~$K$ are both symmetric relative to~$\{H, K\}$, any edge labelled by~$w \in V(\Gamma)$ with~$\link{w} = L$ must be a loop at one of the three vertices above. Therefore, the complete set of edges labelled by~$H$ and~$K$ and their companions forms the graph shown in \emph{Figure \ref{fig:all companions to H and K}}:
        \begin{figure}[H]
            \centering
            \begin{tikzpicture}
                \fill (-3, 0) circle (2pt);
                \fill (0, 0) circle (2pt);
                \fill (3, 0) circle (2pt);
                \draw[black, thick] (-3, 0) -- (0, 0);
                \draw[black, thick] (0, 0) -- (3, 0);
                \draw[smooth, tension=1, black, thick] plot coordinates {(-3, 0) (-1, 1) (1, 1) (3, 0)};
                \node at (-3, 0) [below = 1pt, black, thick]{\small$\{P_H, P_K\}$};
                \node at (0, 0) [below = 1pt, black, thick]{\small$\{\overline{P_H}, P_K\}$};
                \node at (3, 0) [below = 1pt, black, thick]{\small$\{\overline{P_H}, \overline{P_K}\}$};
                \node at (-1.5, 0) [above = 2pt, black, thick]{$e_H$};
                \node at (1.5, 0) [above = 2pt, black, thick]{$e_K$};
                \node at (0, 1.1) [above = 2pt, black, thick]{$e_{H'} = e_v$};
            \end{tikzpicture}
            \caption{The edges labelled by~$H$ and~$K$, and their companions. The edges $e_H$ and $e_K$ are not oriented but note that (although not depicted)~$e_v$ is oriented depending on whether~$v \in P_H \cap P_K$ or~$v \in \overline{P_H} \cap \overline{P_K}$.}
            \label{fig:all companions to H and K} 
        \end{figure}
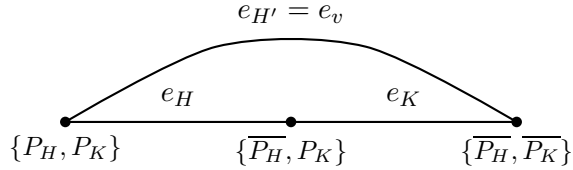

        Hence~$K$ is symmetric relative to~$\{H', K\}$ in~$\salb{\{\calP_H, \calP_K\}}$, and so is also symmetric relative to~$\widehat{\calT}$ in~$X$, as required.

        \item Suppose all hyperplanes in~$\calT$ are symmetric relative to~$\calT$ and let~$H \in \calT \setminus \calT'$. Define the treelike set~$\widehat{\calT} \coloneqq \left(\calT \setminus H\right) \cup H'$ as in (i), where~$H'$ is the unique companion hyperplane of~$H$. By (i), every hyperplane in~$\widehat{\calT}$ is symmetric relative to~$\widehat{\calT}$. Suppose that~$H'$ is labelled by some vertex~$v \in V(\Gamma)$; then~$\link{\calP_H} = \link{v}^\pm$. The union of the carriers of the hyperplanes labelled by vertices or partitions which have link~$\link{v}$ forms a subcomplex of~$X$ isomorphic to~$\Theta_{\link{v}} \times \mathbb{S}_{\link{v}}$, and the edges labelled by any treelike set form a maximal tree in~$\Theta_{\link{v}}$ (cf. \cite{CharneyStambaughVogtmann17}, Example 4.9). Therefore, since~$\calT'$ is treelike, we must have~$H' \in \calT'$. Therefore we have~$\vert \widehat{\calT} \setminus \calT'\vert < \vert \calT \setminus \calT'\vert$, and now the result follows by induction.
    \end{enumerate}
\end{proof}

\emph{Lemma \ref{cor:analogue of [Col89] Lem 4.2}} said that a~$\Gamma$-complex~$X$ is symmetric if and only if for every treelike set~$\calT$ in~$X$ and every~$H \in \calT$,~$H$ is symmetric relative to~$\calT$. \emph{Lemma \ref{lem:analogue of [Col89] Lem 4.3}} improves this to say that actually it suffices to find only one such treelike set~$\calT$.

Recall that we have fixed~$\calW$ to be the set of conjugacy classes of~$\raag$ of length one.

We can now define our retraction~$r \colon \Kmin{\calW} \to \symspine$. Let~$\xi = (X, \alpha) \in \Kmin{\calW}$, where~$X$ is a~$\Gamma$-complex and~$\alpha \colon X \to \sal$ is a marking. Define
\begin{eqnarray*}
    r \colon \Kmin{\calW} & \to & \symspine \\
    \left(X, \alpha\right) & \mapsto & \left(X^\Sigma, \alpha \circ \left(c_\calT^\Sigma\right)^{-1}\right)
\end{eqnarray*}
where~$c_\calT \colon X \to X \ssslash \calT \cong \sal$ is the collapse along a treelike set~$\calT$ such that
\[(c_\calT \circ \alpha^{-1})_\ast \colon \pi_1\left(\sal\right) \to \pi_1\left(\sal\right)\] 
is a symmetric automorphism, and~$c_\calT^\Sigma \colon X \to X^\Sigma$ is given by collapsing along all hyperplanes in~$\calT$ which are \emph{not} symmetric relative to~$\calT$.  

We need to show that this is a well-defined map; in particular, that~$X^\Sigma$ is actually a symmetric~$\Gamma$-complex, and that~$r(\xi)$ is independent of the choices of~$X$,~$\alpha$, and~$\calT$.

\subsection{Change of treelike set (CTS) automorphisms}\label{subsec:CTS auts}

To do this, we introduce \emph{change of treelike set (CTS) automorphisms}. These are our analogue of the `change of maximal tree (CMT) automorphisms' introduced by Gersten \cite{GerstenFixedPointsAutFn84} and utilised by Collins \cite{CollinsSymSpine89}.  

Note that for any treelike set of hyperplanes in a~$\Gamma$-complex~$X$, the edges dual to the hyperplanes in the treelike set form no non-null-homotopic loops. 

Let~$\calT$ be a treelike set of hyperplanes in a~$\Gamma$-complex~$X$. Let~$X$ have basepoint~$x_0$ in the 0-skeleton of~$X$. For each hyperplane~$\overline{H}$ of~$\sal$, there is a unique hyperplane~$H$ of~$X$ such that~$c_\calT(H) = \overline{H}$. Moreover, for each edge~$e_{\overline{H}}$ dual to~$\overline{H}$ in~$\sal$, there is a unique edge~$e_H$ dual to~$H$ in~$X$. There are paths in the 1-skeleton of~$X$ along edges dual to hyperplanes in~$\calT$ to each of the endpoints of each such edge~$e_H$, since the edges labelled by elements of~$\calT$ span the 0-skeleton of~$X$. We claim that these paths are unique up to homotopy. Indeed, let~$\gamma_1$,~$\gamma_2$ be two paths along edges dual to hyperplanes in~$\calT$, from~$x_0$ to the same endpoint of some edge~$e_H$. Then the composition~$\gamma_1\gamma_2^{-1}$ is a loop, which (since~$\calT$ is treelike) must be null-homotopic. Hence~$\gamma_1$ is homotopic to~$\gamma_2$. 

Recall that since~$H$ is labelled by a vertex of~$\Gamma$, all edges dual to~$H$ are oriented. Fix some such edge~$e_H$, and denote its two endpoints by~$s(e_H)$ and~$t(e_H)$, where the orientation travels along~$e_H$ from~$s(e_H)$ to~$t(e_H)$. We have a path~$\gamma$ along edges dual to hyperplanes in~$\calT$ from~$x_0$ to~$s(e_H)$, and a similar path~$\delta$ from~$x_0$ to~$t(e_H)$. Then~$\gamma e_H \delta^{-1}$ is a loop in the 1-skeleton of~$X$. The collapse~$c_\calT$ has canonical homotopy inverse given on~$e_{\overline{H}}$ by~$c_\calT^{-1}(e_{\overline{H}}) = \gamma e_H \delta^{-1}$. 

Now let~$c_{\calT}$ and~$c_{\calT'}$ be collapses~$X \to \sal$ along treelike sets~$\calT$ and~$\calT'$ respectively. This induces an automorphism~$\left(c_{\calT'}\right)_\ast \circ \left(c_{\calT}^{-1}\right)_\ast \colon \pi_1\left(\sal\right) \to \pi_1\left(\sal\right)$, which we call a \emph{change of treelike set (CTS) automorphism}:
\[\begin{tikzcd}
	& X && {} \\
	{X \ssslash \calT} & \sal & {X \ssslash \calT'}
	\arrow["{c_\calT}"', from=1-2, to=2-1]
	\arrow["{c_{\calT'}}", from=1-2, to=2-3]
	\arrow["{=}"{description}, draw=none, from=2-1, to=2-2]
	\arrow["{=}"{description}, draw=none, from=2-2, to=2-3]
\end{tikzcd}\]

This can be calculated by evaluating~$c_{\calT'}$ on the closed paths~$c_\calT^{-1}(e_{\overline{v}})$ for each edge~$e_{\overline{v}}$ of~$\sal$.

It will be helpful to briefly examine when a CTS automorphism~$\left(c_{\calT'}\right)_\ast \circ \left(c_{\calT}^{-1}\right)_\ast$ is symmetric. Recall that a symmetric automorphism is one where each generator~$v \in V(\Gamma)$ is sent to a conjugate of (possibly an inverse of) another generator. Therefore the paths~$c_{\calT}^{-1}(e_{\overline{v}})$ for~$e_{\overline{v}}$ an edge of~$\sal$ must be homotopic to a loop which contains only one edge dual to a hyperplane not in~$\calT'$. Indeed, since the subcomplex~$C^\Pi$ formed of cubes with labels only in~$\calT$ (cf. \cite{BregmanCharneyVogtmann23}, \S3.1) is~$\operatorname{CAT}(0)$, it is contractible, so we can adjust the paths~$\gamma, \delta \subseteq C^\Pi$ so that each does not repeat any labels. Moreover, if~$\delta$ shares a label with~$\gamma$, we can take the loop~$\gamma e_v \delta^{-1}$ to be formed of two loops, one of which is entirely contained in~$C^\Pi$, so is null-homotopic. So we may assume that each label appears only once in the path~$\gamma e_v \delta^{-1}$. Now evaluating this path on the collapse~$c_{\calT'}$, if this loop passes through more than one hyperplane not in~$\calT'$, the corresponding CTS automorphism will not be symmetric. 

\subsection{Proof that~$r$ is a retraction}\label{subsec:proof that r is a retraction}

In this subsection we will show that~$X^\Sigma$, as in the definition of the retraction~$r \colon \Kmin{\calW} \to \symspine$, is a symmetric~$\Gamma$-complex. 

\begin{lemma}\label{lem:analogue of [Col89] Lem 4.4}
    Let~$X$ be a~$\Gamma$-complex, and let~$\calT$,~$\calT'$ be treelike sets such that the corresponding CTS automorphism is symmetric. If~$H \in \calT \setminus \calT'$, then~$H$ is symmetric relative to~$\calT$. 
\end{lemma}

\begin{proof}
    Suppose that~$H$ has companion hyperplanes~$H_v$,~$H_w$ for~$v \neq w \in V(\Gamma)$. Note that~$H_v$,~$H_w$ are not elements of~$\calT$, by \emph{Lemma \ref{lem:no two elements of treelike set are companions}}. Let~$c_{\calT}(H_v) = H_{\overline{v}}$,~$c_\calT(H_w) = H_{\overline{w}} \in \sal$. As in the definition of CTS automorphisms above, fix a basepoint~$x_0$ in the 0-skeleton of~$X$, and pick edges~$e_v$,~$e_w$ in~$X$. Without loss of generality, take~$x_0 \in \beta(H)$. We use the following notation:
    \begin{itemize}
        \item let~$\gamma_v$ be a path from~$x_0$ to~$s(e_v)$;
        \item let~$\delta_v$ be a path from~$x_0$ to~$t(e_v)$;
        \item let~$\gamma_w$ be a path from~$x_0$ to~$s(e_w)$;
        \item let~$\delta_v$ be a path from~$x_0$ to~$t(e_w)$.
    \end{itemize}

    By \emph{Lemma \ref{lem:no two elements of treelike set are companions}}, no hyperplane which is a companion to~$H$ is an element of~$\calT$. Hence, since the loops~$\gamma_v e_v \delta_v^{-1}$ and~$\gamma_w e_w \delta_w^{-1}$ each pass between~$\beta(H)$ and~$\overline{\beta}(H)$, they must traverse along edges labelled by~$H$. Since the CTS automorphism~$\left(c_{\calT'}\right)_\ast \circ \left(c_{\calT}^{-1}\right)_\ast$ is symmetric, we may assume that each of these loops contains only one edge labelled by an element not in~$\calT'$. Since~$H \notin \calT'$, in each case it must be the edge labelled by~$H$. This implies that the hyperplanes dual to~$e_v$ and~$e_w$ are elements of~$\calT'$. Also, there is a path between vertices of~$\beta(H)$ along edges dual to hyperplanes in~$\calT'$ from~$s(e_v)$ to~$s(e_w)$, and a similar path containing only vertices of~$\overline{\beta}(H)$ along edges dual to hyperplanes in~$\calT'$ from~$t(e_v)$ to~$t(e_w)$. These two paths, along with the edges~$e_v$,~$e_w$, thus form a loop labelled by elements of~$\calT'$. Since~$\calT'$ is treelike, this loop must be null-homotopic, which implies that~$e_v$ is homotopic to~$e_w$. Therefore~$e_{\overline{v}}$ is homotopic to~$e_{\overline{w}}$ in~$\sal$, which is a contradiction as~$v \neq w$.
\end{proof}

For a treelike set~$\calT$ in a~$\Gamma$-complex~$X$, let~$\calT_{-\Sigma} \coloneqq \{H \in \calT \;\colon\; \text{$H$ is not symmetric relative to~$\calT$}\}$. 

\begin{proposition}\label{prop:analogue of [Col89] Prop 4.5}
    Let~$X$ be a~$\Gamma$-complex, and let~$\calT$,~$\calT'$ be treelike sets in~$X$ such that the corresponding CTS automorphism is symmetric. Then~$\calT_{-\Sigma} = \calT'_{-\Sigma}$.
\end{proposition}

\begin{proof}
    We induct on~$\vert \calT \setminus \calT' \vert$. If~$\vert \calT \setminus \calT' \vert = 0$, then there is nothing to prove. 

    Let~$H \in \calT \setminus \calT'$. \emph{Lemma \ref{lem:analogue of [Col89] Lem 4.4}} implies that~$H$ is symmetric relative to~$\calT$; let~$H'$ be its unique companion hyperplane. Clearly we have~$H' \in \calT'$, as otherwise~$c_{\calT'}$ would fail to collapse the 0-skeleton of~$X$ to the single vertex in~$\sal$. 
    
    Define~$\widehat{\calT} \coloneqq \left( \calT \setminus H \right) \cup H'$. This is a treelike set, and the collapse map~$c_{\widehat{\calT}}$ differs from~$c_\calT$ by the isomorphism of~$X$ which exchanges the hyperplanes~$H$ and~$H'$ (more precisely, the isomorphism~$h_L$ induced by exchanging the edges~$e_H$ and~$e_{H'}$ in the graph~$\Theta_L$, where the union of the carriers of the hyperplanes corresponding to vertices or partitions with link~$L$ decomposes as~$\Theta_L \times \mathbb{S}_L$). We now consider the CTS automorphism from~$\calT$ to~$\widehat{\calT}$. To do so, for each edge~$e_{\overline{v}} \in \sal$ we evaluate the closed loops~$c_\calT^{-1}(e_{\overline{v}})$ on~$c_{\widehat{\calT}}$. 

    For each edge~$e_{\overline{v}}$ of~$\sal$, let~$\gamma_v e_v \delta_v^{-1}$ be the (unique up to homotopy) loop along edges dual to elements of~$\calT$ (excepting~$e_v$, the unique edge mapped to~$e_{\overline{v}}$ by~$c_\calT$) based at a basepoint~$x_0$, an element of~$\beta(H)$. Let~$B(H)$ be the full subcomplex of~$X$ with 0-skeleton~$\beta(H)$; similarly define~$\overline{B}(H)$.

    \begin{figure}[H]
        \centering
        \begin{tikzpicture}[scale = 0.8]
            %
            \draw (-6, 0) circle (2);
            \draw (6, 0) circle (2);
            \node at (-6, -2.5) [black, thick]{$B(H)$};
            \node at (6, -2.5) [black, thick]{$\overline{B}(H)$};
            \fill (-6.5, 0.5) circle (2pt);
            \node at (-6.5, 0.5) [black, thick, below = 2pt] {$x_0$};
            %
            %
            \draw[black, thick, fill = black!20!white] (-2, 2.4) to (-1.6, 2.7) to (-1.6, 1.7) to (-2, 1.4) to (-2, 2.4);
            \draw[black, thick, fill = black!20!white] (-1.5, -2.7) to (-1.1, -2.4) to (-1.1, -1.4) to (-1.5, -1.7) to (-1.5, -2.7);
            %
            %
            \draw[smooth, tension=0.5, black, thick, ->-] plot coordinates {(-4.8, 1.6) (0, 2.3) (4.8, 1.6)};
            \draw[smooth, tension=0.5, black, thick, ->-] plot coordinates {(-4.7, 1.52) (0, 2.2) (4.7, 1.52)};
            \draw[smooth, tension=0.5, black, thick, ->-] plot coordinates {(-4.57, 1.4) (0, 2) (4.57, 1.4)};
            \fill (0, 2.05) circle (0.3pt);
            \fill (0, 2.1) circle (0.3pt);
            \fill (0, 2.15) circle (0.3pt);
            \draw[smooth, tension=0.5, black, thick, ->-] plot coordinates {(-4.8, -1.6) (0, -2.3) (4.8, -1.6)};
            \draw[smooth, tension=0.5, black, thick, ->-] plot coordinates {(-4.7, -1.52) (0, -2.2) (4.7, -1.52)};
            \draw[smooth, tension=0.5, black, thick, ->-] plot coordinates {(-4.57, -1.4) (0, -2) (4.57, -1.4)};
            \fill (0, -2.05) circle (0.3pt);
            \fill (0, -2.1) circle (0.3pt);
            \fill (0, -2.15) circle (0.3pt);
            %
            %
            \node at (-0.2, 2.45) [black, thick]{$e_H$};
            \node at (-1.9, 2.75) [black, thick]{$H$};
            \node at (0.5, -1.7) [black, thick]{$e_{H'}$};
            \node at (-1, -2.75) [black, thick]{$H'$};
        \end{tikzpicture}
        \caption{Separation of~$X$ into~$B(H)$ and~$\overline{B}(H)$ with respect to~$H \in \calT$, which has unique companion hyperplane~$H' \in \calT'$.}
        \label{fig:hyperplanes H, H'}
    \end{figure}
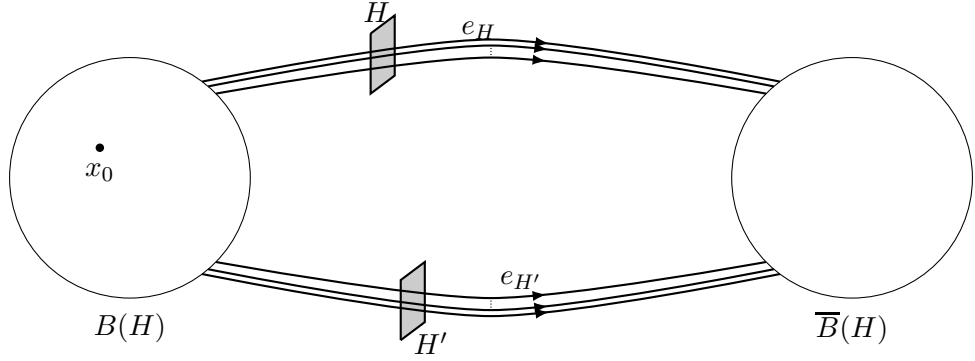

    First of all, we have~$c_{\widehat{\calT}}(e_H) = c_\calT\left(h_L(e_H)\right) = c_\calT(e_{H'})$.

    Suppose that~$e_v \subseteq B(H)$. Suppose that the loop~$\gamma_v e_v \delta_v^{-1}$ passes into~$\overline{B}(H)$. If so, it must leave~$B(H)$ along an edge~$e_H$ labelled by~$H$, and then come back from~$\overline{B}(H)$ along an edge~$e_H'$ labelled by~$H$.~$C^\Pi$ is contractible, so there exists a homotopy between~$e_H$ and~$e_H'$. Hence we may assume that~$e_H = e_H'$, and so since~$\calT$ is treelike, the resulting loop is null-homotopic. Therefore, we may assume the~$\gamma_v e_v \delta_v^{-1}$ is entirely contained in~$B(H)$. Hence every element of~$\calT$ dual to an edge of this loop is also an element of~$\widehat{\calT}$, so~$\gamma_v e_v \delta_v^{-1}$ is also the (unique up to homotopy) homotopy inverse image of~$e_{\overline{v}}$ under~$c_{\widehat{\calT}}$. In other words, 
    \[\left(c_{\widehat{\calT}}\right)_\ast \left(\left(c_\calT^{-1}\right)_\ast(v)\right) = v.\]

    Finally, consider some~$e_v \subseteq \overline{B}(H)$. We can homotope the loop~$\gamma_v \e_v \delta_v^{-1}$ so that it traverses once from~$B(H)$ to~$\overline{B}(H)$ along an edge labelled~$H$, and then after having traversed~$e_v$, passes back into~$B(H)$ along an edge labelled~$H$. Apart from~$e_v$ and these edges dual to~$H$, every edge of this loop is dual to an element of~$\widehat{\calT}$. Hence 
    \begin{eqnarray*}
        c_{\widehat{\calT}} \left(c_\calT^{-1}(v)\right) & = & c_{\widehat{\calT}} \left(e_H e_v e_H^{-1}\right) \\
        & = & c_{\widehat{\calT}}(e_H) e_{\overline{v}} \left(c_{\widehat{\calT}}(e_H)\right)^{-1} \\
        & = & c_\calT(e_{H'}) e_{\overline{v}} \left( c_\calT(e_{H'}) \right)^{-1}. 
    \end{eqnarray*} 
    
    Hence the CTS automorphism from~$\calT$ to~$\widehat{\calT}$ is symmetric. By induction, it is now sufficient to show that~$\calT_{-\Sigma} = \widehat{\calT}_{-\Sigma}$. Since~$H'$ is symmetric relative to~$\widehat{\calT}$, this follows by \emph{Lemma \ref{lem:analogue of [Col89] Lem 4.3}, (i)}, and we have the result.
\end{proof}

\begin{lemma}\label{lem:analogue of [Col89] Lem 4.6}
    Let~$\calT$ be a treelike set of hyperplanes in a~$\Gamma$-complex~$X$.
    \begin{enumerate}[(i)]
        \item Let~$c_H \colon X \to X \ssslash H$ collapse a single hyperplane~$H \in \calT$. Then~$H' \in \calT \setminus H$ is symmetric relative to~$\calT$ if and only if~$c_H(H')$ is symmetric in~$c_H(X)$ relative to~$c_H(\calT)$.
        \item Let~$c_\calT^\Sigma \colon X \to X^\Sigma$ be the hyperplane collapse along~$\calT_{-\Sigma}$. Then~$X^\Sigma$ is a symmetric~$\Gamma$-complex.
    \end{enumerate}
\end{lemma}

\begin{proof}
    \begin{enumerate}[(i)]
        \item Let~$K \in \calT$ have companions~$K_1, \;\dots,\; K_m$, none of which are in~$\calT$, by \emph{Lemma \ref{lem:no two elements of treelike set are companions}}. By the same lemma,~$H$ is not a companion hyperplane to~$K$, so the collapse~$c_H$ collapses edges between vertices of either~$\beta(K)$ or~$\overline{\beta}(K)$; without loss of generality suppose the former. Hence~$c_H(K)$ has companions~$c_H(K_1), \;\dots, \; c_H(K_m)$. The result follows.
        
        \item By \emph{Lemmas \ref{cor:analogue of [Col89] Lem 4.2} \& \ref{lem:analogue of [Col89] Lem 4.3}},~$X^\Sigma$ is symmetric if and only if every element of~$c_\calT^\Sigma(\calT)$ is symmetric relative to~$c_\calT^\Sigma(\calT)$ (which is treelike in~$X^\Sigma$). 

        If~$\vert \calT_{-\Sigma} \vert = 0$, then there is nothing to prove. Otherwise, apply (i) to~$H' \in \calT_{-\Sigma}$; the result follows by induction on~$\vert \calT_{-\Sigma} \vert$.
    \end{enumerate}
\end{proof}

Let~$\xi = (X, \alpha) \in \Kmin{\calW}$---in other words, there is a treelike set~$\calT$ such that
\[\left(c_\calT \circ \alpha^{-1}\right)_\ast \colon \pi_1\left(\sal\right) \to \pi_1\left(\sal\right)\] 
is a symmetric automorphism. Then after collapsing all non-symmetric hyperplanes relative to~$\calT$,~$\left(X^\Sigma, \alpha \circ (c_\calT^\Sigma)^{-1}\right)$ is a symmetric marked~$\Gamma$-complex. In other words, we have now proved that~$r(\xi)$ is always an element of~$\symspine$. To prove that~$r$ is well-defined, it remains to show that this image~$r(\xi)$ is independent of the choices of~$X$,~$\alpha$, and~$\calT$. 

To this end, let~$\xi' = (X', \alpha') \in \Kmin{\calW}$ with a treelike set of hyperplanes~$\calT'$ of~$X'$ such that~$r(\xi) = r(\xi')$. We reduce to considering the case that~$\vert \calT \vert = \vert \calT' \vert$, as if (say)~$\vert \calT' \vert < \vert \calT \vert$, then we can collapse~$(\vert \calT \vert - \vert \calT' \vert)$ hyperplanes of~$\calT$ to obtain some~$\overline{\calT}$ with~$\vert \overline{\calT} \vert = \vert \calT' \vert$. The result for two equally-sized treelike sets, pre-composed with the hyperplane collapse~$X \to X \ssslash (\calT \setminus \overline{\calT})$, will give~$r(\xi) = r(\xi')$, as required. 

Since~$\vert \calT \vert = \vert \calT' \vert$, there is a cube complex isomorphism~$h \colon X \xrightarrow{\cong} X'$ such that
\[\begin{tikzcd}
	X && {X'} \\
	& \sal
	\arrow["h", "\cong"', from=1-1, to=1-3]
	\arrow["\alpha"', from=1-1, to=2-2]
	\arrow["{\alpha'}", from=1-3, to=2-2]
\end{tikzcd}\]
commutes.

Since~$\left(c_\calT \circ \alpha^{-1}\right)_\ast$ and~$\left(c_{\calT'} \circ \alpha'^{-1}\right)_\ast$ are symmetric automorphisms of~$\pi_1\left(\sal\right)$, and the cube complex isomorphism~$h$ preserves symmetric hyperplanes, the CTS automorphism from~$h(\calT)$ to~$\calT'$ is also symmetric. Applying \emph{Proposition \ref{prop:analogue of [Col89] Prop 4.5}}, we obtain~$\calT'_{-\Sigma} = \left(h(\calT)\right)_{-\Sigma} = h\left(\calT_{-\Sigma}\right)$, and so~$h$ induces an isomorphism~$h^\Sigma \colon X^\Sigma \to X'^\Sigma$ such that 
\[\begin{tikzcd}
	X^\Sigma && {X'^\Sigma} \\
	& \sal
	\arrow["h^\Sigma", "\cong"', from=1-1, to=1-3]
	\arrow["\alpha \circ (c_\calT^\Sigma)^{-1}"', from=1-1, to=2-2]
	\arrow["{\alpha' \circ (c_{\calT'}^\Sigma)^{-1}}", from=1-3, to=2-2]
\end{tikzcd}\]
commutes. Hence
\[\left(X^\Sigma, \alpha \circ (c_\calT^\Sigma)^{-1}\right) \sim \left(X'^\Sigma, \alpha' \circ (c_{\calT'}^\Sigma)^{-1}\right);\] i.e.,~$r(\xi) = r(\xi')$. Thus~$r \colon \Kmin{\calW} \to \symspine$ is well-defined. 

In order to finish Step (iii) as outlined in \S\ref{subsec:discussion of proof strategy}, we prove that~$r$ is a deformation retraction, using the Poset Lemma (\emph{Lemma \ref{lem:poset lemma}}).

\newpage 

\begin{theorem}\label{thm:symspine is a retract of Kmin}
   ~$r \colon \Kmin{\calW} \to \symspine$ is a deformation retraction.
\end{theorem}

\begin{proof}
    We check that~$r$ satisfies the conditions of the Poset Lemma.~$\spine$ is finite-dimensional, so~$\Kmin{\calW}$ is a poset of finite height;~$\symspine$ is a subposet of~$\Kmin{\calW}$. Since~$r(\xi)$ is obtained from~$\xi$ by a hyperplane collapse, we always have~$r(\xi) \leq \xi$. 

    It remains to show that~$r$ is actually a poset map; that is, given~$\xi_1 = (X_1, \alpha_1) \leq \xi_2 = (X_2, \alpha_2)$, we have~$r(\xi_1) \leq r(\xi_2)$. Since~$\xi_1 \leq \xi_2$, there is a collapse~$c_{21} \colon X_2 \to X_1$ such that 
    \[\begin{tikzcd}
	X_2 && {X_1} \\
	& \sal
	\arrow["c_{21}", from=1-1, to=1-3]
	\arrow["\alpha_2"', from=1-1, to=2-2]
	\arrow["{\alpha_1}", from=1-3, to=2-2]
    \end{tikzcd}\]
    commutes.

    Choose a collapse~$c_1 \colon X_1 \to \sal$ such that~$\left( c_1 \circ \alpha_1^{-1} \right)_\ast$ is a symmetric automorphism (this can always be done, since~$\xi_1 \in \Kmin{\calW}$). Define~$c_2 \coloneqq c_1 \circ c_{21}$, and then~$\left( c_2 \circ \alpha_2^{-1} \right)_\ast$ is also symmetric; moreover~$r(\xi_1) = \left(X_1^\Sigma, \alpha_1 \circ (c_1^\Sigma)^{-1}\right)$ and~$r(\xi_2) = \left(X_2^\Sigma, \alpha_2 \circ (c_2^\Sigma)^{-1}\right)$. Also, if~$\calT_1$ is the treelike set in~$X_1$ corresponding to the collapse~$c_1$, then the treelike set~$\calT_2$ for~$c_2$ is the inverse image of~$\calT_1$ under~$c_{21}$. Hence we are now just looking for~$c_{21}^\Sigma \colon X_2^\Sigma \to X_1^\Sigma$ such that 
    \[\begin{tikzcd}
	X_2^\Sigma && {X_1^\Sigma} \\
	& \sal
	\arrow["c_{21}^\Sigma", from=1-1, to=1-3]
	\arrow["\alpha_2 \circ (c_2^\Sigma)^{-1}"', from=1-1, to=2-2]
	\arrow["{\alpha_1 \circ (c_1^\Sigma)^{-1}}", from=1-3, to=2-2]
    \end{tikzcd}\]
    commutes. 

    Let~$C^{\calT_2}$ be the subcomplex of~$X_2$ which has all cubes labelled only by elements of~$\calT_2$ (cf. \cite{BregmanCharneyVogtmann23}, \S3.1). Let~$x$ be a vertex of~$X_2^\Sigma$; then its inverse image under~$c_2^\Sigma$ is a subcomplex~$C_x^{\calT_2}$ of~$C^{\calT_2}$. If a hyperplane of~$C_x^{\calT_2}$ is symmetric relative to~$\calT_2$, then~$c_2^\Sigma$ would not collapse this hyperplane, which is a contradiction since~$c_2^\Sigma\left(C_x^{\calT_2}\right) = x$. Hence every hyperplane in~$C_x^{\calT_2}$ is not symmetric relative to~$\calT_2$, and so by \emph{Lemma \ref{lem:analogue of [Col89] Lem 4.6}, (i)}, each hyperplane of~$c_{21}\left(C_x^{\calT_2}\right)$ is not symmetric relative to~$\calT_1$. Hence~$c_1^\Sigma \left( c_{21} \left( C_x^{\calT_2} \right)\right) = y$, a vertex of~$X_1^\Sigma$. Define~$c_{21}^\Sigma(x) \coloneqq y$.

    Similarly, for each edge~$\overline{e}$ of~$X_2^\Sigma$, there is a unique edge~$e$ of~$X_2$ such that~$c_2^\Sigma(e) = \overline{e}$. Therefore, define~$c_{21}^\Sigma(e) \coloneqq c_1^\Sigma \left( c_{21}(e)\right)$. We need to check that~$\alpha_1 \circ \left(c_1^\Sigma\right)^{-1} \circ c_{21}^\Sigma = \alpha_2 \circ \left(c_2^\Sigma\right)^{-1} \coloneqq X_2^\Sigma \to \sal$. By definition,~$c_{21}^\Sigma \circ c_2^\Sigma = c_1^\Sigma \circ c_{21}$, so 
    \begin{eqnarray*}
        \alpha_1 \circ \left(c_1^\Sigma\right)^{-1} \circ c_{21}^\Sigma \circ c_2^\Sigma & = & \alpha_1 \circ \left(c_1^\Sigma\right)^{-1} \circ c_1^\Sigma \circ c_{21} \\
        & = & \alpha_1 \circ \alpha_{21} = \alpha_2,
    \end{eqnarray*}
    since~$\xi_1 \leq \xi_2$. Hence the diagram above commutes with this choice of~$c_{21}^\Sigma$, and so~$r(\xi_1) \leq r(\xi_2)$, as required.
\end{proof}

Thus~$\symspine$ is a deformation retract of~$\Kmin{\calW}$, and Step (iii) is complete, which proves \emph{Theorem \ref{customthm:there exists a symmetric spine}}.

\begin{corollary}\label{cor:symspine is contractible}
    The symmetric spine~$\symspine$ is contractible.
\end{corollary}

\subsection{Equivalent definition of symmetric~$\Gamma$-complex}\label{subsec:equivalent definition of symmetric Gamma-complex}

We can also deduce an equivalent condition for a~$\Gamma$-complex to be symmetric, as mentioned after \emph{Definition \ref{def:symmetric Gamma-complex}}. We call a non-null-homotopic loop \emph{simple} if it cannot be decomposed as a concatenation of more than one non-null-homotopic loop.

\begin{proposition}\label{prop:equivalence of definitions of symmetric Gamma-complex}
    A~$\Gamma$-complex is symmetric if and only if every edge lies in a unique homotopy class of simple non-null-homotopic loops. 
\end{proposition}

\begin{proof}
    We do the backwards direction first, by contrapositive. Let~$\calT$ be a treelike set in a non-symmetric~$\Gamma$-complex~$X$. Pick~$H_0 \in \calT$, and suppose that~$H_0$ is not symmetric relative to~$\calT$; let~$H_1$ and~$H_2$ be two of its companion hyperplanes, labelled by distinct vertices~$v_1 \neq v_2 \in V(\Gamma)$. For each~$i$, pick an edge~$e_i$ dual to~$H_i$, and write~$s(e_i)$ (resp.~$t(e_i)$) for the endpoint in~$\beta(H_0)$ (resp.~$\overline{\beta}(H_0)$). Let~$\gamma_{0i}$ be the unique (up to homotopy) paths contained in~$B(H_0)$ along edges dual to elements of~$\calT$ from~$s(e_0)$ to~$s(e_i)$, for~$i = 1, 2$. Similarly, let~$\delta_{0i}$ be the unique (up to homotopy) paths contained in~$\overline{B}(H_0)$ along edges dual to elements of~$\calT$ from~$t(e_0)$ to~$t(e_i)$. Now the paths~$e_0 \delta_{0i} e_i^{-1} \gamma_{0i}^{-1}$ for~$i = 1, 2$ are both non-null-homotopic loops. Indeed, if one were null-homotopic, then its image under the homotopy equivalence~$c_\calT$ would also be null-homotopic, which is a contradiction, as this image is the generator of~$\pi_1\left(\sal\right)$ corresponding to the vertex~$v_i \in V(\Gamma)$. For the same reason, these two loops~$e_0 \delta_{0i} e_i^{-1} \gamma_{0i}^{-1}$ cannot be homotopic to each other. Hence the edge~$e_0$ lies in two non-homotopic non-null-homotopic loops in~$X$, whence this direction is complete.

    Conversely, suppose that~$X$ is symmetric. Then every hyperplane~$H \in \calT$ in a treelike set~$\calT$ in~$X$ is symmetric relative to~$X$. Let~$e_0$ be an arbitrary edge in~$X$; we will prove that~$e_0$ lies in exactly one non-null-homotopic loop (up to homotopy). If~$e_0$ is a loop, then we're done, so we may suppose this is not the case. Let~$H_0$ be the hyperplane dual to~$e_0$. Since~$e_0$ is not a loop in~$X$, we can take a treelike set~$\calT$ containing~$e_0$; by our assumption,~$e_0$ is symmetric relative to~$\calT$. Let~$H_1$ be the unique companion hyperplane to~$H_0$, and let~$v \in V(\Gamma)$ be the label of~$H_1$.  

    For~$i = 0, 1$, let~$s(e_i)$ (resp.~$t(e_i)$) denote the endpoint of~$e_i$ contained in~$\beta(H_0)$ (resp.~$\overline{\beta}(H_0)$). Let~$\gamma$ (resp.~$\delta$) denote the unique (up to homotopy) path contained in~$B(H_0)$ (resp.~$\overline{B}(H_0)$) along edges dual to elements of~$\calT$ from~$s(e_1)$ to~$s(e_0)$ (resp. from~$t(e_0)$ to~$t(e_1)$). Then~$\lambda \coloneqq e_0 \delta e_1^{-1} \gamma$ is a loop, and its image under~$c_\calT$ is a non-null-homotopic loop in~$\sal$ corresponding to the generator~$v \in \pi_1\left(\sal\right)$. Hence~$\lambda$ is a non-null-homotopic loop in~$X$. 

    It remains to prove that any other simple non-null-homotopic loop involving~$e_0$ is homotopic to~$\lambda$. Let~$\lambda'$ be one such; homotope~$\lambda'$ so that it is of the form~$\lambda' = e_0 \delta' e_1^{-1} \gamma'$, with~$\gamma' \subseteq B(H_0)$ and~$\delta' \subseteq \overline{B}(H_0)$. [For clarity, we draw attention to the fact that~$\gamma'$ and~$\delta'$ are not required to consist entirely of edges dual to elements of~$\calT$.] If we can homotope~$\gamma$ onto~$\gamma'$ and~$\delta$ onto~$\delta'$, then we are done, so without loss of generality suppose that~$\delta$ is not homotopic to~$\delta'$ in~$X$. If~$\delta$ is a loop, then it is null-homotopic, (and~$t(e_0) = t(e_1)$) so~$\delta'$ is a non-null-homotopic loop, in which case~$\lambda'$ is not simple, which is a contradiction. Hence~$\delta$ is not a loop, so it is some non-empty path. If every edge of~$\delta$ appears in (a curve homotopic to)~$\delta'$, then once again~$\lambda'$ is not simple. Therefore,~$\delta$ contains some edge~$e$ which does not appear in (a curve homotopic to)~$\delta'$.

    Let~$e$ be dual to a hyperplane~$H \in \calT$. Since~$X$ is symmetric,~$H$ has a unique companion hyperplane. Since~$e$ forms part of the simple non-null-homotopic loop~$\lambda$, this unique companion hyperplane must have a dual edge somewhere in~$\lambda$. By \emph{Lemma \ref{lem:no two elements of treelike set are companions}}, the only candidate edge is~$e_1$, since it is the only one not dual to an element of~$\calT$. Hence~$\overline{B}(H) \subseteq \overline{B}(H_0)$, and
    \[B(H) = \left(\overline{B}(H_0) \setminus \overline{B}(H)\right) \cup B(H_0).\]
    But now the path~$\delta'$ begins in~$\overline{B}(H_0) \cap B(H)$ and ends in~$\overline{B}(H)$ without ever leaving~$\overline{B}(H_0)$, which necessitates it involving~$e$, which is a contradiction. 
\end{proof}

In the case that~$\raag$ is free, this is precisely Collins's original definition \cite{CollinsSymSpine89} that a graph is \emph{symmetric} if every edge lies in a unique circuit. 

\section{Dimension of the symmetric spine}\label{sec:dimension of sym spine}

In this section we study the relationship between the virtual cohomological dimension ($\textsc{vcd}$) of~$\symout$ and the dimension of~$\symspine$. We recall the following fundamental theorem from \S\ref{subsubsec:vcd}.

\begin{theorem}[\cite{BrownCohomologyofGroups}, \emph{Theorem VIII.11.1}]
    Let~$G$ be a discrete group acting properly and cocompactly on a proper contractible CW-complex~$X$. Then~$\textsc{vcd}(G) \leq \dim(X)$.
\end{theorem}

Therefore, since~$\symout$ acts properly and cocompactly on~$\symspine$, we have 
\[\vcdsymout \leq \dim\left(\symspine\right).\]

Work of Millard--Vogtmann \cite{MillardVogtmann2019} shows that the untwisted group of automorphisms,~$\uag$, contains free abelian subgroups, with `natural' generating sets. The rank of such a subgroup is a lower bound on~$\vcduag$. We summarise Theorems 4.25 \& 4.28 from \cite{MillardVogtmann2019} in \emph{Theorem \ref{thm:[MV19] 4.25 & 4.28 summary}} below; we first establish some notation. 

For a subset~$W \subseteq V = V(\Gamma)$, denote by~$M(W)$ the largest possible size of a compatible set of~$\Gamma$-partitions, all of which are based at an element of~$W$. Let~$L$ denote the set of principal vertices of~$\Gamma$. 

\subsection{Lower bound for~$\vcdsymout$}\label{subsec:lower bound for vcdsymout}

\begin{theorem}[\cite{MillardVogtmann2019}]\label{thm:[MV19] 4.25 & 4.28 summary}
    Fix a finite simplicial graph~$\Gamma$.
    \begin{enumerate}[(i)]
        \item For any set of principal~$\Gamma$-partitions, multipliers can be chosen for the corresponding~$\Gamma$-Whitehead automorphisms so that they generate a free abelian subgroup of~$\uag$.
        \item There exists a free abelian subgroup of~$\uag$ of rank~$M(L)$ generated by~$\Gamma$-Whitehead automorphisms corresponding to principal~$\Gamma$-partitions.
        \item Any free abelian subgroup of~$\uag$ generated by~$\Gamma$-Whitehead automorphisms has rank at most~$M(L)$.
    \end{enumerate} 
\end{theorem}

Here, (i) \& (ii) follow from (\cite{MillardVogtmann2019}, Theorem 4.25) and its proof, while (iii) is Theorem 4.28. In particular, it follows that~$M(L)$ is a natural lower bound for~$\vcduag$. We will prove a similar result in the 'symmetric' setting, adapting some theory developed in (\cite{MillardVogtmann2019}, \S4). 

For any vertex~$v \in V$, write~$[v]_0 \coloneqq \{w \in V \;\vert\; \link{v} = \link{w}\}$. If~$[v]_0$ is a singleton, then call~$[v]$ an \emph{abelian} equivalence class (because in this case all elements of~$[v]$ commute). Otherwise, we say that~$[v]$ is a \emph{nonabelian} equivalence class. We will also write~$[v]_\ast \coloneqq \{w \in V \;\vert\; \str{w} = \str{v}\}$. 

For~$\Pi$ a compatible set of~$\Gamma$-partitions, and~$v \in V$, write~$\Pi_v$ for the subset of~$\Pi$ based at~$v$, and write~$\Pi_{[v]}$ for the subset of~$\Pi$ based at elements of~$[v]$.

For any side~$P$ of a partition~$\calP$ based at a vertex~$v$, we write~$\mathring{P}$ for the intersection of~$P$ with~$V^\pm \setminus \str{m}^\pm$. 

The next proposition is an adaptation of (\cite{MillardVogtmann2019}, 4.22 \& 4.23) to our setting. It relies on the following lemma.

\begin{lemma}[\cite{MillardVogtmann2019}, Lemmas 4.12 and 4.21]\label{lem:[MV19], Lem 4.12 and Lem 4.21}
    Let~$m$ be a principal vertex of~$\Gamma$ and let~$\Pi$ be a compatible set of~$\Gamma$-partitions. Let~$\calP_1, \dots, \calP_k \in \Pi_{[m]_\ast}$. Then there is a nest 
    \[\emptyset \eqqcolon \mathring{P_0} \subsetneq \mathring{P_1} \subsetneq \dots \subsetneq \mathring{P_k} \subsetneq \mathring{P}_{k+1} \coloneqq V^\pm \setminus \str{m}^\pm.\] 
    Suppose that~$\calQ \in \Pi \setminus \Pi_{[m]_\ast}$ is based at~$n$. If~$m$ does not commute with~$n$, then there is a side~$Q$ of~$\calQ$ with~$Q \setminus \mathring{\overline{P_{i-1}}} \cap \mathring{P_i}$ for some~$i \in \{1, \; \dots, \; k+1\}$.
\end{lemma}

\begin{proposition}\label{prop:[MV19], Prop 4.22 adaptation}
    Let~$m$ be a principal vertex of~$\Gamma$ and let~$\Pi$ be a compatible set of symmetric~$\Gamma$-partitions. Let~$\calP_1$,~$\calP_2 \in \Pi_m$ and choose sides~$P_1$,~$P_2$ respectively such that~$P_1 \subsetneq P_2$. Suppose that~$n \leq_\circ m$ is contained in~$\overline{P_1} \cap P_2$. Let~$Q$ be a largest subset of~$\overline{P_1} \cap P_2$ which is a side of a partition in~$\Pi$ and is based at some~$v \sim n$; if there are no such subsets then take~$v = n$ and set~$Q = \{n\}$. Define~$\calP$ to be the symmetric~$\Gamma$-partition determined by the side~$P = P_1 \cup Q \cup \{v, v^{-1}\}$. 

    If~$\calR \in \Pi \setminus \Pi_{[m]_\ast}$, based at a vertex~$s$, is not compatible with~$\calP$, then some side~$R$ of~$\calR$ is contained in~$\overline{P_1} \cap P_2$, and either 
    \begin{itemize}
        \item~$s = v$ and~$\{v, v^{-1}\} \subseteq R \cup Q$; or
        \item~$s >_\circ n$ and~$R$ contains~$Q \cup \{v, v^{-1}\}$.
    \end{itemize}
\end{proposition}

\begin{proof}
    Note that~$\calP$ is based at~$m$; let~$\calR$ be based at~$s$. Since~$s \notin [m]_\ast$ and~$\calR$ is not compatible with~$\calP$,~$s$ and~$m$ do not commute.

    Hence by \emph{Lemma \ref{lem:[MV19], Lem 4.12 and Lem 4.21}},~$\calR$ has a side~$R$ contained in~$\mathring{P_1}$,~$\mathring{\overline{P_1}} \cap \mathring{P_2}$, or~$\mathring{\overline{P_2}}$. In the first and third cases,~$\calR$ is compatible with~$\calP$, so we must have~$R \subseteq \mathring{\overline{P_1}} \cap \mathring{P_2}$. We have~$\link{v} \subseteq \link{m}$ and know that~$s$ does not commute with~$m$, so~$d_\Gamma(s, v) \neq 1$. Hence, compatibility of~$\calR$ with~$\calQ$ implies that one of the four quadrants~$R \cap Q$,~$R \cap \overline{Q}$,~$\overline{R} \cap Q$,~$\overline{R} \cap \overline{Q}$ is empty. Since both~$R$ and~$Q$ are contained in~$\mathring{\overline{P_1}} \cap \mathring{P_2}$, we have~$\overline{R} \cap \overline{Q} \neq \emptyset$. If~$R \cap \overline{Q} = \emptyset$, then~$R \subseteq Q \subseteq P$  and so~$\calR$ is compatible with~$\calP$, which is a contradiction. Hence either~$R \cap Q$ or~$\overline{R} \cap Q$ is empty. 
    
    Without loss of generality, suppose that~$v \in Q$. If~$R \cap Q = \emptyset$, then~$Q \subseteq \overline{R}$, so~$v \in \overline{R}$. If~$\calR$ is not based at~$v$, then we have~$\{v, v^{-1}\} \subseteq \overline{R}$ (since~$\calR$ is symmetric), which implies that~$v^{-1} \notin R$. Combined with the facts that~$R \cap Q = \emptyset$ and~$R \subseteq \mathring{\overline{P_1}} \cap \mathring{P_2}$, this means that~$\calR$ is compatible with~$\calP$, which is a contradiction. 

    Thus we are left with two cases: either~$R \cap Q = \emptyset$ and~$\calR$ is based at~$v$ (so necessarily~$v^{-1} \in R$), which is one half of the statement, or~$\overline{R} \cap Q = \emptyset$. 

    If we are in the latter case, then~$Q \subseteq R$. Due to the maximality assumption in the choice of~$\calQ$, we must have~$s \notin [v] = [n]$. Hence~$\calR$ is not based at~$v$, and~$\calR$ is symmetric, so we have~$v^{-1} \in R$. If there is some~$x \in \link{v} \subseteq \str{m}$ with~$x \notin \link{s}$, then we would have~$d_\Gamma(x, v) = d_\Gamma(x, m) = 1$, and so~$v$ and~$m$ would be in the same component of~$\Gamma \setminus \str{s}$. This is impossible, as~$\calR$ separates~$m$ from~$v$. Hence~$v <_\circ s$, which completes the other half of our statement.
\end{proof}

\begin{lemma}\label{lem:[MV19], Cor 4.23 adaptation}
    Let~$\Pi$ be a compatible set of symmetric~$\Gamma$-partitions and let~$[m]$ be a nonabelian equivalence class of principal vertices of~$\Gamma$. Then for any~$m \in [m]$,~$\Pi_{[m]}$ can be replaced in~$\Pi$ by a set of the same size consisting of symmetric partitions based only at~$m$. 
\end{lemma}

\begin{proof}
    Fix~$m \in [m] = [m]_0$ such that~$\Pi_m \neq \emptyset$; write~$\Pi_m = \{\calP_1,\;\dots,\;\calP_k\}$. Let~$P_1 \subsetneq \dots \subsetneq P_k$ be the sides of the~$\calP_i$ containing~$m$.

    Suppose that~$\calQ \in \Pi_{[m]} \setminus \Pi_m$ is based at~$n \sim m$. Since~$[m]$ is nonabelian,~$n$ does not commute with~$m$, so by \emph{Lemma \ref{lem:[MV19], Lem 4.12 and Lem 4.21}} there is some~$i$ and a side~$Q$ of~$\calQ$ such that~$Q \subseteq \overline{P_{i-1}} \cap P_i$. Choose~$Q$ to be maximal with respect to inclusion among all such sides in~$\overline{P_{i-1}} \cap P_i$, and take~$M$ to be maximal among all such sides which are properly contained in~$Q$; if there is no such~$M$, then put~$M = \{n\}$.

    Define the~$\Gamma$-partition~$\calP$ by the side~$P \coloneqq P_{i-1} \cup M \cup \{n, n^{-1}\}$. We now apply \emph{Proposition \ref{prop:[MV19], Prop 4.22 adaptation}} to~$[m]_\ast = \{m\}$. If some~$\calR \in \Pi \setminus \Pi_m$ is not compatible with~$\calP$, then~$\calR$ has a side~$R \subseteq \overline{P_{i-1}} \cap P_i$ and either~$R = Q$ or one of the following holds:
    \begin{enumerate}[(i)]
        \item~$s = n$ and~$\{n, n^{-1}\} \subseteq R \cup M$; or
        \item~$s >_\circ n$ and~$R$ contains~$M \cup \{n, n^{-1}\}$.
    \end{enumerate}

    We can reject case (ii), since~$n$ is principal, so either~$\calR = \calQ$ or (i) holds. In the former case, we can replace~$\calQ$ by~$\calP$, and pass to considering another element of~$\Pi_{[m]} \setminus \Pi_m$.

    We are left in case (i). Choose such a partition~$\calR$ so that~$R$ is maximal with respect to inclusion. Define~$P' \coloneqq P_i \cup M \cup R$. Since~$M \subsetneq Q$ and~$\calR$ is compatible with~$\calQ$ (which, since~$s = n$, implies that~$R \cap Q = \emptyset$, as in the proof of \emph{Proposition \ref{prop:[MV19], Prop 4.22 adaptation}}), there must be another component of~$\Gamma^\pm \setminus \str{m}^\pm = \Gamma^\pm \setminus \str{n}^\pm$ contained in~$P_i$ but not in~$P'$. Hence~$P'$ defines a new symmetric~$\Gamma$-partition~$\calP'$, based at~$m$, distinct from any of the~$\calP_j$, and compatible with all of~$\Pi_m$.

    If~$\calP'$ is compatible with all of~$\Pi \setminus \calQ$, then it can replace~$\calQ$. Assume this is not the case; suppose that~$\calR' \in \Pi \setminus \Pi_m$ is incompatible with~$\calP'$. Let~$\calR'$ be based at~$s'$; since~$s'$ does not commute with~$m$, by \emph{Lemma \ref{lem:[MV19], Lem 4.12 and Lem 4.21}}~$\calR'$ has a side~$R'$ contained in~$P_{i-1}$,~$\overline{P_{i-1}} \cap P_i$, or~$\overline{P_i}$. In the first and third cases,~$\calR'$ is compatible with~$\calP'$, so we must have~$R' \subseteq \overline{P_{i-1}} \cap P_i$. We cannot have~$R' \subseteq P'$ or~$R' \subseteq \overline{P'}$, as otherwise~$\calR'$ would be compatible with~$\calP'$.

    Now,~$\calR$ is compatible with~$\calM$, and~$s'$ does not commute with~$n$, (as then~$s'$ would commute with~$m$), so one of the four quadrants~$R' \cap M$,~$R' \cap \overline{M}$,~$\overline{R'} \cap M$,~$\overline{R'} \cap \overline{M}$ is empty. The fourth of these is impossible since~$R, M \subseteq \overline{P_{i-1}} \cap P_i$, and the second of these implies~$R' \subseteq M \subseteq P'$, which is also a contradiction. Applying an identical argument to~$\calR$ instead of~$\calM$ leaves us with four possibilities:
    \begin{itemize}
        \item~$R' \cap M = R' \cap R = \emptyset$. This implies~$R' \cap P' = \emptyset$, so is impossible;
        \item~$R' \cap M = \overline{R'} \cap R = \emptyset$. This implies that~$R \subseteq R'$, so (w.l.o.g.)~$n \in R'$, so (since~$\calR'$ is symmetric and not based at~$n$)~$\{n, n^{-1}\} \in R'$, which contradicts~$R' \cap M = \emptyset$;
        \item~$\overline{R'} \cap M = R' \cap R = \emptyset$ is similarly impossible;
        \item~$\overline{R'} \cap M = \overline{R'} \cap R = \emptyset$ is the only remaining possibility. This implies that~$R \cup M \subseteq R'$.
    \end{itemize}
    Compatibility of~$\calR'$ with~$\calQ$ similarly implies that one of~$R' \cap Q$,~$R' \cap \overline{Q}$,~$\overline{R'} \cap Q$,~$\overline{R'} \cap \overline{Q}$ is empty. The first of these is impossible since~$M \subseteq Q \cap R$; the last is impossible since~$R, Q \subseteq \overline{P_{i-1}} \cap P_i$. We cannot have~$R' \cap \overline{Q} = \emptyset$, as then~$\{n, n^{-1}\} \subseteq R \cup M \subseteq R' \subseteq Q$, which is impossible since~$\calQ$ is based at~$n$. Hence~$\overline{R'} \cap Q = \emptyset$, so~$Q \subseteq R'$. 

    Suppose that~$Q \neq R'$. Since~$\calR'$ separates~$n$ and~$m$,~$n$ is not in the same component of~$\Gamma \setminus \str{s'}$ as~$m$, so there does not exist any~$x \in \link{n}$ with~$x \notin \link{s'}$. Hence~$s' \geq_\circ n$;~$n$ is principal so we must have~$s \in [n]_\ast = [m]_\ast$.

    We now repeat the entire proof, but with~$s'$ replacing~$n$ and~$\calR'$ replacing~$\calQ$. If case (i) holds once again, and we obtain some~$\calR'' \neq \calR'$, based at some~$s'' \in [m]_\ast$, then we will have~$R' \subsetneq R''$, so~$R'' \neq Q$. This ascending chain of sides must terminate at some point. We may therefore assume that in fact~$Q = R'$ above. Hence~$\calP'$ is compatible with all of~$\Pi \setminus \Pi_m$ except for~$\calQ$, and so can replace~$\calQ$. 

    We can now repeat this proof until all of~$\Pi_{[m]_\ast} \setminus \Pi_m$ has been replaced by symmetric partitions based at~$m$, which completes the argument.
\end{proof}

For any subset~$W \subseteq V$, let~$M^\Sigma(W)$ denote the largest possible size of a compatible set of symmetric~$\Gamma$-partitions, all of which are based at an element of~$W$.

We will need the following statement from \cite{MillardVogtmann2019}.

\begin{theorem}[\cite{MillardVogtmann2019}, Theorem 4.2]\label{thm:[MV19], Thm 4.2}
    Let~$\calP$ and~$\calQ$ be~$\Gamma$-partitions, based at~$m$ and~$n$ respectively, with~$[m, n] \neq 1$. The outer automorphisms~$\varphi(\calP, m)$ and~$\varphi(\calQ, n)$ commute if and only if~$\calP$ and~$\calQ$ are compatible,~$\calP$ does not split~$n$, and~$\calQ$ does not split~$m$.
\end{theorem}

\begin{proposition}\label{prop:symout contains free abelian subgp of rank MSigma(L)}
   ~$\symout$ contains a free abelian subgroup of rank~$M^\Sigma(L)$.
\end{proposition}

\begin{proof}
    We show that there exists a set of symmetric principal~$\Gamma$-partitions of maximal size (that is, of size~$M^\Sigma(L)$) with corresponding~$\Gamma$-Whitehead automorphisms which pairwise commute and are independent, so generate a subgroup of~$\symout$ isomorphic to~$\mathbb{Z}^{M^\Sigma(L)}$. To do this, we follow the outline of the proof of Theorem 4.25 in \cite{MillardVogtmann2019}.

    Let~$\Pi$ be a compatible set of symmetric principal~$\Gamma$-partitions of size~$M^\Sigma(L)$.
    
    By \emph{Lemma \ref{lem:[MV19], Cor 4.23 adaptation}}, we can assume that for any nonabelian equivalence class~$[m]$, all partitions in~$\Pi$ are based at a single~$m \in [m]$. Using~$m$ as the multiplier for each of these partitions, the associated symmetric outer~$\Gamma$-Whitehead automorphisms pairwise commute. 

    If~$\calP$,~$\calQ \in \Pi$ are based at~$m$ and~$n$ respectively with~$[m, n] = 1$, then~$\varphi(\calP, m)$ and~$\varphi(\calQ, n)$ commute. 

    If~$\calP$,~$\calQ \in \Pi$ are based at~$m$ and~$n$ respectively with~$[m, n] \neq 1$, then~$\calP$ does not split~$n$ and~$\calQ$ does not split~$m$, since~$\calP$ and~$\calQ$ are symmetric. Hence by (\cite{MillardVogtmann2019}, Theorem 4.2),~$\varphi(\calP, m)$ and~$\varphi(\calQ, n)$ commute.

    Hence we have a collection of infinite-order pairwise-commuting symmetric outer automorphisms of size~$M^\Sigma(L)$. We are now at a stage in the proof of (\cite{MillardVogtmann2019}, Theorem 4.25) after which can be replicated without edit; this section proves that the elements of this collection are pairwise independent. This gives the result.
\end{proof}

\begin{corollary}\label{cor:vcdsymout geq MSigma(L)}
   ~$\vcdsymout \geq M^\Sigma(L)$.
\end{corollary}

\subsection{Upper bound for~$\vcdsymout$}\label{subsec:upper bound for vcdsymout}

\begin{proposition}\label{prop:MSigma(L) = MSigma(V)}
   ~$M^\Sigma(L) = M^\Sigma(V)$.
\end{proposition}

\begin{proof}
    Our strategy is to start with a maximal compatible set~$\Pi$ of symmetric~$\Gamma$-partitions (that is, one of size~$M^\Sigma(V)$), and replace non-principal partitions with symmetric principal partitions, one-by-one.

    If there are no non-principal partitions in~$\Pi$, then we are done, so suppose this is not the case. By \emph{Lemma \ref{lem:[MV19], Cor 4.23 adaptation}} we may assume that for each principal nonabelian equivalence class and choice of representative~$m$, we have~$\Pi_{[m]} = \Pi_m$.

    Let~$u$ be non-principal with~$\Pi_u \neq \emptyset$. Choose a principal~$m >_\circ u$. Since~$d_\Gamma(u, m) = 2$, both~$m$ and~$m^{-1}$ are always on the same side of an element of~$\Pi_{[u]}$. Choose~$Q$ to be maximal among sides of elements of~$\Pi_{[u]}$ containing~$\{m, m^{-1}\}$. Without loss of generality, assume that~$\calQ$ is based at~$u$ and moreover that~$u \in Q$.

    Since~$u$ and~$m$ do not commute, any partition based at~$m$ has a side either containing or contained in~$Q$. Let~$M$ be the maximal side of an element of~$\Pi_m$ such that~$m^{-1} \in M \subseteq Q$. 

    Define~$P \coloneqq \left(Q \setminus M\right) \cup \{u^{-1}\}$. This defines a symmetric~$\Gamma$-partition~$\calP$ based at~$m$ which is incompatible with~$\calQ$, and so is not an element of~$\Pi$. We aim to show that it can replace~$\calQ$. Note that~$m \in P$.

    If there does not exist~$\calR \in \Pi \setminus \calQ$ which is incompatible with~$\calP$ then we can replace~$\calQ$ with~$\calP$, and pass to considering another non-principal partition in~$\Pi$. So suppose this is not the case; let~$\calR \in \Pi$ be incompatible with~$\calP$. Suppose that~$\calR$ is based at~$n$.

    \textbf{Case I}:~$n = u$.

    If~$\calR$ is based at~$u$, then it has one side~$R$ which either contains or is contained in~$Q$. By the maximality assumption in the choice of~$\calQ$, this side does not contain~$\{m, m^{-1}\}$, so we must have~$R \subseteq Q \setminus M$. Hence~$R \subseteq P$, and so~$\calR$ is compatible with~$\calP$, which is a contradiction. 

    \textbf{Case II}:~$n \neq u$.
    
    Since~$\calR$ is incompatible with~$\calP$, we have~$d_\Gamma(n, m) \in \{0, 2\}$. We cannot have~$d_\Gamma(n, u) = 1$, as~$m >_\circ u$, so that would mean~$d_\Gamma(n, m) = 1$. Hence~$\calR$ is not adjacent to~$\calQ$, so there exists a choice of sides~$R^\times$,~$Q^\times$ such that~$R^\times \cap Q^\times = \emptyset$.

    Suppose that~$Q^\times = Q$. This means that~$u \notin R^\times$. We have dispensed with the case~$n = u$, so since~$\calR$ is symmetric, we also have~$u^{-1} \notin R^\times$. Hence~$R^\times \cap P = \emptyset$, which is a contradiction to the incompatibility of~$\calR$ with~$\calP$. 

    Hence~$Q^\times = \overline{Q}$, which means that~$R^\times \subseteq Q$, which in turn means (since~$\calQ$ is symmetric and not based at~$n$) that~$\{n, n^{-1}\} \subseteq Q$.

    \textbf{Case IIa}:~$m = n$.

    If~$R^\times \ni m$, then~$P \supseteq R^\times$, and we have a contradiction as~$\calP$ and~$\calR$ are incompatible. Hence~$R^\times$ contains~$m^{-1}$. By the maximality assumption in the choice of~$M$, we must have~$R^\times \subseteq M$, whence~$R^\times \cap P = \emptyset$, which is another contradiction. 

    \textbf{Case IIb}:~$m \neq n$.
    
    Since~$d_\Gamma(m, n) \neq 1$, if~$\{m, m^{-1}\} \subsetneq R^\times$, then~$\left(M \cup \{m\}\right) \cap R^\times = \emptyset$, in which case~$R^\times \subseteq P$ and we contradict incompatibility of~$\calR$ and~$\calP$. Therefore~$\left(M \cup \{m\}\right) \subseteq R^\times$. Since~$Q$ was chosen to be maximal among sides of elements of~$\Pi_{[u]}$ containing~$\{m, m^{-1}\}$, we must have~$n \notin [u]$. Hence either~$n >_\circ u$ or there exists~$x \in \link{u} \subseteq \str{m}$ such that~$x \notin \link{n}$. Then~$d_\Gamma(x, u) = d_\Gamma(x, m) = 1$, which implies that~$u$ and~$m$ are in the same component of~$\Gamma^\pm \setminus \str{n}^\pm$. But this is impossible, since~$\calR$ separates~$u$ and~$m$.

    Hence~$n >_\circ u$. In this case, repeat the argument with~$n$ replacing~$m$, and with some~$N \supseteq R^\times$ based at~$n$ replacing~$M$. If we reach the same point in this second iteration, finding some~$n' >_\circ u$ and~$\calR'$ based at~$n'$ playing the same role as~$\calR$ here, then we cannot have~$n' = m$. Indeed, if~$n' = m$ then~$R \cap R' \ni n$,~$\overline{R} \cap R' \ni n^{-1}$,~$R \cap \overline{R'} \ni m$ or~$m^{-1}$, and~$\overline{R} \cap \overline{R'} \ni u$, which implies that~$\calR'$ is incompatible with~$\calR$ (as~$d_\Gamma(n, m) \neq 1)$; this is a contradiction. Hence these iterations must terminate at some point, at which we have exhausted all ways of finding an element of~$\Pi \setminus \calQ$ which is incompatible with~$\calP$. Thus, we may replace~$\calQ$ in~$\Pi$ with~$\calP$.  

    We can now repeat this argument until every non-principal element of~$\Pi$ has been replaced by a symmetric principal partition, at which point we're done.
\end{proof}

\begin{theorem}\label{thm:vcdsymout = dim(symspine)}
   ~$\vcdsymout = \dim(\symspine)$.
\end{theorem}

\begin{proof}
    By \emph{Corollary \ref{cor:vcdsymout geq MSigma(L)}}, we have~$M^\Sigma(L) \leq \vcdsymout$. Combining this with \emph{Proposition \ref{prop:MSigma(L) = MSigma(V)}} and the observation that~$\vcdsymout \leq M^\Sigma(V) = \dim(\symspine)$ gives the result.
\end{proof}

\footnotesize
\bibliographystyle{alpha}
\bibliography{Bibliography.bib}\medskip

\noindent Gabriel Corrigan \\
University of Glasgow \\
\href{mailto:g.corrigan.1@research.gla.ac.uk}
{g.corrigan.1@research.gla.ac.uk}

\end{document}